\documentclass{article} 
\usepackage{amssymb}
\usepackage{amsmath}
\usepackage{algorithm}
\usepackage{algpseudocode}
\usepackage[framemethod=tikz]{mdframed}

\usepackage[colorlinks=true]{hyperref} 
\AtBeginDocument{%
\hypersetup{
    linkcolor=blue,    
    filecolor=magenta,      
    urlcolor=blue,     
    citecolor=blue,
}
}

\usepackage{amsmath,amsfonts,units,amssymb,amsthm,mathtools,bigstrut,amstext}

\usetikzlibrary{shadows}

\newmdenv[shadow=true,shadowcolor=lightblue,font=\sffamily,rightmargin=6pt]{shadedbox}
\usepackage{xspace}
\usepackage{newtxtext} 
\usepackage{newtxmath} 
\DeclareSymbolFont{CMlargesymbols}{OMX}{cmex}{m}{n} 
\DeclareMathDelimiter{(}{\mathopen} {operators}{"28}{CMlargesymbols}{"00}
\DeclareMathDelimiter{)}{\mathclose}{operators}{"29}{CMlargesymbols}{"01}
\DeclareMathAlphabet\mathcal{OMS}{cmsy}{m}{n} 
\SetMathAlphabet\mathcal{bold}{OMS}{cmsy}{b}{n} 
\usepackage[twoside,
	paperheight=247mm,
	paperwidth=174mm,
	marginratio={3:5,2:3},
	left=18mm,
	top=20mm]{geometry}
\usepackage{amssymb,amsmath}
\numberwithin{figure}{section}
\numberwithin{table}{section}

\usepackage{graphics,graphicx}
\graphicspath{{./FIGS/}}

\usepackage[hang]{subfigure}
\usepackage[UKenglish]{babel}
\usepackage{enumerate} 
\newcommand{\ignore}[1]{}
\allowdisplaybreaks

\newcommand{\eq}[1]{\begin{equation}\label{#1}}
\newcommand{\en}{\end{equation}}

\def\RR{\mathbb{R}}
\def\CC{\mathbb{C}} 
\def\PP{\mathcal{P}}
\def\XX{\mathcal{X}}
\def\calL{\mathcal{L}}
\def\FF{\mathcal{F}}
\def\calB{\mathcal{B}} 
\newcommand{\K}{\mathcal{K}}
\newcommand{\epsalg}{$\epsilon$-algorithm\xspace}

\def\eps{\epsilon}
\def\emphc#1{\textcolor{red}{#1}}
\def\emphk#1{\textcolor{black}{\textsl{#1}}}
\def\gbox#1{\colorbox{lightblue}{\textcolor{red}{\textsl{ #1 }}\,}} 
\def\lt{{\tt <}}%
\def\Span{\mbox{Span}}

\def\nref#1{(\ref{#1})}
\def\inv{^{-1}}
\def\up#1{^{(#1)}}%

\definecolor{lightblue}{rgb}{0.72,0.72,1.0}
\def\dwn#1{_{(#1)}}%
\def\half{\frac{1}{2}}
\def\grad{\nabla} 

\def\betab{\vspace{-5pt}\begin{tabbing}
xx\=xx\=xx\=xx\=xx\=xx\=xx\=xxx\= \kill} 
\def\entab{\end{tabbing}\vspace{-3pt}}

\newtheorem{corollary}{Corollary}
\newtheorem{proposition}{Proposition}[section]
\newtheorem{theorem}{Theorem}[section]
\newtheorem{lemma}{Lemma}[section]

\title{Acceleration methods for fixed point iterations}

\author{Yousef Saad
  \thanks{
  Department of Compter Science and Eng., University of Minnesota,
  Twin cities, MN 55455, USA. 
E-mail: {saad@umn.edu}} }

\begin{document}

\label{firstpage}
\maketitle

\begin{abstract}

  A pervasive approach in scientific computing is to express the solution to a given
  problem as the limit of a sequence of vectors or other mathematical objects.
  In many situations these sequences are generated by
  slowly converging iterative procedures and this led practitioners to seek
  faster alternatives to reach the limit.  ``Acceleration
  techniques'' comprise a broad array of methods specifically designed with this
  goal in mind.  They started as a means of improving the convergence of general
  scalar sequences by various forms of ``extrapolation to the limit'', i.e., by
  extrapolating the most recent iterates to the limit via linear
  combinations. Extrapolation methods of this type, the best known example of
  which is Aitken's Delta-squared process, require only the sequence of vectors
  as input.

  However, limiting methods to only use the iterates is too
  restrictive. Accelerating sequences generated by fixed-point iterations by
  utilizing both the iterates and the fixed-point mapping itself has proven
  highly successful across various areas of physics. A notable example of these
  Fixed-Point accelerators (FP-Accelerators) is a method developed by
  Donald Anderson in 1965 and now widely known as Anderson Acceleration (AA). Furthermore,
  Quasi-Newton and Inexact Newton methods can also be placed in this category
  since they can be invoked to find limits of fixed-point iteration
  sequences by employing the exact same ingredients as those of the
  FP-accelerators.
      
  This paper presents an overview of these methods -- with an emphasis on those,
  such as AA, that are geared toward accelerating fixed-point iterations.  We
  will navigate through existing variants of accelerators, their implementations,
  and their applications,   to unravel the close connections between them.
  These   connections were often not recognized by the
  originators of certain methods, who sometimes stumbled on slight variations of
  already established ideas.  Furthermore, even though new accelerators were
  invented in different corners of science, the underlying principles behind them are
  strinkingly similar or identical.
 
  The plan of this article will approximately follow the historical trajectory
  of extrapolation and acceleration methods, beginning with a brief description of
  extrapolation ideas, followed by the special case of linear systems, the
  application to self-consistent field (SCF) iterations, and a detailed view of
  Anderson Acceleration.  The last part of the paper is concerned with more recent
  developments, including theoretical aspects, and a few thoughts on accelerating
  Machine Learning algorithms. 
\end{abstract}



\tableofcontents 

\section{Historical perspective and overview}
Early iterative methods for solving systems of equations, whether linear or
nonlinear, often relied on simple \emph{fixed-point iterations} of the form
\eq{eq:FixedPt} x_{j+1} = g(x_j) \en which, under certain conditions, converge to a
\emph{fixed-point} $x_*$ of $g$, i.e., a point such that $g(x_*) = x_*$.  Here,
$g$ is some mapping from $\RR^n$ to itself which we assume to be at least
continuous.  Thus, the iterative method for solving linear systems originally
developed by Gauss in 1823 and commonly known today as the Gauss-Seidel
iteration, see, e.g., \cite{Saad-book2},
can be recast in this form.  The fixed-point iterative approach can
be trivially adopted for solving a system of equations of the form: \eq{eq:fx}
f(x) = 0, \en where $f$ is again a mapping from $\RR^n$ to itself.  This can be
achieved by selecting a non-zero scalar $\gamma$ and defining the mapping
$g(x) \equiv x + \gamma f(x)$, whose fixed-points are identical with the zeros
of $f$.  The process would generate the iterates:
\eq{eq:FixedPt1} x_{j+1} = x_j + \gamma f(x_j) , \quad j=0,1, \cdots \en
starting from some initial guess
$x_0$.  Approaches that utilize the above framework are common in optimization
where $f(x)$ is the negative of the gradient of a certain objective function $\phi(x)$ whose
minimum is sought.  The simplicity and versatility of the fixed-point iteration
method for solving nonlinear equations are among the reasons it
has been successfully used across various fields. 

Sequences of vectors, whether of the type \eqref{eq:FixedPt} introduced above, or
generated by some other `black-box' process, often converge slowly or may even
fail to converge.  As a result practitioners have long sought to build sequences
that converge faster to the same limit as the original sequence or even to
establish convergence in case the original sequence fails to converge.
We need to emphasize a key distinction between two different strategies that
have been adopted in this context.  In the first, one is just given all, or a
few of the recent members of the sequence and no other information and the goal
is to produce another sequence from it, that will hopefully converge faster.
These procedures typically rely on forming a linear combination of the current
iterate with previous ones in an effort to essentially extrapolate to the limit and
for this reason they are often called \emph{`extrapolation methods'.}  Starting
from a sequence $\{ x_j \}_{j=0,1,...}$ a typical extrapolation technique builds the new,
extrapolated, sequence as follows:
\eq{eq:extrap} y_j = \sum_{i=[j-m]}^j
\gamma_i\up{j} x_i \en
where we recall the notation $[j-m] = \max \{j-m,0 \}$ and $m$ is a parameter,
known as the `depth' or `window-size', and the $\gamma_i\up{j}$'s are
`acceleration parameters'. If the new sequence is to converge to the same
limit as the original one, the condition
\eq{eq:cond1} \sum_{i=[j-m]}^j
\gamma_i\up{j} = 1 \en must be imposed. Because the process works by forming linear
combinations of iterates of the original sequences, it is also often
termed a \emph{`linear acceleration'} procedure, see, e.g.,
\cite{Brezinski-RZ-book,BrezinskiReview00,Forsythe53} but the term
`extrapolation' is more common.  Extrapolation methods are
 discussed in Section~\ref{sec:extrapol}.

In the second strategy, we are again given all,  or a few,  of
the recent iterates, but now we also have access to the fixed-point mapping $g$
to help build the  next member of the sequence. A typical example of these methods
is the Anderson/Pulay acceleration which is discussed in detail in
Sections~\ref{sec:AA1} and~\ref{sec:AA2}.
When presented from the equivalent form of Pulay mixing, this method builds a new iterate  as follows
\eq{eq:pulay0} 
x_{j+1} =  \sum_{i=[j-m]}^j \theta_i\up{j}  g(x_{i}) ; 
\en
where the $\theta_i\up{j} $'s satisfy the constraint $ \sum_{i=[j-m]}^j \theta_i\up{j} = 1$.
As can be seen, the function $g$ is now invoked when building a new iterate.
Of course having access to $g$ may allow one to develop more powerful methods than
if we were to only use the iterates as is done in extrapolation methods.
It is also clear that there are  situations when this is not practically feasible.

The emphasis of this paper is on 
this second class which we will refer to
as \emph{Fixed-Point acceleration methods}, or \emph{FP-Acceleration methods}.
Due to their broad applicability, FP-acceleration  emerged from
different corners of science.
Ideas related to extrapolation can be traced as   far back as the 17th century,
see  \cite[sec. 6]{BrezinskiGenesis2019} for references. 
However, one can say that modern extrapolation and acceleration ideas took off
with the work of \cite{Richardson1910} and \cite{Aitken},
among others, in the early part of 20th century and then gained importance
toward the mid 1950s, with the advent of electronic computers, e.g.,
\cite{Romberg1955,Shanks55}.

\begin{figure}[h]
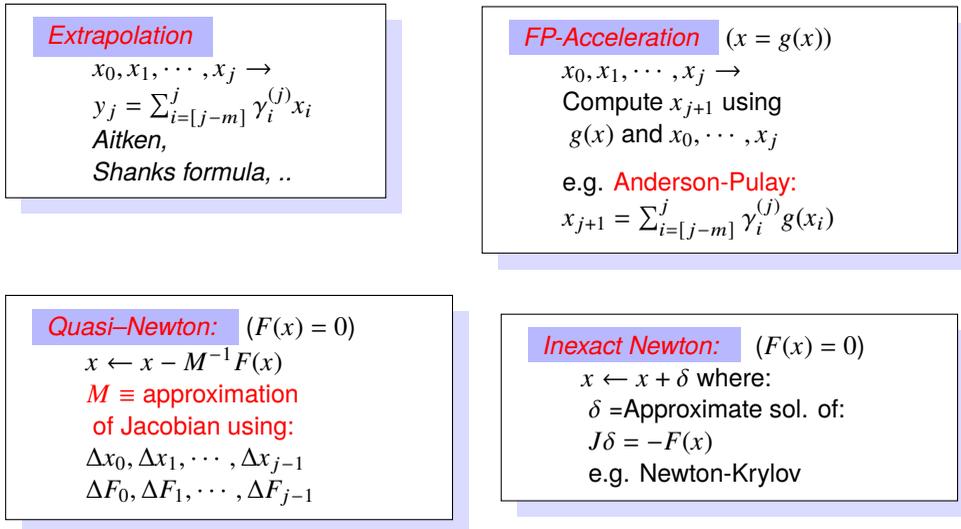

\noindent
\begin{minipage}{0.4\textwidth}
    \begin{shadedbox}      
    {\gbox{Extrapolation}}

\rule{20pt}{0pt} $x_0, x_1, \cdots, x_j  \rightarrow$

\rule{20pt}{0pt}  $y_j = \sum_{i=[j-m]}^j \gamma_i\up{j} x_i$ 

\rule{20pt}{0pt}  \emphk{Aitken, \\
\rule{20pt}{0pt}  Shanks  formula, ..}
\end{shadedbox}
 \end{minipage}
\hfill 
\begin{minipage}{0.5\textwidth}
    \vspace{.3in} 
  \begin{shadedbox}
  \gbox{FP-Acceleration} {$(x = g(x))$} \\
\rule{20pt}{0pt}$x_0, x_1, \cdots, x_j  \rightarrow$

\rule{20pt}{0pt}Compute  $x_{j+1}$ using \\
\rule{20pt}{0pt} $g(x)$  and $x_0,\cdots, x_j$

\medskip
\rule{20pt}{0pt}e.g. \emphc{ Anderson-Pulay:} 

\rule{20pt}{0pt}$x_{j+1} = \sum_{i=[j-m]}^j \gamma_i\up{j} g(x_i) $
       \end{shadedbox}
  \end{minipage}

  \vspace{0.2in}
  \noindent
  \begin{minipage}{0.47\textwidth}
    \begin{shadedbox}
\gbox{Quasi--Newton:} {($F(x) = 0$) }

\rule{20pt}{0pt}$x \gets x- M^{-1} F(x)$ 

\rule{20pt}{0pt}\emphc{  $M  \equiv$ approximation \\
  \rule{20pt}{0pt} of Jacobian using:}

\rule{20pt}{0pt}$\Delta x_0, \Delta x_1,\cdots, \Delta x_{j-1}$\\
\rule{20pt}{0pt}$\Delta F_0, \Delta F_1,\cdots, \Delta F_{j-1}$
\end{shadedbox} 
\end{minipage}\hfill \hfill 
  \begin{minipage}{0.48\textwidth}
      \begin{shadedbox}
        \gbox{Inexact Newton:}
        { ($F(x)=0$)}

        \rule{20pt}{0pt}$x \gets x + \delta $ where: 
        
        \rule{20pt}{0pt} $\delta$ =Approximate sol. of: 

        \rule{20pt}{0pt} $J \delta = - F(x)$ 

        \rule{20pt}{0pt} e.g. Newton-Krylov
    \end{shadedbox}
  \end{minipage}
  \caption{Four distinct classes of acceleration methods}\label{fig:groups}
\end{figure}



Because acceleration and extrapolation methods are often invoked for solving
nonlinear equations, they naturally compete with ``second-order type methods''
such as those based on Quasi-Newton approaches.  Standard Newton-type techniques
cannot be placed in the same category of methods as those described above
because they exploit a local second order model and invoke the differential of
$f$ explicitly.  In many practical problems, obtaining the Jacobian of $f$ is
not feasible, or it may be too expensive.  However, inexact Newton and
quasi-Newton methods,  are two alternatives that bypass the need to compute the
whole Jacobian $J$ of $f$.
Therefore, they satisfy our requirements for what can be viewed as a method for
accelerating iterations of the form \nref{eq:FixedPt1}, in that they only invoke
previous iterates and our ability to apply the 
mapping $g$ (or $f$) to any given vector.
There should therefore be no
surprise that one can find many interesting connections between this class
methods and some of the FP-acceleration techniques.  A few of the methods
developed in the quasi-Newton context bear strong similarities, and are in some
cases even mathematically equivalent, to techniques from the FP acceleration class.

The four classes of methods discussed above are illustrated in
Figure~\ref{fig:groups}.
This article aims to provide a coverage of these four classes of acceleration
and interpolation methods, with an emphasis on those geared toward accelerating
fixed-point iterations, i.e., those in the top-right corner of
Figure~\ref{fig:groups}.  There is a vast amount of literature on these methods
and it would be challenging and  unrealistic to try to be
exhaustive. However, one of our specific goals is to unravel the various
connections between the different methods.  Another is to present, whenever
possible,  interesting variants of these methods and details on their
implementations.  


\subsubsection*{Notation}

\begin{itemize}
  \item 
Throughout the article $g$ will denote a mapping from $\RR^n $ to itself, and
one is interested in a fixed-point $x_*$
of $g$. Similarly, $f$ will denote also a
mapping from $\RR^n $ to itself, and one is interested in a zero of $f$.
\item  $\{ x_j \}_{j=0,1,\cdots,}$ denotes a sequence in $\RR^n$.

  \item
Given a sequence $\{x_j\}_{j = 0,1,\cdots,}$ the forward difference operator $\Delta$
builds a new sequence whose terms are defined by:
\eq{eq:Deltaxj}
\Delta x_j = x_{j+1} - x_j, j=0,1, \cdots,
\en

\item We will often refer to an evolving set of columns
  where the most recent $m$ vectors from a sequence are retained.
  In order to cope with the different indexing used in the algorithms,
  we found it convenient to define for any $k\in \mathbb{Z}$ 
  \eq{eq:mj}
  [k] = \max \{k,0 \}
  \en
Thus, we will often see matrices of the form
$X_{j} = [x_{[j-m+1]}, x_{[j-m+1]+1},\cdots, x_{j}]$ that have 
$j+1$ columns when $j<m$ and $m$ columns otherwise. 

\item Throughout the paper $\| . \|_2$ will denote the Euclidean norm
  and $\| . \|_F$ is the Frobenius norm of matrices.

\end{itemize}

\section{Extrapolation methods for general sequences}\label{sec:extrapol}
Given a sequence $\{ x_j \}_{j=0,1,\cdots,}$,  an extrapolation method builds
an auxiliary sequence $\{ y_j \}_{j=0,1,\cdots,}$, where $y_j$ is typically a linear
combination of the most recent iterates as in 
\nref{eq:extrap}.  The goal is to produce a sequence that converges faster to
the limit of $x_j$. Here $m$ is a parameter 
 known as the `depth' or `window-size', and the
$\gamma_i\up{j}$'s are `acceleration parameters', which sum-up to one, see
\eqref{eq:cond1}. 
Aitken's $\delta^2$ process \cite{Aitken} is 
an early instance of such a procedure that had  a major impact.
 
\subsection{Aitken's procedure}
Suppose we have a scalar sequence $\{ x_i \}$ for $i=0,1,\cdots,$ that converges
to a limit $x_*$. 
Aitken  assumed that in this situation the sequence roughly satisfies the relation
\eq{eq:Aitken0}
x_{j+1} - x_*= \lambda (x_{j} - x_*) \quad {\forall j}
\en
where $\lambda$ is some unknown scalar.
This is simply an expression of a geometric convergence to the limit.
The above condition defines a set of sequences and the condition is termed
a `kernel'. In this particular case  \nref{eq:Aitken0} is the Aitken kernel.

The scalar     $\lambda, $ and the limit $x_*$  can be trivially determined from
three consecutive iterates $x_{j}, x_{j+1}, x_{j+2}$   by writing:
\eq{eq:ansatz}
\frac{x_{j+1} - x_*}{x_{j} - x_*} = \lambda, \quad
\frac{x_{j+2} - x_*}{x_{j+1} - x_*} = \lambda 
\en
from which it follows by eliminating $\lambda$ from the two equations that 
\begin{equation}  \label{eq:Aitken1}
x_*  = \frac{x_j x_{j+2} - x_{j+1}^2}
{x_{j+2} - 2 x_{j+1} +x_j} 
= x_j - \frac{(\Delta x_j)^2}{\Delta^2 x_j}  .
\end{equation}

Here $\Delta $ is the forward difference operator defined earlier in \nref{eq:Deltaxj}
and  $\Delta^2 x_j = \Delta (\Delta x_j)$.
As can be expected, the set of sequences that satisfy
Aitken's kernel, i.e., \nref{eq:Aitken0} is very narrow. In fact, subtracting
the same relation obtained by replacing $j$ by
$j-1$ from  relation \nref{eq:Aitken0}, we would obtain $ x_{j+1} - x_{j} = \lambda (x_{j} - x_{j-1})$,
hence $x_{j+1} - x_{j} = \lambda^{j} (x_1 - x_0) $. Therefore, scalar
sequences that satisfy Aitken's kernel exactly are of the form
\eq{eq:AitkenEx}
x_{j+1} = x_0 + (x_1 - x_0) \sum_{k=0}^j \lambda^k \quad \text{for} \quad j \ge
0.
\en


Although a given sequence is  unlikely to satisfy Aitken's kernel exactly, 
it may nearly satisfy it when approaching the limit and in this case,
the extrapolated value \nref{eq:Aitken1} will provide a way to build an `extrapolated' sequence defined as follows: 
\eq{eq:Aitken2} 
y_j  = x_j - \frac{(\Delta x_j)^2}{\Delta^2 x_j} \quad j=0,1, \cdots, 
\en
Note that to compute $y_j $ we must have available three consecutive iterates,
namely, $x_{j},x_{j+1},x_{j+2}$, so $y_j$  starts with a delay relative to the
original sequence.

A related approach known as Steffensen's method is geared toward solving the  equation
$f(x)=0$ by the Newton-like iteration:
\eq{eq:steff}
x_{j+1} = x_j - \frac{f(x_j)}{d(x_j)} \quad \mbox{where}\quad d(x) = \frac{f(x+ f(x))-f(x)}{f(x)} . 
\en
One can recognize in $d(x)$ an approximation of the derivative of $f$ at
$x$. This scheme converges quadratically under some smoothness assumptions for
$f$. In addition, it can be easily verified, that when the sequence $\{x_j\}$ is
produced from the fixed-point iteration $x_{j+1}=g(x_j)$, one can recover
Aitken's iteration for this sequence by applying Steffensen's scheme to the
function $f(x) = g(x)-x$. Therefore, Aitken's method also converges
quadratically for such sequences when $g$ is smooth enough.

 \subsection{Generalization: Shanks transform and the $\epsilon$-algorithm}
 Building on the  success of Aitken's acceleration procedure, 
 \cite{Shanks55} explored ways to generalize it by replacing
 the kernel~\ref{eq:Aitken0} with  a kernel of the form:


\eq{eq:shanksKernel} 
a_0 (x_j - x_*) + a_1 (x_{j+1} - x_*) + \cdots + a_m (x_{j+m} - x_*)  = 0 .
\en
where the scalars $a_0, \cdots, a_{m}$ and the limit $x_*$ are unknowns. The
sum of the scalars $a_i$  cannot be equal to zero and there is no loss of generality
in assuming that they add up to 1, i.e.,
\eq{eq:shanks2}
a_0 + a_1 + \cdots + a_m  = 1.
\en
In addition, it is commonly assumed that $a_0 a_m \ne 0$, so that exactly $m+1$
terms are involved at any given step. 
We can set-up a linear system to compute $x_*$ by putting equation
\nref{eq:shanksKernel} and with \nref{eq:shanks2} together into the
 $(m+2) \times (m+2) $ linear system: 
\begin{equation}
  \left\{
\begin{array}{lllllllllllll} 
  a_0         &+& a_1           &+& \cdots      &+& a_m            & &      &=& 1 & \\
  a_0 x_{j+i} &+& a_1 x_{j+i+1} &+& \cdots      &+& a_m x_{j+i+m}  &-& x_*  &=& 0 &
   \quad  i=0,\cdots,m  . 
\end{array} \right.
\end{equation} 
Cramer's rule can now be invoked to solve this system and derive a formula for  $x_*$.
With a few row manipulations, we end up with the following formula known as 
\emph{
 the Shanks (or 'Schmidt-Shanks') Transformation for scalar sequences}:
\eq{eq:shanksFormula}
y_j\up{m}  =  
\frac{
\left| \begin{array}{cccc} 
x_j           &      x_{j+1}  & \cdots & x_{j+m} \cr
 \Delta x_j & \Delta x_{j+1}  & \cdots & \Delta x_{j+m} \cr 
\vdots        &  \vdots       &  \vdots  &  \vdots     \cr
\Delta x_{j+m-1}  & \Delta x_{j+m} & \cdots & \Delta x_{j+2m-1}
\end{array} 
\right|} 
{
\left| \begin{array}{cccc} 
 1           &    1      & \cdots & 1 \cr
 \Delta x_j & \Delta x_{j+1}  & \cdots & \Delta x_{j+m} \cr 
\vdots        &  \vdots       &  \vdots  &  \vdots     \cr
\Delta x_{j+m-1}  & \Delta x_{j+m} & \cdots & \Delta x_{j+2m-1}
\end{array} 
\right|} 
\en

A more elegant way to derive the above formula, one that will lead to
useful extensions, is to exploit the  following relation which follows from
the kernel~\nref{eq:shanksKernel}
\eq{eq:shanksKernel2}
a_0 \Delta x_j + a_1 \Delta x_{j+1} + \cdots + a_m \Delta x_{j+m}   = 0 .
\en
With this, we will build  a new system, now 
for the unknowns $a_0, a_1, \cdots, a_m$, as follows:

\begin{equation}
  \left\{
\begin{array}{lllllllllllll} 
  a_0         &+& a_1           &+& \cdots      &+& a_m            & =& 1 & \\
  a_0 \Delta x_{j+i} &+& a_1 \Delta  x_{j+i+1} &+& \cdots      &+& a_m
  \Delta x_{j+i+m}  & =& 0 & ; \ 
  i=0,\cdots,m-1 .                               
\end{array} \right. 
\label{eq:Kernel3} 
\end{equation} 
The right-hand side of this $(m+1)\times (m+1)$ system is the vector
$e_1 = [1, 0, \cdots, 0]^T \ \in \ \RR^{m+1}$.  Using Kramer's rule will yield
an expression for $a_k $ as the fraction of 2 determinants. The denominator of
this fraction is the same as that of \nref{eq:shanksFormula}. The numerator is
$(-1)^k$ times the determinant of that same denominator where the first row
and the $(k+1)$-st column are deleted.  Substituting these formulas for $a_k$ into
the sum $a_0 x_j + a_1 x_{j+1} + \cdots + a_m x_{j+m}$ will result in an
expression that amounts to expanding the determinant in the numerator of
\nref{eq:shanksFormula} with respect to its first row.
By setting $a_i \equiv \gamma_i\up{m}$ we can therefore rewrite  Shank's formula
in the form: 
\eq{eq:extrapolGen}
y_j\up{m} = \gamma_0\up{m} x_j + \gamma_1\up{m} x_{j+1} + \cdots + \gamma_m\up{m} x_{j+m}
\en
where each $\gamma_j\up{m}$ is  the ratio of two determinants, as was just explained.

It can be immediately seen that the particular case $m=1$ yields exactly
 Aitken's $\delta^2$ formula: 
\[ 
y_j^{(2)} =  
\frac{
\left| \begin{array}{cc} 
x_j           &      x_{j+1} \cr
 \Delta x_j & \Delta x_{j+1} 
\end{array} 
\right|} 
{
\left| \begin{array}{cccc} 
 1           &    1    \cr 
 \Delta x_j & \Delta x_{j+1}  
\end{array} 
\right|}  =
\frac{x_j \Delta x_{j+1} - x_{j+1} \Delta x_j}{ \Delta^2 x_j } 
= x_j - \frac{(\Delta x_j)^2}{ \Delta^2 x_j } . 
\]
As written, the above expressions are meant for scalar sequences,
but they can be generalized to vectors and this will be discussed later.

The generalization discovered by Shanks relied on ratios of determinants of size
of order $m$ where $m$ is the depths of the recurrence defined above.  The
non-practical character of this technique prompted Peter Wynn in a 1956 article
to explore an alternative implementation and this resulted in an amazingly
simple formula which he dubbed the `$\epsilon$-algorithm' \cite{Wynn}.  This
remarkable discovery spurred a huge following among the numerical linear algebra
community
\cite{CabayJackson,EddyWang,Brezinski,rre,rrem,sidi-al-86,Wynn62,Brezinski-TEA,KanielStein,Brezinski-mmpe,Jbilou-Sbeta,JbilouSadok91,bgb}
among many others.  The article~\cite{BrezinskiReview00} gives a rather
exhaustive review of these techniques up to the year 2000. The more recent
article \cite{BrezinskiGenesis2019}  surveys these methods while
also providing a wealth of information on the history of their development.

So what is Wynn's procedure to compute  the  $y_j\up{m}$?
The formula given above leads to expressions with 
determinants that involve Hankel matrices from which some  recurrences can be obtained
but these are not only complicated but also  numerically unreliable.

Wynn's  $\epsilon$-algorithm  is a recurrence relation to compute $y_j\up{m}$ 
shown in Equation~\nref{eq:epsalg}. It defines 
 sequences $\{\eps_j\up{m}\}_{j=0,1,\cdots,}$ each
 indexed by an integer \footnote{
   Note that this article adopts a different notation from
   the common usage in the literature: 
  the index of the sequence is a subscript rather than a superscript.
  In Wynn's notation $\{\eps_m\up{j}\}_{j=0,1,\cdots}$ denotes the $m$-th extrapolated sequence.}
$m$ starting with $m=-1$.

 The sequence $\{ \eps_{j}\up{-1}\}$, is just the zero sequence
$ \eps_{j}\up{-1} = 0 , \forall \ j$, while the sequence $\{ \eps_j\up{0}\}$, is
the original sequence $\eps_j\up{0} = x_j, \forall j$.
The other sequences, i.e.,  $\{\eps_j\up{m}\} $ are then obtained
recursively from as follows: 
\eq{eq:epsalg}
\epsilon_{j}\up{-1} = 0; \quad \epsilon_{j}\up{0} = x_j; \quad
\epsilon_{j}\up{m+1} = \epsilon_{j+1}\up{m-1} + \frac{1}{ \epsilon_{j+1}\up{m} -
  \epsilon_{j}\up{m} } \quad \mbox{for} \ m, j \ge 0 .
\en Each sequence
$\{\eps_j\up{m}\}$, where $m$ is constant,  can be placed on a vertical line on a grid as shown in
Figure~\ref{fig:rhombus}. With this representation, the $j$-th iterate for the
$(m+1)$-st sequence can be obtained using a simple rhombus rule from two members
of the $m$-th sequence and one member of the $(m-1)$-st sequence as is shown in
the figure.  Only the even-numbered sequences are accelerations of the original
sequence. The odd-numbered ones are just intermediate auxiliary sequences.
For example, the sequence $\eps_2\up{m}$ is identical with  the sequence obtained from Aitken's
$\delta^2$ process but the sequence $\eps_1\up{m}$ is auxiliary.
 
\begin{figure}[h]
\centerline{
\includegraphics[width=0.4\textwidth]{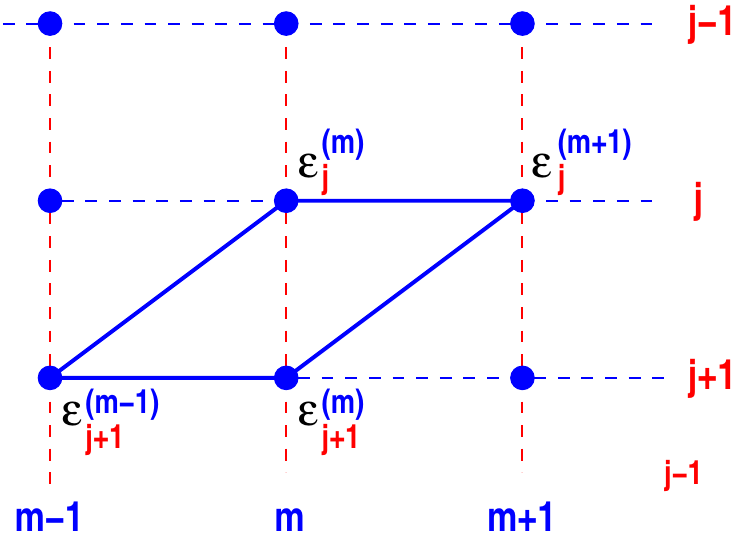}}
\caption{Wynn's rhombus rule for the $\epsilon$-algorithm}\label{fig:rhombus}
\end{figure} 



\begin{figure}[h]
\centerline{
\includegraphics[width=0.8\textwidth]{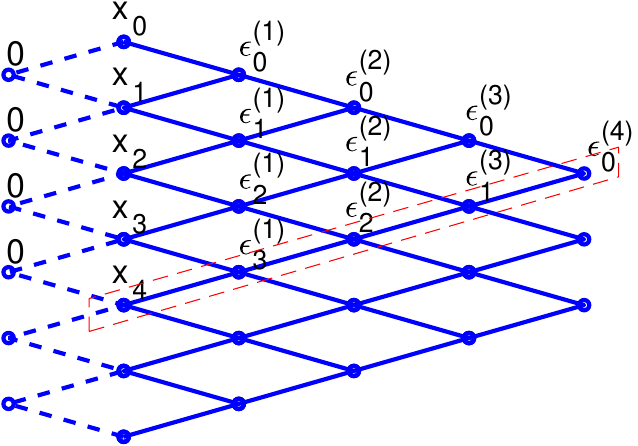} } 
\caption{Computational diagram of the $\epsilon$ algorithm.}\label{fig:epsDiagrm}
\end{figure}

If we denote by $\eta_m$  the transformation that maps a sequence $\{x_j\}$
into $\{t_j\up{m}\}$, i.e., we have $\eta_m(x_j) = t_j\up{m}$, then it can be shown that 
$\eps_j\up{2m} = \eta_m (x_j) $ and
$\eps_j\up{2m+1} = 1/  \eta_m (\Delta x_j ) $.

From a computational point of view, we would proceed differently in order not to
have to store all the sequences. The process, shown in
Figure~\ref{fig:epsDiagrm}, works on diagonals of the table.  When the original
sequence is computed we usually generate $x_0, x_1, \cdots, x_j, ..$ in this order.
Assuming that we need to compute up to the $ m$-th sequence $\{\eps_j\up{m}\}$
then once $\eps_j\up{0}=x_j$ becomes available we can obtain the entries
$\eps_{j-1}\up{1}, \eps_{j-2}\up{2}, \cdots, \eps_{j-m}\up{m}$.  This will
require the entries
$\eps_{j-1}\up{0}, \eps_{j-2}\up{1}, \cdots, \eps_{j-m}\up{m-1}$ and so we will
only need to keep the previous diagonal.

Wynn's result on the equivalence between the \epsalg and Shank's
formula is stunning not only for its elegance but also because it is highly
nontrivial and complex to establish, see comments regarding this in 
\cite[p. 162]{Brezinski}.

\subsection{Numerical illustration}
Extrapolation methods were quite popular for dealing with scalar 
sequences such as those originating  from numerical integration, or from
computing  numerical series.
The following illustration highlights the difference between fixed-point
acceleration and extrapolation.

In the  example we compute $\pi$ from the Arctan expansion: 
\eq{eq:atanz}
\text{atan}(z) = z - \frac{z^3}{3} + \frac{z^5}{5} - \frac{z^7}{7} + \cdots
= \sum_{j=0}^\infty \frac{ (-1)^{j} z^{2j+1}}{2j+1} .
\en
Applying this to $z=1$ will give a sequence that converges to $\pi/4$.
For an arbitrary $z$, we would  take the following sequence, starting with $x_0 = 0$:
\eq{eq:atanz0}
 x_{j+1} = x_j + \frac{(-1)^j z^{2j+1}}{2 j+1}, \quad   j=0,1,\cdots, n-1.
 \en
 In this illustration, we take $n=30$ and $z=1$.

 \begin{figure}[H]
   \centerline{
     \includegraphics[width=0.6\textwidth]{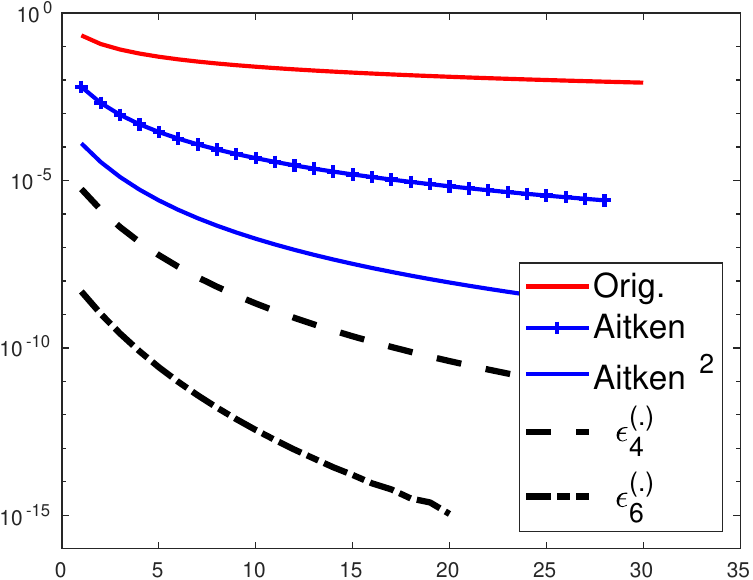}
   }
   \caption{Comparison of a few extrapolation methods for computing $\pi/4$ with Formula~\nref{eq:atanz0}.}\label{fig:comparExtrap}
\end{figure}

Figure~\ref{fig:comparExtrap} shows the convergence of the original sequence
along with 4 extrapolated sequences.  In the figure, Aitken$^2$ refers to Aitken
applied twice, i.e., Aitken is applied to the extrapolated sequence obtained
from applying the standard Aitken process. Note that all extrapolated sequences
are shorter since, as was indicated above, a few
iterates from the original sequence are needed in order to get the
new sequence  started. The plot shows a
remarkable improvement in convergence relative to the original sequence: 15
digits of accuracy are obtained for the 6-th order $\epsilon$-algorithm in 20
steps in comparison to barely 2 digits obtained with 30 steps of the original
sequence.


\subsection{Vector sequences}\label{sec:vecs} 
Extrapolation methods have been generalized to vector sequences in a number of
ways. There is no canonical generalization that seems compelling and natural
but the simplest  consists of applying the acceleration procedure
component-wise. However, this naive approach is not recommended in the literature as
it fails to capture the intrinsic vector character of the sequence.
Many  of the generalizations relied on extending in some way the notion of inverse
of  a vector or that  of the division of a vector by another vector.
For the Aitken procedure this leads to quite a few possible generalizations, see,
e.g., \cite{ramiere15}.

Peter Wynn himself considered a generalization of his scheme to
vectors \cite{Wynn62}.
Note that generalizing the recurrence formula \nref{eq:epsalg} of the \epsalg
only requires that we  define the inverse of a vector. To this end 
Wynn considered several options and ended up adopting a definition
proposed by Samelson:
\eq{eq:invvec}
x\inv = \frac{\bar x} { (\bar x,x)}  = \frac{\bar x} { \| x \|_2^2 } 
\en
where $\bar x$ denotes the complex conjugate of $x$.
This is nothing but  the pseudo-inverse of $x$ viewed as a matrix from $\CC^1 $ to $\CC^n $. It is also known as the Samelson inverse of a vector. 
Recently, Brezinski et al \cite{Brezinski23} considered  more complex extensions of the
$\epsilon$-algorithm by resorting to Clifford Algebra.

The various generalizations of the \epsalg  to vectors based on extending the inverse
of a vector were not too convincing to  practitioners. 
For this reason,  \cite{Brezinski-TEA,Brezinski-mmpe}
adopted a different approach that essentially introduced duality as a tool.
We start with the simple case of Aitken's acceleration for which
we rewrite  the original ansatz \nref{eq:Aitken0} as follows: 
\[
  x_* = x_{j} + \mu \Delta x_j 
\]
where $\mu =  1 / (1-\lambda)$ is a scalar. In the scalar case, it can be readily seen
from \nref{eq:Aitken1}, that $\mu$ is given by:
\eq{eq:mu} \mu = - \frac{\Delta x_j}{\Delta^2
  x_j} .
\en
In the vector case, we need $\mu$ to be a scalar, and so a natural
extension to vector sequences would entail taking inner products of the
numerator and denominator in \nref{eq:mu} with a certain vector $w$:
\eq{eq:muw}
\mu = - \frac{(w,\Delta x_j)}{(w,\Delta^2 x_j)} .
\en
This leads to the formula,
\eq{eq:aitkV}
y_j = x_{j} - \frac{(w,\Delta x_j)}{(w,\Delta^2 x_j)} \ \Delta x_j
= \frac{
\left| \begin{array}{cc} 
x_j           &      x_{j+1} \cr
 (w, \Delta x_j) & (w, \Delta x_{j+1}) 
\end{array} 
\right|} 
{
\left| \begin{array}{cccc} 
 1           &    1    \cr 
 (w, \Delta x_j) & (w, \Delta x_{j+1)}  
\end{array} 
\right|}  . 
\en
We need to clarify notation. The determinant in the numerator of
the right-hand side in \nref{eq:aitkV} now has the vectors $x_{j}$ and $x_{j+1}$
in its first row, instead of scalars. This is to be interpreted with the help of
the usual  expansion 
of this determinant with respect to this row. Thus, this determinant is
evaluated as: 
\begin{align*}
   (w, \Delta x_{j+1}) x_j -  (w,\Delta x_j) x_{j+1} &= 
  (w, \Delta x_{j+1} -\Delta x_j) x_j   + (w,\Delta x_j) (x_j - x_{j+1})\\
 &= (w, \Delta^2 x_j) x_j   \ - \ (w,\Delta x_j) \Delta x_j  .
\end{align*}
This establishes the second equality in \nref{eq:aitkV} by noting that 
the denominator in the last expression of \nref{eq:aitkV} is just $(w,\Delta^2 x_j)$.

The vector $w$ can be selected in a number of ways but it is rather natural to
take $w = \Delta^2 x_j $ and this leads to
\eq{eq:Aitken3} y_{j} = x_{j} -
\frac{ (\Delta x_j, \Delta^2 x_j) }{ \| \Delta^2 x_j \|_2^2 } \Delta x_j .
\en
Our aim in showing the expression on the right of \nref{eq:aitkV} is to unravel
possible extensions of the more general Shanks formula \nref{eq:shanksFormula}.
It turns out that the above formulation of Aitken acceleration for vector
sequences is a one-dimensional version of the RRE extrapolation method, see
Section~\ref{sec:RRE} for details.

The vector version of  Aitken's extrapolation described above can easily be
 extended to a vector form of  Shanks formula \nref{eq:shanksFormula}
by selecting a vector $w$ and replacing every difference
$\Delta x_{j+i}$ in \nref{eq:shanksFormula} by $(w, \Delta x_{j+i})$.
The first row of the determinant in the numerator will still have the vectors
$x_{j+i}$ and the resulting determinant is to be interpreted as was explained above.
Doing this would lead the following extension of Shanks formula~\nref{eq:shanksFormula}:
\eq{eq:shanksFormulaV}
y_j\up{m}  =  
\frac{
\left| \begin{array}{cccc} 
x_j           &      x_{j+1}  & \cdots & x_{j+m} \cr
 (w, \Delta  x_j) & (w, \Delta  x_{j+1})  & \cdots & (w, \Delta  x_{j+m}) \cr 
\vdots        &  \vdots       &  \vdots  &  \vdots     \cr
(w, \Delta  x_{j+m-1})  & (w, \Delta  x_{j+m}) & \cdots & (w, \Delta  x_{j+2m-1})
\end{array} 
\right|} 
{
\left| \begin{array}{cccc} 
 1           &    1      & \cdots & 1 \cr
 (w, \Delta  x_j) & (w, \Delta  x_{j+1})  & \cdots & (w, \Delta  x_{j+m}) \cr 
\vdots        &  \vdots       &  \vdots  &  \vdots     \cr
(w, \Delta  x_{j+m-1})  & (w, \Delta  x_{j+m}) & \cdots & (w, \Delta  x_{j+2m-1})
\end{array} 
\right|} 
\en
This generalization was advocated   in \cite{Brezinski-mmpe}.

However, there is a missing step in our discussion so far: Although we now have
a generalized Shanks formula we still need an effective way to evaluate the related expressions,
hopefully with something similar to the \epsalg, that avoids determinants.
This is not an issue for  Aitken's procedure which corresponds to the case $m=2$
as the corresponding determinants
are trivial. For the cases where $m$ is larger, a  different approach is required.
This was examined by
\cite{Brezinski-mmpe,Brezinski-TEA} who proposed a whole class of  techniques referred to
\emph{Topological \epsalg}, see also, \cite{BrezinskiGenesis2019,Brezinski}.

\subsection{The projection viewpoint and MMPE}\label{sec:mmpe}
Another way to extend extrapolation methods to vector sequences
is to invoke a projection approach. We return to
equations~\nref{eq:Kernel3} where the second equation is taken with $i=0$ and
eliminate the first equation by setting $a_0 =  1 - a_1 - a_2 - \cdots - a_{m}$
which leads to
\eq{Deltax0}
     \Delta x_j + \sum_{i=1}^m a_i (\Delta x_{j+i} - \Delta x_{j}) = 0 .
   \en
   Making use of the simple identity
   $
     u_{j+i} - u_j = \Delta u_{j+i-1} + \Delta u_{j+i-2} +\cdots + \Delta u_{j}
   $ 
   will show that the term $\Delta x_{j+i} - \Delta x_{j} $ is the sum 
   of the vectors $\Delta^2 x_{k} $ for $k=j$ to $k=j+i-1$  and this leads to the following
   reformulation of the left-hand side of \nref{Deltax0}:
     \begin{align*}
       \Delta x_j + \sum_{i=1}^m a_i \sum_{k=j}^{j+i-1} \Delta^2  x_{k} &= 
    \Delta x_j + \sum_{k=j}^{j+m-1} [a_{k-j+1} + a_{k-j+2} +\cdots + a_m] \Delta^2 x_{k} \\
         &\equiv     \Delta x_j + \sum_{k=j}^{j+m-1} \beta_{k-j+1}  \Delta^2 x_{k} 
     \end{align*} 
     where we set $\beta_{k-j+1} \equiv a_{k-j+1} + a_{k-j+2} + \cdots + a_m$.
     Thus, equation \nref{Deltax0} becomes
   \eq{eq:Deltax2}
     \Delta x_j + \sum_{i=1}^{m} \beta_{i} \Delta^2 x_{j+i-1} = 0 
     \en
     where $\beta_i = a_{i}+a_{i+1} + \cdots + a_m$.     It is convenient to   define:
   \begin{align} 
   \Delta X_j & = [\Delta x_j, \Delta x_{j+1}, \cdots , \Delta x_{j+m-1}] \label{eq:Delta1} \\
     \Delta^2 X_j & = [\Delta^2 x_j, \Delta x_{j+1}, \cdots , \Delta^2 x_{j+m-1}] \label{eq:Delta2}\\
     \beta &= [\beta_1, \beta_2, \cdots, \beta_m]^T  .  \label{eq:beta}                    
   \end{align} 
   With this notation the  system~\nref{Deltax0} takes the matrix form:
   \eq{eq:Deltax3}
   \Delta x_j + \Delta^2 X_j \ \beta = 0 . 
   \en
   This is an over-determined system of equations with $m$ unknowns. Taking a
   projection viewpoint, we can select a set of vectors
   $W \ \ \in \ \RR^{n\times m} $ and extract a solution to the projected system
   \eq{eq:shanksProj0} W^T \left( \Delta x_j + \Delta^2 X_j \ \beta \right) = 0
   .  \en Assuming that the $m \times m$ matrix $ W^T [\Delta^2 X_j] $ is
   nonsingular then the accelerated iterate exists and is given by
   \eq{eq:shanksProj} y_j = x_j - \Delta X_j \ \beta \quad \mbox{where} \quad
   \beta = \left[ W^T \Delta^2 X_{j}\right]\inv (W^T \Delta x_j) .  \en A number
   of methods developed in \cite{Brezinski-mmpe} were of the type shown above
   with various choices of the set $W$.  Among these, one approach in particular
   is worth mentioning due do its connection with other methods.

       This approach starts with another  natural extensions of \nref{eq:shanksFormulaV} which
       is to use a different vector $w$ for each of the
       rows of the determinants but apply these to the same $\Delta x_k$ on the same column:
\eq{eq:shanksFormulaV2}
y_j\up{m}  =  
\frac{
\left| \begin{array}{cccc} 
x_j           &      x_{j+1}  & \cdots & x_{j+m} \cr
 (w_1, \Delta  x_j) & (w_1, \Delta  x_{j+1})  & \cdots & (w_1, \Delta  x_{j+m}) \cr 
\vdots        &  \vdots       &  \vdots  &  \vdots     \cr
(w_m, \Delta  x_{j})  & (w_m, \Delta  x_{j+1}) & \cdots & (w_m, \Delta  x_{j+m})
\end{array} 
\right|} 
{
\left| \begin{array}{cccc} 
 1           &    1      & \cdots & 1 \cr
 (w_1, \Delta  x_j) & (w_1, \Delta  x_{j+1})  & \cdots & (w_m, \Delta  x_{j+m}) \cr 
\vdots        &  \vdots       &  \vdots  &  \vdots     \cr
(w_m, \Delta  x_{j})  & (w_m, \Delta  x_{j+1}) & \cdots & (w_m, \Delta  x_{j+m})
\end{array} 
\right|}  \ .
\en
It was first discussed in \cite{Brezinski-mmpe}  and essentially the same method
was independently published  in the Russian literature in 
\cite{Pugachev77}. The method was later 
 rediscovered by  \cite{sidi-al-86}  who gave it the name
 \emph{Modified Minimal Polynomial Extrapolation} (MMPE) by which it is commonly known today.
 Because there is some freedom in the  selection of the set $W$, MMPE represents more
 than just one method.  We will discuss it further in Section~\ref{sec:RRE}. 

 It is rather interesting that the above determinant can be expressed in the
 same form as \nref{eq:shanksProj}.  To see this, we need to process the
 numerator and denominator of \nref{eq:shanksFormulaV2} as follows.  Starting
 from the last column and proceeding backward, we subtract column $k-1$ from
 column $k$, and this is done for $k=m+1, m, \cdots, 2$. With this, the first
 row of the determinant in the denominator will be a one followed by $m$ zeros.
 The first row of the determinant in the numerator will be the vector $x_j$
 followed by $\Delta x_j, \Delta x_{j+1}, \Delta x_{j+m-1}$.  The entries in the
 first columns of both denominators are unchanged.  The block consisting of
 entries $ (2:m+1) \times (2:m+1)$ of both denominators will have the entries
 $(w_i, \Delta^2 x_{j+k-1}) $ in its $k$-th column, with $k=1,\cdots,m$.  This
 block is nothing but the matrix $W^T \Delta^2 X_j$.  To expand the resulting
 determinant in the numerator, we utilize the following relation where $\tau$ is
 a scalar and $S$ is an invertible $m \times m$ matrix: \eq{eq:SchurF}
 \left| \begin{matrix} \tau & f \\ b & S\end{matrix} \right| = \det(S) [\tau - f
 S\inv b ] .  \en
 This relation is also true in the case when $\tau $ is a
 vector in $\RR^{\nu\times 1}$ and $f$ is in $\RR^{\nu \times m}$ with
 the interpretation of such determinants seen earlier.
  After transforming the numerator of \nref{eq:shanksFormulaV2} as discussed above,
  we will obtain a determinant in the form of the left-hand side of \nref{eq:SchurF} in which  $f = \Delta X_j $,
  $b = W^T \Delta X_j$ and $S = W^T \Delta^2 X_j$. Applying the above formula to this
  determinant results in the expression~\nref{eq:shanksProj}.


Vector acceleration algorithms were also defined as processes devoted
solely to vector sequences generated from linear iterative procedures.
The next section  explores this
framework a  little  further.

\subsection{RRE and related methods}\label{sec:RRE}
The difficulties encountered in extending Aitken's method and the \epsalg to
vector sequences led researchers to seek better motivated alternatives, by
focusing on vector sequences generated from specific linear processes.  Thus,
\cite{CabayJackson} introduced a method called the Minimal
Polynomial Extrapolation (MPE) which exploits the low rank character of the set
of sequences in the linear case. At about the same time,  \cite{rre} and
\cite{Mesina77} developed a method dubbed `Reduced Rank Extrapolation' (RRE)
which was quite similar in spirit to MPE.  Our goal here is mainly to link these
methods with the ones seen earlier and developed from a different viewpoint.

We begin by describing RRE by following the notation and steps of the original
paper \cite{Mesina77}.  RRE was initially designed for vector sequences generated from the
following fixed-point (linear) iterative process: \eq{eq:rre0} x_{j+1} = M x_{j} + f .
\en
The author chose to express the extrapolated sequence in the form:
\eq{eq:rre_x}
y_{m} = x_0 + \sum_{i=1}^{m} \beta_i \Delta x_{i-1} 
\en
where $\beta_i$, $i=1,\cdots, m$ are scalars to be determined.

With  the notation ~(\ref{eq:Delta1} -- \ref{eq:beta}) 
we can write  $y_m = x_0 + \Delta X_0 \beta $. For the 
 sequences under consideration, i.e., those defined by  \nref{eq:rre0},
 the relation $\Delta^2 x_j = - (I-M) \Delta x_j$ holds, and therefore we also have
  $ \Delta^2 X_0 = -(I-M) \Delta X_0$.  Furthermore, since we are in effect solving the system
 $(I-M) x = f$ the residual  vector associated with $x$ is  $r = f - (I-M)x $.
 Consider now the residual $r_m$ for the vector $y_m$: 
 \begin{align}
   r_m &= f- (I-M) y_m = f - (I-M) [x_0 + \Delta X_0 \beta] \nonumber\\
       &= f  + M x_0 - x_0 + \Delta^2 X_0 \beta \nonumber\\
       &= \Delta x_0 + \Delta^2 X_0 \beta . \label{eq:resrre}
 \end{align}
 Setting the over-determined system $ \Delta x_0 + \Delta^2 X_0 \beta = 0$ yields the exact same system as
 \nref{eq:Deltax2} with $j=0$.

 The idea  in the original paper was  to determine $\beta$ so as to minimize
 the Euclidean norm of the residual \nref{eq:resrre}.
 It is common to formulate this method in the
 form of a projection technique: the norm 
 $\| \Delta x_0 + \Delta^2 X_0 \beta \|_2 $  is  minimized by imposing
 the condition that the
 residual $r = \Delta x_0 + \Delta^2 X_0 \beta$ be orthogonal to the span of the columns of
 $\Delta^2 X_0$, i.e., to the second order forward differences $\Delta^2 x_i$:
 \eq{eq:rreP} 
   (\Delta^2 X_0)^T [ \Delta x_0 + \Delta^2 X_0 \beta ] = 0 .
 \en
 This is  the exact same system as in \nref{eq:shanksProj0}  with $j=0$ and
 $W=\Delta^2 X_0$. Therefore, RRE is an instance of the MMPE method seen earlier
 where we set $W$ to be $W=\Delta^2 X_0$.
 The case $m = 1$ is  interesting. Indeed,
 when $m=1$ equations~\nref{eq:rre_x} and \nref{eq:rreP} yield
 the same extrapolation as the vector version of the Aitken process shown in Formula
 \nref{eq:Aitken3} for the case $j=0$.  One can therefore view the vector version
 of the Aitken process \nref{eq:Aitken3} as a particular instance of RRE with
 a projection dimension $m$ equal to 1.
 
 In the early development of the method, the idea was to exploit a low-rank
 character of $\Delta^2 X_0$. If $\Delta^2 X_0$ has rank $m$ then $y_m$ will be an
 exact solution of the system $(I-M) x =f $  since $r_m$ will be zero.
 Let us explain this in order  to make a connection with Krylov subspace methods to
 be discussed in Section~\ref{sec:Krylov}.
 From~\nref{eq:rre0} it follows immediately that $x_{j+1} - x_j = M (x_j -x_{j-1})$ and therefore
 \eq{eq:RREKry1}
 \Delta x_j = x_{j+1} - x_j = M^j (x_1 -x_{0})  = M^j r_0 .
 \en
 Therefore, the accelerated vector  $y_m$ is of the form
 $ y_m = x_0+q_{m-1} (M) r_0 $ where $q_{m-1}$ is a polynomial of degree $\le m-1$.
 The residual vector $r_m $ is $f - (I-M) y_m$ and thus:
 \eq{eq:rreKry2}
     r_m = r_0 - (I-M) q_{m-1}(M) r_0 \equiv p_m(M) r_0
 \en
 where $p_m(t) \equiv 1 - (1-t)q_{m-1}(t)$ is a degree $m$ polynomial such that $p_m(1) = 1$.
 The polynomial $q_{m-1}$ and therefore also the residual polynomial $p_m$ is parameterized
 with the
 $m$ coefficients $\beta_i, i=1,\cdots, m$. The minimization problem of RRE, can be
 translated in terms of polynomials:
 find the degree $m$ polynomial $p_m$ satisfying  the constraint $p_m(1)=1$ for which 
 the norm of the residual $p_m(M)r_0 $ is minimal. If the minimal polynomial for $M$ is
 $m$ then the smallest norm residual is zero. This was behind the motivation of the method. 
 Note that the accelerated solution $y_m$ belongs to the subspace
 \[K_m(M,r_0) = \Span \{r_0, M r_0, \cdots, M^{m-1} r_0 \}\]
 which is called a \emph{Krylov subspace}. This is the same subspace as the one,
 more commonly associated with linear systems, in which $M$ is replaced by the coefficient matrix
 $I-M$.

 In the scenario where the minimal polynomial is of degree $m$, the rank of $M$
 is also $m$ hence, the term `Reduced Rank Extrapolation' given to the method.
 Although the method was initially designed with this special case in mind, it
 can clearly be used in a more general setting.  Note also that although the RRE
 acceleration scheme was designed for sequences of the form \nref{eq:rre0}, the
 process is an extrapolation method, in that it does not utilize any other
 information than just the original sequence itself. All we need are the first
 and second order difference matrices $\Delta X_0$ and $\Delta^2 X_0$ and the
 matrix $M$ is never referenced. Furthermore, nothing prevents us from using the
 procedure to accelerate a sequence produced by some nonlinear fixed-point
 iteration.

 The MPE method mentioned at the beginning of this section
 is closely related to RRE. Like RRE, MPE
 expresses the extrapolated solution in the form \nref{eq:rre_x} and so the
 residual of this solution also satisfies \nref{eq:resrre}.  Instead of trying
 to minimize the norm of this residual, MPE imposes a Galerkin condition of the
 form $ W^T( \Delta x_0 + \Delta^2 X_0 \beta) = 0$, but now $W$ is selected to
 be equal to $\Delta X_0$.  As can be seen this can again be derived from the
 MMPE framework discussed earlier.  Clearly, if $M$ is of rank $m$ then $y_m$
 will again be the exact solution of the system.  Note also that RRE was
 presented in difference forms by \cite{rre} and \cite{Mesina77}, and
 the equivalence between the two methods was established in
 \cite{SmithFordSidi87}.


\subsection{Alternative formulations of RRE and MPE}
In formula \nref{eq:rre_x} the accelerated vector $y_m$ is expressed as an
update to $x_0$, the first member of the sequence.  It is also possible to express
it as an update to $x_m$, its most recent member. Indeed,  the affine space
$x_0 + \Span \{ \Delta X_0 \}$ is identical with $x_m + \Span \{ \Delta X_0 \}$
because of the relation:
\[
  x_m = x_ 0 + \Delta x_0 + \Delta x_1 + \cdots + \Delta x_{m-1}
  = x_0 + \Delta X_0 e 
\]
   where $e$ is the vector of all ones. If we set $y_m  = x_m + \Delta X_0 \gamma $
   then the residual in \nref{eq:resrre} becomes $\Delta x_m + \Delta X_0 \gamma$ and so we
   can reformulate RRE as:
\eq{eq:rre2}
y_m = x_m  + \Delta X_0 \gamma  \quad \mbox{where} \quad \gamma = \text{argmin}_\gamma
\| \Delta x_m + \Delta X_0 \gamma \|_2 . 
\en 
A similar formulation also holds for MPE.
The above formulation will be useful when comparing
RRE with the Anderson acceleration to be seen later. 

In another expression of MPE and RRE the $x_i$'s are invoked directly instead of
their differences in \nref{eq:rre_x}.  We will  explain this now because 
similar alternative formulations will appear a few times later in the paper.
Equation \nref{eq:rre_x} yields:
 \begin{align*} 
   y_m &= x_0 + \beta_1 (x_1 - x_0) +  \beta_2 (x_2 - x_1) + \cdots + \beta_{m} (x_m - x_{m-1}) \\
       &= (1 - \beta_1) x_0 + (\beta_1 - \beta_2) x_1 + \cdots + (\beta_{i} -\beta_{i+1}) x_i + \cdots +
         \beta_{m} x_m .
 \end{align*}
 This can be rewritten as
 $   \sum_{i=0}^{m} \alpha_i x_i $ by setting
 $ \alpha_i \equiv  \beta_i -\beta_{i+1}$ for $i=0,\cdots,m$, with the convention that
 $\beta_{m+1} = 0$, and $\beta_{0}  =1$.
 We then observe that
 the $\alpha_i$'s sum up to one.
Thus, we can reformulate \nref{eq:rre_x} in the form
\eq{eq:yjnew}
y_m = \sum_{i=0}^{m} \alpha_i x_i  \quad  \text{with} \quad \sum_{i=0}^{m} \alpha_i = 1.
\en 
The above setting is a quite common alternative to that of
\nref{eq:rre_x} in acceleration methods.
It is also possible to proceed in reverse by formulating an accelerated sequence stated as in
\nref{eq:yjnew} in the form \nref{eq:rre_x} and this was done previously,
see Section~\ref{sec:mmpe}.

\subsection{Additional comments and references}

It should be stressed that extrapolation algorithms, of the type discussed in
this section, often played a major role in providing a framework for the other
classes of methods.  They were invented first, primarily to deal with scalar
sequences, e.g., those produced by quadrature formulas.  Later, they served as
templates to deal with fixed-point iterations, where the fixed-point mapping $g$
was brought to the fore to develop effective techniques.  While
`extrapolation'-type methods were gaining in popularity, physicists and chemists
were seeking new ways of accelerating computationally intensive, and
slowly converging, fixed-point-iterations.  The best known of these fixed-point
techniques is the Self-Consistent Field (SCF) iteration which permeated a large
portion of computational chemistry and quantum mechanics.  As a background, SCF
methods, along with the acceleration tricks developed in the context of the
Kohn-Sham equation, will be summarized in Section~\ref{sec:dft}.  Solving linear
systems of equations is one of the most common practical problems encountered in
computational sciences, so it should not be surprising that acceleration methods
have also been deployed in this context.  The next section addresses the special
case of linear systems.

We conclude this section with a few bibliographical pointers. 
Extrapolation methods generated a rich literature starting in the 1970s and a
number of surveys and books have appeared that provide a wealth of details, both
historical and technical.  Among early books on the topic, we
mention~\cite{Brezinski,Brezinski-RZ-book}.  Most of the developments of
extrapolation methods took place in the 20th century and these are surveyed by
\cite{BrezinskiReview00}.
A number of other  papers  provide an in-depth review of extrapolation and acceleration methods,
see, e.g., \cite{higham2016anderson,JbilouSadok00,sidi2012review}.
A nicely written more recent 
survey of extrapolation and its relation to rational approximation is the 
volume by  \cite{Brezinski-RZ-book20} which contains
a large number of references while also discussing the fascinating lives of the
main contributors to these fields.






\section{Accelerators for linear iterative methods}
\label{sec:LinCase}
Starting with the work of  Gauss in 1823, see ~\cite{Letter2Gerling}, 
quite a few iterative methods for solving linear systems of equations were developed.
The idea of accelerating these iterative procedures is natural and it has been invoked
repeatedly in the past.  A diverse set of techniques were advocated for this purpose,
including Richardson's method,  Chebyshev acceleration, and the class of
Krylov subspace methods.  It appears that acceleration techniques were
first suggested in the early 20th century with the work of Richardson, and
reappeared in force a few decades later as modern electronic computers started
to emerge.

\subsection{Richardson's legacy}\label{sec:Richardson}
Consider a linear system of the form
\eq{eq:Ax=b} A x = b,
\en
where $A \ \in \ \RR^{n \times n}$ and $b \ \in \ \RR^{n}$.
 Adopting the point of view expressed in Equation~\nref{eq:FixedPt1} with the function
 $f(x) \equiv Ax - b $ we obtain the iteration
 \eq{eq:Richardson0} x_{j+1}  = x_j -  \gamma (Ax_j -b) = x_j + \gamma r_j  
 \en
 where $x_j$ is the current iterate and $r_j = b - A x_j$ the related residual.
 This  simple `first-order' scheme was  proposed by ~\cite{Richardson1910}.
Assuming that the eigenvalues of $A$ are real and included in the interval
$[\alpha, \beta]$ with $\alpha > 0$, it is not difficult to see that the scheme will
converge to the solution for any $\gamma $ such that $0 \lt \gamma < 2/\beta$ and that the value of
$\gamma$ that yields the best convergence rate is $\gamma_{opt} = 2/[\alpha + \beta]$,
see, e.g., \cite{Saad-book2}.

Richardson also considered a more general procedure where the scalar
$\gamma$  changed at each step: 
\eq{eq:mr} x_{j+1} = x_j + \gamma_j r_j . \en
If $x_*$ is the exact solution, he studied the question of selecting the
 best sequence of $\gamma_j$'s to use if we want the
error norm $\| x_j -x_*\|_2$ at the $j$-th step to be the smallest possible.
This will be addressed in the next section. 

\typeout{======================================== $x_*$ defined???}

A number of procedures discovered at different times can be cast into the
general Richardson iteration framework represented by \nref{eq:mr}.  Among these, two
examples stand out. The first 
example is the \emph{Minimal Residual iteration} (MR) where $\gamma_j $
is selected as the value of $\gamma$ that minimizes the next residual norm:
$\| b - A (x_j + \gamma r_j)\|_2$.  The second is the steepest descent
algorithm where instead of minimizing the residual norm, we 
minimize $\| x - x_* \|_A $ where $\| v\|_A^2 = (Av,v)$.  Both methods are
one-dimensional projection techniques and will be discussed in
Section~\ref{sec:proj}.

\subsection{The Chebyshev procedure}\label{sec:Cheb}
 \cite{Richardson1910} seems to have been the first to consider a
general scheme given by \nref{eq:mr}, where the $\gamma_j$'s are sought with the
goal of achieving the fastest possible convergence.  From \nref{eq:mr} we get
\eq{eq:resPol1} r_{j+1} = b - A (x_j +\gamma_j r_j) = r_j - \gamma_j A r_j = (I
-\gamma_j A ) r_j \en which leads to the relation \eq{eq:resPol2} r_{j+1} = (I
-\gamma_j A ) (I -\gamma_{j-1} A )\cdots (I -\gamma_0 A ) r_0 \equiv p_{j+1}(A)
r_0 \en
where $p_{j+1}(t) $ is a polynomial of degree $j+1$. Since
$p_{j+1} (0) =1$ then $p_{j+1}$ is of the form $p_{j+1}(t) = 1 - t q_j (t)$,
with deg $(q_j) = j$.  Therefore,
\[r_{j+1} = (I - A q_j(A) ) r_0 = (b - A x_0) - A q_j(A) r_0 = b - A [x_0 + q_j (A) r_0], \]
  which means that 
\eq{eq:xjp1}
x_{j+1} = x_{0} + q_j(A) r_0 .
\en

Richardson  worked with error vectors instead of residuals. Defining the error  
 $u_j = x_* - x_j$ where $x_*$ is the exact solution, and   using  the relation
 $A u_j = b - A x_j = r_j$ we can multiply \nref{eq:resPol2} by
 $A\inv$ to obtain  the relation
\eq{eq:errPol2}
  u_{j+1} = p_{j+1} (A) u_0 
  \en
  for the  error vector at step $j+1$,   where $p_{j+1}$ is the same polynomial as above.
He formulated the problem of selecting the
  $\gamma_i$'s in \nref{eq:errPol2} with a goal of making the error $u_{j+1}$ as
  small as possible.  The $\gamma_i$'s in his formula were the inverses of the
  ones above - but the reasoning is identical.
  Such a scheme can be viewed as an `acceleration' of the first-order Richardson
  method \nref{eq:Richardson0} seen earlier.   Richardson assumed only the
  knowledge of an interval $[\alpha, \beta]$ containing the eigenvalues of $A$.
  When $A$ is Symmetric Positive Definite (SPD) then there exists $\alpha, \beta$,
  with $\alpha >0$ such that   $\Lambda(A) \subset [\alpha, \ \beta] $.


  If we wish to minimize the maximum deviation from zero in the interval then
  the best polynomial can be found from a well-known and simple result in
  approximation theory.  We will reason with residuals, recall equation
  \nref{eq:resPol2}, and denote by $\PP_{j+1,0}$ the set of Polynomials $p$ of
  degree $j+1$ such that $p(0) = 1$.  Thus, $p_{j+1}$ in \nref{eq:resPol2} and
  \nref{eq:errPol2} is a member of $\PP_{j+1,0}$ and our problem is to find a
  polynomial $ p \in \PP_{j+1,0}$ such that
  $ \max_{t \ \in \ [\alpha, \beta] } |p (\lambda)| $ is minimal. In other words
  we seek the solution to the min-max problem:
  \eq{eq:minmax} \min_{p \in
    \PP_{j+1,0} }\ \max_{ t \in [\alpha,\beta] } | p( t ) | .  \en The
  polynomial $T_{j+1}$ that realizes the solution to \nref{eq:minmax} is known
  and it can be expressed in terms of the Chebyshev polynomials of the first
  kind $C_j(t)$ : \eq{eq:cheb} T_{j+1} (t) \equiv \frac{1 }{ \sigma_{j+1} }
  C_{j+1} \left( \frac{{ \beta + \alpha - 2 t } }{ {\beta-\alpha} } \right)
  \quad \mbox{with} \quad \sigma_{j+1} \equiv C_{j+1} \left( \frac{{\beta +
        \alpha }}{{\beta -\alpha} } \right) .  \en If the polynomial $T_{j+1}$
  is set to be the same as $p_{j+1}$ in \nref{eq:resPol2} then clearly the
  inverses of its roots will yield the best sequence of $\gamma_i$'s to use in
  \nref{eq:mr}. Richardson seems to have been unaware of Chebyshev
  polynomials. Instead, his approach was to select the roots $1/\gamma_i$ by
  spreading them in an ad-hoc fashion in $[\alpha, \beta]$.  One has to wait
  more than four decades before this idea, or similar ones based on Chebyshev
  polynomials, appeared.

A few of the early methods in this context computed the roots of the modified
Chebyshev polynomial \nref{eq:cheb} and used the inverses of these roots as the $\gamma_j$'s
in \nref{eq:mr} \cite{shortley53,sheldon55,Young54}.  These methods were
difficult to use in practice and prone to numerical instability.  A far more 
elegant approach is to exploit the 3-term recurrence of the Chebyshev
polynomials:
\eq{eq:ChRec} C_{j+1}(t) = 2tC_j(t) - C_{j-1}(t), \quad j\ge 1.
\en starting with $C_0(t) = 1, C_1(t) = t$.  A 1952 article by 
\cite{Lanczos52} suggested a process for preprocessing a right-hand side of a
linear system prior to solving it with what was then a precursor to a Krylov subspace
method. The residual polynomials related to this process are more
 complicated than Chebyshev polynomials as was noted by \cite{Young54}. However they too rely on Chebyshev
 polynomials and Lanczos does exploit the 3-term recurrence in his 1952 article while Shortley, Sheldon, and
 Young do not.

The first real acceleration scheme based on the Chebyshev polynomials that
exploits the 3-term recurrence seems to be the 1959 paper by
\cite{VonNeumann59}. The article included two appendices written by von
Neumann and the second of these discussed the method.  The method
must have been developed around the year 1956 or earlier by von Neumann who died
on Feb 8, 1957. Two years after the von Neumann article
\cite{GolubVarga1961} published a very similar technique which they
named `semi-iterative method'.  They included a footnote 
acknowledging the  earlier contribution by von Neumann.


The arguments of von Neumann's contribution were rooted in acceleration
techniques with a goal of producing a process for linearly combining the previous
iterates of a given sequence, in order to achieve faster
convergence. Specifically, the new sequence is of the form
$ y_j = \sum_{i=0}^{j} \eta_{i,j} x_{i} $ where the $\eta$'s satisfy the
constraint that $\sum\eta_{i,j} = 1 $.  Though this may seem different from what
was done above, it actually amounts to the same idea.

Indeed, the residual of the `accelerated' sequence is:
\[
 b - A y_j = b -  \sum_{i=0}^{j} \eta_{i,j}  A x_{i}
  =  \sum_{i=0}^j \eta_{i,j} [b - A x_{i}] 
  =  \sum_{i=0}^{j} \eta_{i,j} p_i(A) r_0 
\]
where $p_i(t)$ is the residual polynomial of degree $i$ associated with the original sequence,
i.e., it is defined by  \nref{eq:resPol2}. As was already seen, the
polynomial $p_i$ satisfies  the constraint
$p_i(0)=1$. Because  $\sum_{i=0}^j  \eta_{ij} p_i(0) = \sum  \eta_{ij} =1$, the new
residual polynomial $\tilde p_j(t) = \sum_{i=0}^{j} \eta_{i,j} p_i(t)$ also satisfies
the same condition  $\tilde p_j(0) = 1$. Therefore, we can say that  
the procedure seeks to find a degree $j$ polynomial, expressed in the form
$ \sum_{i=0}^j  \eta_{i,j} p_i(t)$, whose value at zero is one and which is optimal in some sense.

We now return to Chebyshev acceleration to provide details on the procedure discovered
by John von Neumann and Golub and Varga.   Letting: 
\eq{eq:scalars}
\theta \equiv \frac{\beta + \alpha}{ 2}, \quad 
\delta \equiv  \frac{\beta - \alpha}{ 2},
\en 
we can write $T_j$ defined   by \nref{eq:cheb} as: 
\eq{eq:cheb1}
T_j (t) \equiv \frac{1 }{ \sigma_j } 
C_j \left( \frac{{ \theta -  t } }{ \delta } \right) 
 \quad \mbox{with} \quad
\sigma_j \equiv 
 C_j \left( \frac{{\theta}}{{\delta} } \right)    . 
\en 
The three-term recurrence for the Chebyshev polynomials leads to
 \eq{eq:sigmaj}
\sigma_{j+1} = 2 \ \frac{\theta }{\delta}  \sigma_j 
- \sigma_{j-1} , \ j=1, 2\, \ldots,
\quad \mbox{with:}  \quad  \sigma_0 = 1, \quad 
\sigma_1 =  \frac{\theta }{\delta} .
\en
We combine  the recurrence \nref{eq:ChRec} and \nref{eq:cheb1} into
a 3-term recurrence for the polynomials $T_j$, for $j\ge 1$:
\eq{eq:Rec0}
\sigma_{j+1} T_{j+1} (t)  =  2 \ \frac{\theta- t}{\delta}
\sigma_j T_j(t) - \sigma_{j-1} T_{j-1} (t) 
 \en
starting with $ T_0(t) = 1 , \ T_1(t) = 1 - \frac{ t  }{ \theta } $. 

We now need a way of expressing the sequence of iterates from the above recurrence
of the residual polynomials. There are at least two ways of doing this.
One  idea is to note that $ r_{j+1}-r_j  = - A( x_{j+1} - x_j) = T_{j+1}(A)r_0 -  T_j(A)r_0$.
Therefore we need to find a recurrence for $(T_{j+1}(t)-T_j(t))/(-t)$. Going back
to \nref{eq:Rec0} and exploiting the recurrence \nref{eq:sigmaj} we write for $j\ge 1$:
\begin{align}
  \sigma_{j+1} T_{j+1} -\sigma_{j+1} T_j
  & =  2 \ \frac{\theta- t}{\delta} \sigma_j T_j 
    -  \left(2 \frac{ \theta }{\delta} \sigma_j - \sigma_{j-1}\right)T_j -  \sigma_{j-1} T_{j-1}
  \nonumber \\
  & =  - 2 \ \frac{t}{\delta} \sigma_j T_j  +  \sigma_{j-1} (T_j -T_{j-1}) \qquad  \to \nonumber  \\
    \frac{T_{j+1} -T_j}{-t} 
  & =2 \ \frac{1}{\delta} \frac{\sigma_j}{\sigma_{j+1} } T_j  + \frac{\sigma_{j-1}}{\sigma_{j+1} }
    \frac{T_j -T_{j-1}}{-t} . \label{eq:Rec1}  
\end{align}
When translated into vectors of the iterative  scheme, $T_j(t)$ will give $r_j$, and
$(T_{j+1}(t) - T_j(t))/(-t)$ will translate to $x_{j+1}-x_j$ and if we set $d_j = x_{j+1}-x_j$
then \nref{eq:Rec1} yields:
\eq{eq:Rec2}
d_j = \frac{2}{\delta} \frac{\sigma_j}{\sigma_{j+1} } r_j  + \frac{\sigma_{j-1}}{\sigma_{j+1} } d_{j-1} . 
\en
We can set: 
$\rho_j  \equiv  \frac{ \sigma_j }{ \sigma_{j+1} } , \quad j=1,2, \ldots ,$ 
and invoke  the relation ~\nref{eq:sigmaj} to get 
$\rho_j = 1 / [2 \sigma_1 - \rho_{j-1}]$ and  then \nref{eq:Rec2} becomes
$d_j = \frac{\rho_j}{\delta} r_j  + \frac{\rho_j}{\rho_{j-1}} d_{j-1}$.
This leads to Algorithm~\ref{alg:ChebAcc}.


A second way to obtain a recurrence relation for the iterates $x_j$ is to 
write the  error as $x_{j+1} - x_* = T_{j+1}(A) (x_0 - x_*) $ and then
exploit  \nref{eq:Rec0}:
\begin{align} 
  \sigma_{j+1} (x_{j+1} - x_*)
  &= 2 \ \frac{\theta I - A}{ \delta}   \sigma_j (x_j - x_*) - \sigma_{j-1} (x_{j-1} - x_*)  \\
  &= 2 \frac{\theta}{\delta} \sigma_j (x_j - x_*) +
    \frac{2 \sigma_j}{\delta} r_j - \sigma_{j-1} (x_{j-1} - x_*) 
\end{align}  
From the relation \nref{eq:sigmaj}, we see that the terms in $x_*$ cancel out and we get:
\eq{eq:RecA1}
\sigma_{j+1} x_{j+1} 
= 2 \frac{\theta}{\delta} \sigma_j x_j + \frac{2 \sigma_j}{\delta} r_j - \sigma_{j-1} x_{j-1} .
\en
Finally, invoking \nref{eq:sigmaj} again we can write 
$ 2 \frac{\theta}{\delta} \sigma_j = \sigma_{j+1} + \sigma_{j-1}$ and hence:
\eq{eq:RecA2}
\sigma_{j+1} x_{j+1} 
= \sigma_{j+1} x_j + \sigma_{j-1} (x_j - x_{j-1}) + \frac{2 \sigma_j}{\delta} r_j \ .
\en

\begin{algorithm}[tb]
  \centering
  \caption{Chebyshev Acceleration}\label{alg:ChebAcc}
    \begin{algorithmic}[1]
      \State \textbf{Input}: System $A, b$, initial guess $x_0$ and parameters $\delta, \theta$\\ 
      Set $r_0 \equiv b - Ax_0$, $\sigma_1 = \theta  / \delta $; 
      $\rho_0 = 1/\sigma_1 $ and $d_0 = \frac{ 1 }{ {\theta } } r_0 $;
      \For{$j=0,1,\cdots,$ until convergence:} 
      \State $x_{j+1}  =  x_j + d_j$
      \State $ r_{j+1} = r_j - A d_j$ 
      \State $ \rho_{j+1}  = ( 2 \sigma_1 - \rho_j )\inv$
      \State $    d_{j+1}   = \frac{ 2 \rho_{j+1} }{  \delta }  r_{j+1} + \rho_{j+1} \rho_j d_j $ 
      \EndFor
    \end{algorithmic}
    \end{algorithm}
    Note that  Lines 4 and 7 of the algorithm,
    can be merged in order to rewrite the iteration as follows: 
    \eq{eq:chebOneLine} 
      x_{j+1} = x_j + \rho_j \left[  \rho_{j-1} (x_j - x_{j-1}) + 
        \frac{  2   }{  \delta }  (b - A x_j) \right]
    \en
    which is a `second order iteration', of the same class as 
    \emph{momentum-type methods} seen in optimization and machine learning,
    to be discussed in the     next section.
This is a  common form used in particular by \cite{GolubVarga1961} in their
seminal work on ``semi-iterative methods''.
A major advantage of the  Chebyshev iterative method is that it does not require
any inner products. On the other hand, the scheme requires   estimates for the extremal
eigenvalue in order to set the sequence of scalars necessary for the iteration. 

It is easy to see that the
iteration parameters $\rho_j$ used in the algorithm converge to a limit.
Indeed, the usual formulas for Chebyshev polynomials show that 
\[
  \sigma_j = \text{ch} \left[ j \text{ch}\inv  \frac{\theta}{\delta} \right] =
  \frac{1}{2} \left[
    \left(\frac{\theta}{\delta} + \sqrt{\left(\frac{\theta}{\delta}\right)^2 -1} \right)^j
 + \left(\frac{\theta}{\delta} + \sqrt{\left(\frac{\theta}{\delta}\right)^2 -1} \right)^{-j}
\right]  . \]
      As a result, we have
      \eq{eq:limrho}
      \lim_{j \to \infty} \rho_j = \lim_{j \to \infty} \frac{\sigma_j}{\sigma_{j+1}}
      =  \left(\frac{\theta}{\delta} + \sqrt{\left(\frac{\theta}{\delta}\right)^2 -1} \right)^{-1}
      = \frac{\theta}{\delta} - \sqrt{\left(\frac{\theta}{\delta}\right)^2 -1} 
      \equiv \rho . 
      \en

      One can therefore consider replacing the scalars $\rho_j$ by their limit
      in the algorithm. This scheme, which we will refer to as the
      \emph{Stationary Chebyshev iteration}, typically results in a small
      reduction in convergence speed relative to the
      standard Chebyshev iteration  \cite{Kerkhoven-Saad}.

\subsection{An overview of Krylov subspace methods}\label{sec:Krylov}
Polynomial iterations of the type introduced by Richardson  lead to
a residual of the form~\eqref{eq:resPol2} where $p_{j+1}$ is a polynomial of degree $j+1$, see Section~\ref{sec:Richardson}.
The related approximate solution  at step $j+1$    is given in Equation~\eqref{eq:xjp1}.
The approximate solution $x_j$ at step $j$,  is of the form $x_0 + \delta $
 where $\delta $ belongs to the subspace
\eq{eq:Kj}
\K_j(A,r_0) = \Span \{r_0, A r_0, \cdots, A^{j-1} r_0 \} .
\en 

This is the \emph{$j$-th Krylov subspace}. As can be seen, $\K_j(A,r_0)$ is
nothing but the space of all vectors of the form $q(A)r_0$ where $q$ is an
arbitrary polynomial of degree not exceeding $j-1$.  When there is no ambiguity
we denote $\K_j(A,r_0)$ by $\K_j$.

Krylov subspace methods for solving a system of the form \nref{eq:Ax=b}, are
projection methods on the subspace \eqref{eq:Kj} and can be viewed as a form of
optimal polynomial acceleration implemented via a projection process.  We begin
with a brief discussion of projection methods.

  \subsubsection{Projection methods}\label{sec:proj}
The primary objective of a projection method is to extract an approximate
solution to a problem from a subspace. Suppose we wish to obtain a
solution $x \ \in \ \RR^n$ to a given problem $(P)$.  The problem is projected into
a problem $(\tilde P)$ set in a subspace $\K$ of $\RR^n$ from which we obtain
an approximate solution $\tilde x$. Typically, the dimension
$m$ of $\K$ is much smaller than $n$, the size of the system.

When applied to the solution of linear systems of equations, we assume the
knowledge of some initial guess $x_0$ to the solution and two subspaces $\K$ and
$\calL$ both of dimension $m$.
From these the following projected problem is formulated:
\eq{eq:Prox0}
\mbox{Find} \quad \tilde x = x_0 +  \delta, \ \delta \in \K \quad
\mbox{such that}\quad  b-A \tilde x \perp \calL  .
\en
We have  $m$ degrees of freedom (dimension of $\K$) and $m$ constraints (dimension of $\calL$),
and so \nref{eq:Prox0} will result in an $  m \times m $ linear system
which is nonsingular under certain mild conditions on $\K $ and $\calL$.
The Galerkin projection process just described satisfies important
\emph{optimality properties} that play an essential role in their analysis.
There are two important cases.

\def\tlx{\tilde x}

\paragraph{Orthogonal Projection (OP) methods.}
This case corresponds to the situation where the subspaces $\K$ and $\calL$ are the
same. In this scenario the residual $b - A \tilde x$ is orthogonal to $\K$ and 
the corresponding approximate solution is the closest vector
from affine space $x_0 + \K$ 
to the exact solution $x_*$, where distance is measured with  the $A$ norm:
$\| v \|_A = (Av,v)^{1/2}$.

\begin{proposition} \label{prop:opt0} 
Assume that $A$ is Symmetric Positive Definite and $\calL =\K$.
Then a vector $\tlx$ is the result of an (orthogonal) projection method
onto $\K$ with the starting vector $x_0$ if and only if  it 
minimizes the $A$-norm of the error over $x_0 +\K$, i.e.,
if and only if 
\[
E(\tlx ) = \min_{x \in x_0 +\K} E(x) , \] 
where 
\[ E(x) \equiv (A (x_* - x), x_* - x)^{1/2}   .  \] 
\end{proposition}

The necessary condition means the following:
\[
  \calL = \K \quad \text{and} \quad A \ \text{SPD}  \longrightarrow
  \| x_*  -  \tilde x \|_A = \min_{ x \in x_0 + \K } \| x_* - x  \|_A . \]

It is interesting to note that Richardson's scheme shown in Equation
\nref{eq:mr} can be cast to include another prominent one-dimensional projection
method namely the well-known steepest descent algorithm. Here, $A$
is assumed to be symmetric positive definite and at each step the algorithm   
computes the vector of the form $x=x_{j} + \gamma r_j$ where $r_j = b - Ax_j$,
that satisfies the orthogonality condition $b - A x \perp r_j$. 
This leads to selecting at each
step $\gamma = \gamma_j \equiv (r_j,r_j)/(A r_j, r_j)$ which, according to the
above proposition, minimizes
$\| x - x_* \|_A $ over $\gamma$.  Another method in this category is the Conjugate
Gradient method which will be covered shortly.

\paragraph{Minimal Residual (MR) methods.}
This case  corresponds to the situation when 
$\calL = A \K$.
It can be shown that if $A$ is nonsingular, then $\tilde x$ minimizes the
Euclidean norm of the residual over the affine space $x_0 + \K$.

  \begin{proposition}\label{prop:opt2} 
Let  $A$ be  an arbitrary  square matrix and assume that  $\calL =A \K$.
Then a vector $\tlx$ is the result of an (oblique)
projection method onto $\K$ orthogonally to $\calL$  with the starting
vector $x_0$ if and only if it minimizes the $2$-norm of the 
residual vector $b - A x $ over $x  \ \in \ x_0 +\K$, i.e., if and
only if 
\[
R(\tlx ) = \min_{x \in x_0 +\K} R(x)   , \]
where
$R(x) \equiv  \| b - A x \|_2 $. 
\end{proposition}

The necessary condition now means the following:
\[ \calL = A \K \quad \text{and} \quad A \  \text{nonsingular}\  \longrightarrow 
    \| b - A \tilde x \|_2 = \min_{ x \in x_0 + \K } \| b - A x \|_2 . \]
Methods in this category include the Conjugate Residual (CR) method,
the Generalized Conjugate Residual (GCR) method, and GMRES among others, see
Section~\ref{sec:MR}.

  Another instance of Richardson's general iteration of the form 
  \nref{eq:mr} is the \emph{Minimal Residual iteration} (MR)
where $\gamma_j $ is selected as the value of $\gamma$ that minimizes
the next residual norm: $\| b - A (x_j + \gamma r_j)\|_2$. A little calculation
will show that the optimal $\gamma$ is 
$\gamma_j = (r_j, A r_j)/ (Ar_j, A r_j)$, where it is assumed that
$Ar_j \ne 0$, see, e.g., \cite{Saad-book2}.
MR is a one-dimensional projection method since it computes a vector of the form
$x_j + d$ where $ d \ \in \XX = \Span \{ r_j \} $ which satisfies the
orthogonality condition $b - A x \perp A \XX$.

\subsubsection{OP-Krylov and the Conjugate Gradient method}
One particularly important instance in the OP class of projection methods  is the
\emph{Conjugate Gradient  algorithm} (CG) algorithm,  a clever
implementation of the case where $\K=\calL$ is  a Krylov subspace
of the form \nref{eq:Kj} and $A$ is SPD. This implementation is shown
in Algorithm~\ref{alg:CG}.

\begin{algorithm}[htb]
  \caption{Conjugate Gradient} \label{alg:CG}  
    \begin{algorithmic}[1]
   \State Compute $r_0   :=   b -A x_0 $,  $p_0   := r_0$.  
   \For {$j=0,1,\ldots$,   until convergence Do:}  
   \State   $ \alpha_j   :=   ( r_j  ,  r_j )  /  ( Ap_j  ,  p_j )$ 
   \State     $ x_{j+1}   :=   x_j  +  \alpha_j   p_j $ 
   \State     $ r_{j+1}   :=   r_j  -  \alpha_j   Ap_j $ 
   \State    $\beta_j  :=   (r_{j+1}  ,  r_{j+1} )  /  ( r_j ,  r_j ) $ 
   \State    $ p_{j+1} :=   r_{j+1}  + \beta_j   p_j $ 
   \EndFor
   \end{algorithmic} 
\end{algorithm}

The discovery of CG \cite{Hestenes-Stiefel} was a major breakthrough in the
early 1950s.  The original CG article was co-authored by Magnus Hestenes [UCLA]
and Eduard Stiefel [ETH], but these authors made the discovery independently
and, as they learned of each other's work, decided to publish the paper
together, see \cite{KrylovHist22} for a brief history of Krylov methods.

Nowadays the CG algorithm is often presented as a projection method on Krylov
subspaces, but in their paper Hestenes and Stiefel invoked purely geometric
arguments as their insight from the 2-dimensional case and knowledge of `conics'
led them to the notion of \emph{`conjugate directions'}.  The goal is to find
the minimum of $f(x) = \half (Ax,x) - (b,x)$. In $\RR^2$ the contour lines of
$f(x)$, i.e., the sets $\{ x | f(x) = \kappa \} $ where $\kappa$ is a constant,
are con-focal ellipses the center of which is $x_*$  the desired solution.
    
\begin{figure}[h]
\begin{center}
  \includegraphics[width=0.5\textwidth]{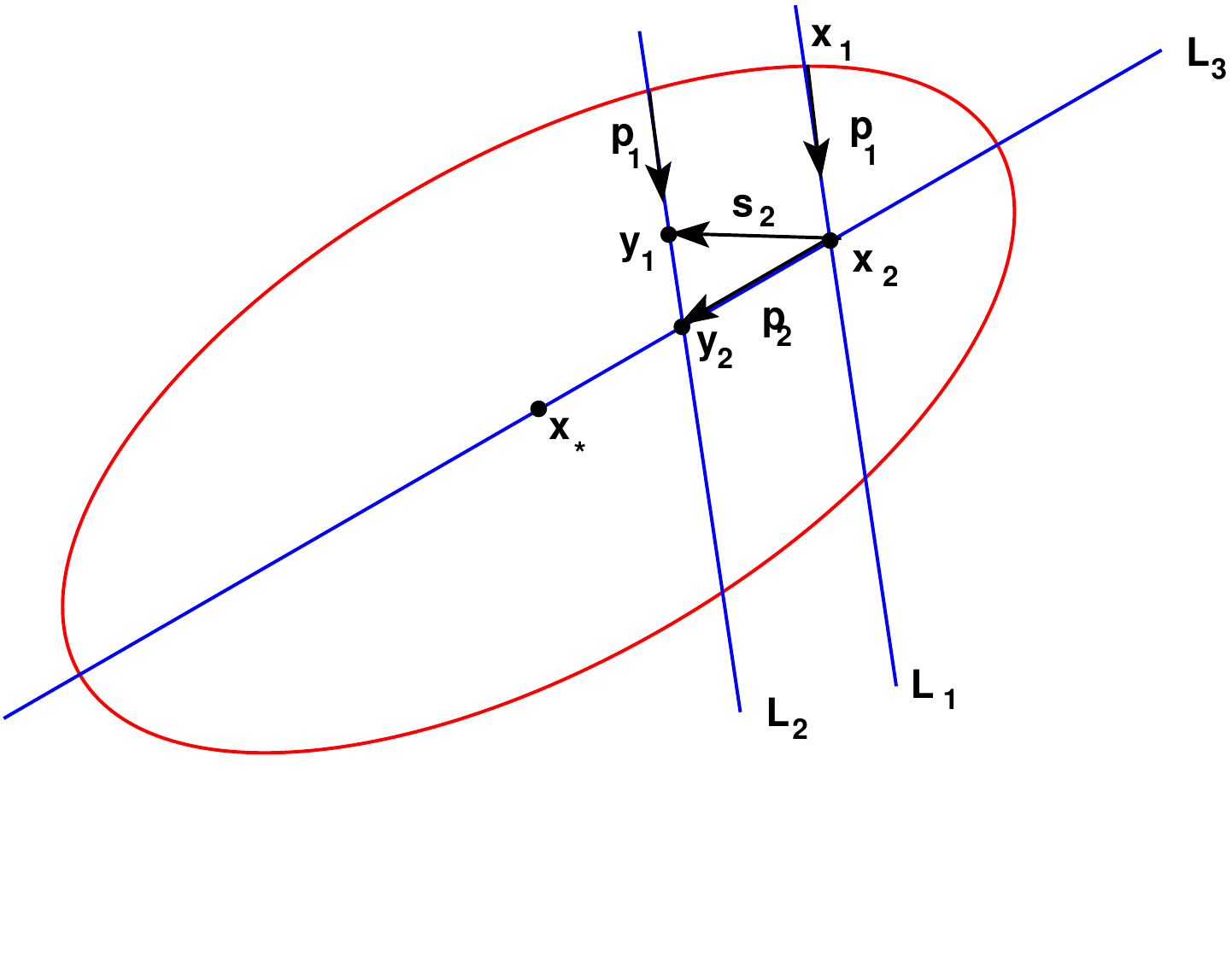} 
\end{center}
\caption{Illustration of the construction of conjugate gradient directions. 
  This illustration is based on the reference \cite{HestTodd91}}\label{fig:cgfig}
\end{figure}

Figure~\ref{fig:cgfig} provides an illustration borrowed from \cite{HestTodd91}.
The minimum of $f(x) $ on a 
given chord of the ellipse is reached in the middle of the chord.
The middles $x_2, y_2$ of  two given parallel chords $L_1, L_2$,  
will define a line $L_3$ that passes through the center $x_*$ of the ellipse.
The minimum of $f(x)$ along $L_3$ will be
at the center, i.e., at the exact solution.  If $x_2 = x_1 + \alpha_1 p_1$, is
an iterate and $y_1 = x_2 + s_2$, an intermediate iterate, the line $L_2$:
$y(t) = y_1 + t p_1$, is parallel to $L_1$. Its minimum is reached at a point
$y_2 = y_1 + \beta_1 p_1$, and the direction $p_2 = y_2 - x_2 \equiv p_1 + \beta_1 s_2$
is conjugate to $p_1$.  In the CG
algorithm, the direction $s_2$ is taken to be the residual $r_2$ of $x_2$.

At the same time as the CG work was unfolding, Cornelius Lanczos, who also
worked at the INA, developed a similar method using a different notation and
viewpoint \cite{Lanczos52}.  This was a minimal residual approach (MR method
discussed below),
implemented with the use of what we now call a Lanczos basis \cite{Lanczos50}
for solving eigenvalue problems, another major
contribution to numerical linear algebra from the same era.

The Conjugate Gradient algorithm was not too well received initially.  It was
viewed as an unstable direct solution method and for this reason it laid dormant
for about two decades. Two important articles played a role in its revival. The
first one was by John \cite{Reid71} who discovered that the method could be
rather effective when used as an iterative procedure. At the same time, work by
\cite{Paige-thesis,paig:80a}, analyzed the Lanczos process for
eigenvalue problems from the point
of view of inexact arithmetic, and this work gave a much better understanding of
how the loss of orthogonality affected the process. The same loss of
orthogonality affected the CG method and lead to its initial rejection by
numerical analysts.

\subsubsection{MR-Krylov methods, GCR, GMRES}\label{sec:MR} 
Quite a few  projection methods on the Krylov subspaces $\K_j (A,r_0)$
were developed starting in the late 1970s, with the objective of
minimizing the residual norm. As was seen earlier this
optimality condition is equivalent to the property that the approximate solution
$\tilde x \in \K_j(A,r_0)$ be orthogonal to the subspace $ A \K_j(A,r_0)$.


 Implementations with orthonormal basis of $K_m$ lead to the Generalized Minimal
 Residual (GMRES) method \cite{Saad-Schultz-GMRES}, a procedure based on the
 Arnoldi process. Other implementations included Axelsson's GCG-LS method
 \cite{Axelsson80}, ORTHOMIN (Vinsome \cite{Vinsome76}), ORTHODIR (Young and
 Jea, \cite{Jea-Young}), the Generalized Conjugate Residual Method (GCR) of
 \cite{Eis-Elm-Sch}, among others.  A rather
 exhaustive treatise on this work can be found in the book by \cite{Meurant-Tebbens}.
We now briefly discuss the Generalized Conjugate Residual (GCR) method for
solving the system \nref{eq:Ax=b}, since this approach will be exploited later.

The original GCR algorithm \cite{Eis-Elm-Sch} for solving \nref{eq:Ax=b} exploits  an
orthogonal basis of the subspace $A \K_j(A,r_0)$.
The $j$-th step of the procedure can be described as follows. 
Assume that we already have vectors
$p_0, p_0, \cdots, p_{j}$ that are $A^T A $-orthogonal, i.e., such that
$(A p_i, Ap_k) = 0 \ \mbox{if} \ i\ne k$ for $i,k \le j$.
Thus, the set $\{p_i\}_{i=0,\cdots,j}$ forms an $A^TA$ -orthogonal 
basis of the Krylov subspace $\K_{j+1}$.  
In this situation the solution $x_{j+1}$ at the current iteration, i.e., the one
that minimizes the residual norm in
$x_0 + \K_{j+1} = x_0 + \Span \{p_0, p_1, \cdots, p_{j} \}$ becomes easy to express.
It can be written as follows, see, \cite[Lemma 6.21]{Saad-book2} and
Lemma~\ref{lem:gcr} below:
\eq{eq:gcr-sol} x_{j+1} = x_{j} + \frac{(r_{j}, A p_{j})}{(A p_{j},A
  p_{j})} p_{j} ,  \en
where $x_j$ is the current iterate, and  $r_j$  the related residual $r_j = b - Ax_j$.

Once $x_{j+1}$ is computed, the next basis vector, i.e., $p_{j+1}$,  is obtained from 
$A^TA$-orthogonalizing the residual vector $r_{j+1} = b - A x_{j+1}$
 against the previous $p_i$'s by the loop:
\eq{eq:AtA-orth}
   p_{j+1} = r_{j+1} - \sum_{i=0}^{j} \beta_{ij} p_i \quad \mbox{where} \quad
   \beta_{ij} :=  ( A r_{j+1} , Ap_i) / ( A p_i, Ap_i) . 
\en 
The above describes in a succinct way one step of the algorithm. Practically it
is necessary to keep a set of vectors for the $p_i$'s and another set for
the $Ap_i$'s. We will set $v_i = A p_i$ in what follows.
In the classical (`full') version of GCR, these sets are denoted by
\eq{eq:PjVj}
P_j=[p_0, p_1,\cdots, p_j] ; \qquad V_j=[v_0, v_1,\cdots, v_j] .
\en 
Note here are that \emph{all} previous search directions $p_i$ and the
corresponding $v_i$'s must be saved in this full version.

We will bring two modifications to the basic procedure just described. The first is to
 introduce a  \emph{``truncated''} version of the algorithm whereby
 only the most recent $\min \{m,j+1 \}$ vectors $\{p_i\}$ and $\{Ap_i\}$ are kept.
 Thus, the summation in \nref{eq:AtA-orth} starts at $i= [j-m+1]$ instead of $i=0$,
 i.e., it starts at $i=0$ for $j<m$ and at $i=j-m+1$ otherwise.
 The sets $P_j, V_j$ in \nref{eq:PjVj} are replaced by:
 \eq{eq:PjVjA} 
P_j=[p_{[j-m+1]},\cdots, p_j] ; \qquad V_j=[v_{[j-m+1]},\cdots, v_j] .
\en

The second change is to 
make the set of $v_i$'s \emph{orthonormal}, i.e., such that $(v_j, v_i) = \delta_{ij}$.
Thus, the new vector $v_{j+1} $ is made orthonormal to
$ v_{j-m+1},v_{j-m+2},\cdots, v_{j},$ when
$j-m+1\ge 0$ or to  $v_{0},v_{1},\cdots, v_{j}$ otherwise.

\begin{algorithm}[htb]
  \centering
  \caption{TGCR(m)}\label{alg:gcr}
    \begin{algorithmic}[1]
  \State \textbf{Input}: Matrix $A$, Right-hand side $b$,
  initial guess  $x_0$. \\
  Set $r_0 \equiv b-Ax_0$; $v = A r_0$; 
  \State $v_0 = v/\| v \|_2$;   $p_0 = r_0 / \| v \|_2$; 
  \For{$j=0,1,2,\cdots,$ Until convergence} 
\State  $\alpha_j = \langle r_j, v_j\rangle $
\State $x_{j+1} = x_j + \alpha_j p_j$
\State $r_{j+1} = r_j - \alpha_j v_j$
\State $p = r_{j+1}$;  \quad  $ v = A p $;
\If{j$>0$}  [with: $V_j, P_j$ defined in \nref{eq:PjVjA}]
\State Compute $s = V_j^T v$
\State Compute $v = v - V_js $ and $p = p - P_j s$
  \EndIf
\State $p_{j+1} :=p /\| v\|_2$ ; \qquad  $v_{j+1} :=v/\|v\|_2$ ; 
\EndFor
\end{algorithmic}
\end{algorithm}

Regarding notation, it may be helpful to observe that the last column of $P_j$
(resp. $V_j$) is always $p_j$ (resp. $v_j$) and that the number of columns of
$P_j$ (and also $V_j$) is $\min \{m,j+1\}$.  Note also that in practice, it
is preferable to replace the orthogonalization steps in Lines~10-11 of the
algorithm by a modified Gram-Schmidt procedure.

The following lemma explains why the update to the solution takes the simple
form of Equation~\ref{eq:gcr-sol}.  This result was stated in a slightly
different form in \cite[Lemma 6.21]{Saad-book2}.

  \begin{lemma}\label{lem:gcr} 
Assume that we are solving the linear system $Ax=b$ and
    let $\{ p_i, v_i \}_{i=[j-m+1]:j}$  be a paired set of vectors with
    $v_i = A p_i, i=[j-m+1], \cdots,   j$.  Assume also   that
    the set  $V_j $ is orthonormal. Then the solution vector of the affine space
    $x_{j  }  + \Span \{ P_j \}$ with smallest residual norm is 
    $x_{j  } + P_j y_j $ where $y_j = V_j^T r_j $. In addition,
    only the bottom component of $y_j$, namely $v_j^T r_j $, is nonzero.
  \end{lemma}
  \begin{proof} 
    The residual of  $x = x_j + P_j y$ is  $r = r_j - A P_j y = r_j - V_j y$ and
    its norm is smallest when $V_j^T r = 0$. Hence the 1st result.
    Next, this condition implies that the 
    inner products $v_i^T r_{j+1}$ are all zero
    when $i \le j$, so the vector $y_{j+1}=V_{j+1}^T r_{j+1} $ at the next iteration will have
    only one nonzero component, namely its last one,
    i.e., $v_{j+1}^T r_{j+1}$. This proves the second result for the index
    $j+1$ replaced by $j$.   
  \end{proof}
  It follows from the lemma that the approximate solution obtained by a
  projection method that minimizes the residual norm in the affine space
  $x_j + \Span\{P_j\}$ can be written in the simple form
  $x_{j+1} = x_j + \alpha_j p_j$ where $\alpha_j = v_j^T r_j$. This explains formula
  \nref{eq:gcr-sol} shown above and Lines 5-6 of 
  Algorithm~\ref{alg:gcr} when we take into consideration the orthonormality
  of the $v_i$'s.

We will call the `full-window' case of the algorithm the situation when there is no
truncation. This is equivalent to setting  $m=\infty$ in the algorithm.
The truncated variation to GCR $(m<\infty)$, which we will call Truncated GCR (TGCR)
was first introduced in \cite{Vinsome76} and was named `ORTHOMIN'.
\cite{Eis-Elm-Sch}  established a number of results for both the full
window case ($m=\infty$) and the truncated case ($m<\infty$) including, a
convergence analysis for the situation when $A$ is positive real, i.e., when its 
symmetric part is positive definite.
In addition to the orthogonality of the vectors
$Ap_i$, or  equivalently the $A^TA$-orthogonality of the $p_i$'s, 
 another property of (full) GCR is that the residual vectors
that it produces are `semi-conjugate' in that $ (r_i, Ar_j ) = 0$ for $i > j$.

Note that when $j\ge m$ in TGCR $(m<\infty)$, then the approximate solution
$x_j$ no longer minimizes the residual norm over the whole Krylov subspace
$x_0+\K_j(A,r_0)$ but only over $x_{j-m} + \mbox{span} \{p_{j-m},\cdots p_j\}$, see
~\cite[Th. 4.1]{Eis-Elm-Sch}.

\subsection{Momentum-based techniques}

The Chebyshev iteration provides a good introduction to the notion of
\emph{momemtum}. It is sufficient to 
 frame the method for an optimization problem, where we seek to minimize
 the quadratic function $\phi(x) = \frac{1}{2} x^T A x - b^T x$ where $A$ is SPD.
In this case $\nabla \phi(x) = Ax - b$ which is the negative of the residual. 
With this we see that Chebyshev iteration \nref{eq:chebOneLine} can be written as:
\eq{eq:mmt1}
x_{j+1} = x_j + \eta_j \Delta x_{j-1} -\nu_j \nabla \phi(x_j)
\en
where $\eta_j = \rho_j \rho_{j-1}$ and $\nu_j = 2 \rho_j /\delta$.
Recall the notation: $\Delta x_{j-1} = x_j - x_{j-1}$.

\subsubsection{The `heavy-ball' method}\label{sec:HeavyBall}
Equation~\nref{eq:mmt1} is the general form of a gradient-type method \emph{with
  momentum}, whereby the next iterate $x_{j+1}$ is a combination of the term
$x_j - \nu_j \nabla \phi(x_j) $, which can be viewed as a standard gradient-type
iterate, and the previous increment, i.e., the difference $x_j-x_{j-1}$. This
difference $\Delta x_{j-1}$ is often termed \emph{`velocity'} and denoted by
$v_{j-1}$ in the literature.  Thus, a method with momentum takes the gradient
iterate from the current point and adds a multiple of \emph{velocity}.  A
comparison is often made with a mechanical system representing a ball rolling
downhill.  Often the term \emph{heavy-ball} method is used to describe the
iteration \nref{eq:mmt1} in which the coefficients $\eta_j, \mu_j$ are constant.
An illustration is provided in Figure~\ref{fig:mtm}.

It is common to rewrite Equation~\nref{eq:mmt1} by explicitly invoking a
momentum part.  This can be done by defining
$v_j \equiv \Delta x_j = x_{j+1} -x_{j}$. Then the update \nref{eq:mmt1} can be
written in two parts as \eq{eq:mmtV1} \left\{ \begin{array}{lcl}
                                                v_j &=& \eta_j v_{j-1} - \nu_j \nabla \phi(x_j) \\
                                                x_{j+1} &=& x_j + v_j \ .
\end{array} \right.  
\en
However, often the velocity  $v_j$ is defined with the opposite sign
\footnote{The motivation here is that when $\eta = 0$, which corresponds to the gradient method
without momentum, the vector $v_j$ in the update $x_{j+1} = x_j + v_j$   should be
a negative multiple of the gradient of $\phi$, so changing the sign makes $v_j$ a positive multiple of the
gradient.}, i.e., 
 $v_j \equiv - \Delta x_j =  x_j - x_{j+1}$,  in which case  the update \nref{eq:mmt1} becomes
\eq{eq:mmtV2}
 \left\{ \begin{array}{lcl}
  v_j &=& \eta_j v_{j-1} + \nu_j \nabla \phi(x_j) \\
  x_{j+1} &=& x_j - v_j .
\end{array} \right.
\en
Both of expressions \nref{eq:mmtV1} and \nref{eq:mmtV2} can be found in the
literature but we will utilize \nref{eq:mmtV2} which is equivalent to a form
seen in Machine Learning (ML). In ML, the scalar parameters of the
sequence are constant, i.e., $\eta_j \equiv \eta, \nu_j \equiv \nu$ and the
velocity vector $v_j$ is often scaled by $\nu$, i.e., we set $v_j = \nu w_j $ upon
which \nref{eq:mmtV2} becomes
\eq{eq:mmtV2a}
\left\{ \begin{array}{lcl}
          w_j &=& \eta w_{j-1} + \nabla \phi(x_j) \\
          x_{j+1} &=& x_j - \nu w_j \ .
\end{array} \right. 
\en
In this way, $w_j$ is just  a damped average of previous gradients where the
damping coefficient is a power of $\eta$ that gives more weight to recent gradients.
In Deep Learning, the gradient is actually a sampled gradient corresponding to a `batch' of
functions that make up the final objective function, see, Section~\ref{sec:stoch}.

\begin{figure}
  \centering{
\includegraphics[width=0.6\textwidth]{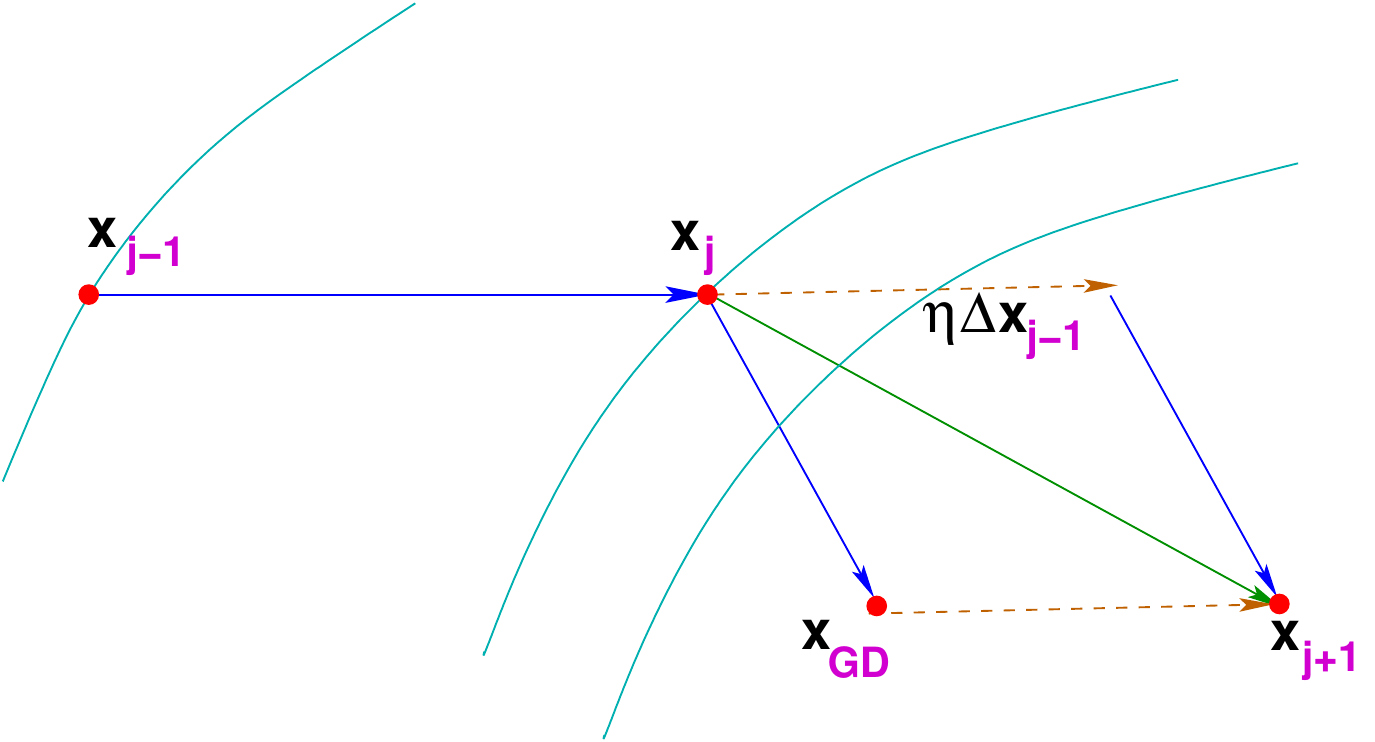}}
\caption{Illustration of the gradient method with momentum
}\label{fig:mtm} 
\end{figure}

\subsubsection{Convergence}\label{sec:convCheb}
The convergence of Chebyshev iteration for linear problems 
is well understood, see, e.g.,
\cite{Saad-book3}. Here we consider, more generally, the momentum scheme \nref{eq:mmt1} but
we restrict our study to  the particular case
where the scalars $\eta_j, \nu_j$ are constant, denoted by $\eta, \nu$ respectively. We also make the same assumption as for the Chebyshev iteration
that the eigenvalues of $A$ are real and located in the interval
$[ \alpha, \beta ]$ with $\alpha > 0$.
Equation \nref{eq:mmt1} becomes
\begin{align} 
x_{j+1} &= x_j + \eta  (x_j-x_{j-1}) -\nu (A x_j - b) \nonumber  \\
        &= [(1 + \eta) I - \nu A] x_j  - \eta x_{j-1}  + \nu b  \label{eq:mmt2}
\end{align} 
To analyze the convergence of the above iteration we can write it in the form
\eq{eq:mmt3}
\begin{pmatrix} x_{j+1} \\ x_j \end{pmatrix} =
\left( \begin{array}{c|c} 
         (1 + \eta) I - \nu A  & -\eta I \\ \hline
         I & 0  \end{array}  \right)        
\begin{pmatrix}  x_j  \\ x_{j- 1} \end{pmatrix} +
\begin{pmatrix} \nu b  \\ 0  \end{pmatrix}  . 
\en
It is helpful to  introduce two matrices: 
\eq{eq:2mats}
B = \frac{1}{\delta} (A - \theta I) \quad \mbox{and} \quad 
C = \frac{1}{2} \left[(1 + \eta) I - \nu A \right]
\en
where we recall that $\delta, \theta$ are defined in \nref{eq:scalars}
and  that the eigenvalues $\lambda_i(B)$ are in the interval $[-1, 1]$.
Then the iteration matrix in \nref{eq:mmt3} is
\eq{eq:MatG}
G =
\begin{pmatrix}  2 C & -\eta I   \\ I & 0 \end{pmatrix} .
\en 
If  $\mu_j$ $j=1,\cdots,n$, are the eigenvalues of $C$, then those of $G$ are
\eq{eq:lamG}
\lambda_j = \mu_j \pm \sqrt{\mu_j^2 - \eta} . 
\en
This expression can help determine if the scheme \nref{eq:mmt1} will converge.
As a particular case, 
a sufficient condition for convergence is that $ 0< \eta < 1$ and all $\mu_j$'s be less that
$\sqrt{\eta} $ in magnitude. In this case $\mu_j^2 -\eta $ is negative and the modulus of
$\lambda_j$ is a constant equal to $\sqrt{\eta}$, which is independent of $j$.


As an example, we will look at   what happens in the case of
the stationary Chebyshev iteration defined earlier.

\begin{proposition}
  Consider  the stationary Chebyshev iteration
  in which $\eta_j = \eta \equiv \rho^2$ and
  $\nu_j = \nu = 2 (\rho/\delta)$, where $\rho $ was 
 defined in \nref{eq:limrho}.  
  Then each of the eigenvalues $\mu_j$ of
  the matrix $C$ in \nref{eq:2mats}   satisfies the inequality
  $|\mu_j | \le \rho $.
  In addition, if $\rho <1 $ the stationary Chebyshev iteration converges
  and its convergence factor, i.e., 
  the spectral radius of the matrix $G$ in \nref{eq:MatG}, is equal to  $\rho$.
\end{proposition}

\begin{proof}
  The eigenvalues $\mu_j $ of $C$ are related to those of $B$ as follows
  \eq{eq:rellam}
  \mu_j = \frac{1}{2}
  \left[ 1+\eta - \nu \lambda_j(A)\right]
  =\frac{1}{2} \left[ 1+\eta - \nu (\theta + \delta \lambda_j(B))\right] .
  \en
  Substituting  $\eta = \rho^2$ and   $\nu = 2 \rho/\delta$, leads to: 
  \eq{eq:rellam1}
  \mu_j 
  =\frac{1}{2} \left[ 1+\rho^2 - \frac{2 \rho}{\delta}
    (\theta + \delta \lambda_j(B))\right]
  = \frac{1}{2} \left[ \rho^2 - 2 \frac{\theta}{\delta} \rho +1 \right]
   - \rho \lambda_j(B)  .
   \en
   It can be observed that $\rho$ in \nref{eq:limrho} is a root of the
   quadratic term in the  brackets in the second part, i.e.,
   $ \rho^2 - 2 \frac{\theta}{\delta} \rho +1 = 0$, and so
   $ \mu_j = - \rho \lambda_j(B) $.  Since  
   $|\lambda_j(B)| \le 1$, it is clear that
   $| \mu_j | \le \rho$ and therefore the term
   $\mu_j^2 - \eta = \mu_j^2 - \rho^2$ in    \nref{eq:lamG} is non-positive.
   Hence,  the    eigenvalues $\lambda_j $ are of the form 
   $\lambda_j = \mu_j \pm  i \sqrt{\rho^2 - \mu_j^2}$ and they all have the
   same modulus $\rho$. 
\end{proof}

It may seem counter-intuitive that a simple fixed-point
iteration like \nref{eq:mmt3} can be competitive with more
advanced schemes but we note that we have doubled the dimension of the problem
relative to a simple first order scheme.
The strategy of moving a  problem into a higher dimension
to achieve better convergence is common.

\subsubsection{Nesterov acceleration} \label{sec:nesterov} 
A slight variation of the  momentum scheme discussed above is 
\emph{Nesterov}'s iteration \cite{nesterov}. In this approach
the gradient is evaluated at an intermediate point instead of the most recent iterate:

\eq{eq:nesterov}
x_{j+1} = x_j + \eta_j \Delta x_{j-1} - \nu_j
\nabla \phi (x_j + \eta_j \Delta x_{j-1} ) .
\en
Using earlier notation where $v_{j-1} = - \Delta x_{j-1}$  this can also be re-written as:
\eq{eq:nestV}
 \left\{ \begin{array}{lcl}
  v_j &=& \eta_j v_{j-1} + \nu_j \nabla \phi(x_j-\eta_j v_{j-1}) \\
  x_{j+1} &=& x_j - v_j
\end{array} \right.
  \en

To analyze convergence in the linear case, we need to rewrite \nref{eq:nesterov} in a block form similar to 
\nref{eq:mmt3}. Setting $\eta_j = \eta, \nu_j = \nu$ and $\nabla \phi(x) = Ax-b$ in 
\nref{eq:nesterov} yields:
\begin{align*} 
  x_{j+1} & = x_j + \eta (x_j - x_{j-1}) - \nu \left[A (x_j + \eta(x_j - x_{j-1}) ) - b \right] \\
          & = (1+\eta) x_j -\eta x_{j-1} - \nu (1+\eta) A x_j + \nu\eta A x_{j-1} + \nu b \\
          & = (1+\eta) (I - \nu A) x_j - \eta [I - \nu A] x_{j-1}     + \nu b 
\end{align*}
which leads to:
\eq{eq:mmt5}
\begin{pmatrix} x_{j+1} \\ x_j \end{pmatrix} =
\left( \begin{array}{c|c} 
         (1 + \eta) (I - \nu A)  & -\eta (I - \nu A)  \\ \hline
         I & 0  \end{array}  \right)        
\begin{pmatrix}  x_j  \\ x_{j- 1} \end{pmatrix} +
\begin{pmatrix} \nu b  \\ 0  \end{pmatrix}
\en

The scheme represented by \nref{eq:nesterov} and its matrix form \nref{eq:mmt5}
can be viewed as a way of accelerating a Richardson-type scheme, see,
\nref{eq:Richardson0}, with the parameters $\gamma_j$ set equal to $ \nu$ for
all $j$.  The corresponding iteration matrix \nref{eq:resPol1} is the matrix
$B\equiv I - \nu A$.  If $\mu_j, j=1,\cdots,n$ are the eigenvalues of $B$, those
of the iteration matrix in \nref{eq:mmt5} are roots of the equation
$\lambda^2 - (1+\eta)\mu_j \lambda + \eta \mu_j = 0$ and they can be written as
follows:

\eq{eq:eigNest}
\lambda_j^{\pm}  
  = \mu_j \left[ \frac{1+\eta}{2} \pm \sqrt{\left(\frac{1+\eta}{2}\right)^2 -\frac{\eta}{\mu_j}}\right]
\en
with the convention that when $\mu_j = 0$ both roots are zero. Assume that the eigenvalues
of $B$ are real and such that 
\eq{eq:intv}
- \theta_1 \le \mu_j \le \theta_2 
\en
where $\theta_1, \theta_2$ are non-negative.
It is  convenient to define
\eq{eq:thetastar}
\theta_* =  \frac{\eta}{\half (1+\eta)^2} .
\en

The simplest scenario to analyze is when $\nu$ and $\eta$ are selected such that  $\theta_1=0$
and $\theta_2 = \theta_*$. For example, we can first set $\nu  = 1/\lambda_{max}(A)$ to satisfy
the requirement $\theta_1 =0$ since in this case
\[\mu_i = 1-\nu \lambda_i(A) \ge 1 - \frac{\lambda_{max}}{\lambda_{max}}  = 0 .
  \] 
  Then with $\nu$ set to the value just selected, we will find   $\eta $ so that $\theta_* = \theta_2$,
  i.e., $ \frac{\eta}{(\half (1+\eta))^2} = \theta_2$, which yields  the quadratic equation:
   \eq{eq:quad}
     \eta^2 -2 \left(\frac{2 }{\theta_2} -1\right) \eta +1 = 0 . 
   \en
   Note that if the eigenvalues of $A$ are    positive, then $\theta_2 \le 1$ and
   the discriminant $\Delta = \left(\frac{2 }{\theta_2} -1\right)^2 -1$ is non-negative so the roots
   are real. 
   It is important to also note that the product of the two roots of this equation is one,
   and we will select $\eta$ to be  the smallest of the two roots so we know it will not exceed one.

   In this scenario, the eigenvalues $\mu_j$ are in the interval $[0, \ \theta_* ]$ and will
yield a negative value inside the square root in formula \nref{eq:eigNest} for $\lambda_j^\pm$.
The  squared modulus of $\lambda_j^\pm$ is 
\eq{eq:eig1N}
\mu_j^2 \left[ \left(\frac{1+\eta}{2}\right)^2 + \frac{\eta}{\mu_j} -
  \left(\frac{1+\eta}{2}\right)^2 \right] = \eta \mu_j .
\en
So each of these eigenvalues is transformed into a complex conjugate pair of
eigenvalues with modulus $\sqrt{\eta \mu_j}$,  but since $\mu_j \le \theta_*$
the maximum modulus is $\le  2  \eta/(1+\eta)$ which is less than one
when $\mu$ is  selected to be the smallest root of \nref{eq:quad} as discussed above.
In this situation  the scheme will converge with a convergence
factor $ 2 \eta/(1+\eta)$.

It is also possible to analyze convergence in a more general scenario when
\[
  -\theta_1 \le 0 \le \theta_* \le \theta_2
\]
In this situation we would need to distinguish between 3 different cases corresponding to the
intervals $[\theta_1, \ 0) $, $[0,  \ \theta_*)$, and $ [\theta_*, \ \theta_2]$.
The analysis is more complex and is omitted.

\section{Accelerating the SCF iteration}\label{sec:dft} 
The Schr\"odinger's equation $ H \Psi = E \Psi $ was the result of a few decades
of exceptionally productive research starting in the late 19th century that
revolutionized our understanding of matter at the nanoscale, see
\cite{Gamov-book}.  In this equation $E$ is the energy of the system and
$\Psi = \Psi(r_1, r_2, \ldots, r_n, R_1, R_2, \ldots, R_N) $ the wavefunction
which depends on the positions $R_i$, $r_i$ of the nuclei and electrons
respectively.  This dependence makes the equation intractable except for the
smallest cases.
Thus, in 1929, Dirac famously stated:
\emph{The underlying physical laws necessary for the mathematical theory of a large part of
physics and the whole of chemistry are thus completely known, and the difficulty is only
that the exact application of these laws leads to equations much too complicated to be
soluble. It therefore becomes desirable that approximate practical methods of applying
quantum mechanics should be developed.
}
It took about four additional decades to develop such
methods, a good representative of which is \emph{Density Functional Theory}
(DFT).

\subsection{The Kohn-Sham equations and the SCF iteration}
DFT manages to obtain a tractable version of the original equation by
replacing the many-particle wavefunction $\Psi$ with one that depends
on one fictitious particle which will generate the same charge density as
that of the original interacting multi-particle interacting system.
The foundation of the method is  a fundamental result by \cite{KS65}, namely
 that observable quantities are  uniquely determined by the ground state charge density.
The resulting Kohn-Sham equation can be written as follows:
\eq{eq:KS}
\left[ -  \frac{h^2}{2 m} \nabla^2 + V_{tot} [\rho(r)] 
\right] \Psi(r) = E \Psi(r)
\en
where the total potential $V_{tot}$ is the sum of three terms:
\eq{eq:Ppts}  V_{tot} = V_{ion} + V_H + V_{xc} . 
\en
Both the Hartree potential $V_{H}$ and the Exchange-Correlation
$V_{xc}$ depend on the charge density (or electron density)
which is defined from `occupied' wavefunctions:
\eq{eq:rho}
\rho (r) = \sum_i^{occup} | \Psi_i (r) |^2 . 
\en
Thus, the Hartree Potential $V_H $ is solution of the Poisson equation:
$\nabla^2 V_H = - 4 \pi \rho(r)$ on the domain of interest while
the Exchange-Correlation potential depends nonlinearly on $\rho$
under various expressions that depend on the model.
Once discretized, Equation \eqref{eq:KS} yields a large eigenvalue problem,
typically involving $n_{occup}$  eigenvalues and associated eigenvectors.
For additional details, see \cite{SaChSh-Nano}.

The basic SCF iteration would start from a guessed value of the charge density
and other variables and obtain a total potential from it. Then the eigenvalue
problem is solved to give  a new set of wavefunctions $\psi_i$ from which
a new charge density $\rho$ is computed and the cycle is repeated until
self- consistence is reached, i.e., until the input $\rho_{in}$ and output $\rho_{out}$
are close enough.
One can capture the process of computing $\rho$ in this way by a fixed-point
iteration:
\eq{eq:fixp}
\rho_{j+1} = g(\rho_j)
\en
where $g$ is a rather complex mapping that computes a new charge density
from the old one. This mapping involves solving an eigenvalue problem,
a Poisson equation, and making some other updates.

Convergence of this process can be slow in some specific situations. In addition,
the cost of each iteration can be enormous. For example, one could have a
discretized eigenvalue problem with a size in the tens of millions, and a number of
occupied states, i.e., eigenpairs to compute,  in the tens of thousands.
It is therefore natural  to think of employing procedures that
accelerate the process and this can be done in a number of ways.

\subsection{Simple mixing}\label{sec:simpleMix}
The idea of acceleration used in the context of DFT is different from that of
extrapolation discussed earlier. Here, researchers invoke the idea
of `mixing'  a current $\rho$ with older ones. The most basic  of these is 
known as `simple mixing' and it defines a new $\rho$ from an old one as follows,
where $\gamma$ is a scalar parameter:

\begin{align} 
  \rho_{j+1} & = \gamma g(\rho_j) + (1-\gamma) \rho_j  \label{eq:mixing0}   \\
               & = \rho_j  + \gamma( g( \rho_j) -  \rho_j)  \label {eq:mixing1}  
\end{align}

It may be helpful to link the above acceleration scheme with other
methods that have been developed in different contexts and to this end
a proper notation is key to unraveling these links.
We will replace $\rho$ by a more general variable $x$, so \nref{eq:fixp} becomes
\eq{eq:fixp1}
x_{j+1} = g(x_j)
\en
and we will denote by $f$ the function:
\eq{eq:fxgx} f(x) =
g(x) - x .
\en
With this,  the `simple mixing' iteration \eqref{eq:mixing1} becomes:
\eq{eq:mixing2} x_{j+1} = x_j + \gamma f(x_j) .  \en
The first link that will help set the notation and terminology is with Richarson's iteration,
see Equation \nref{eq:Richardson0}, for solving linear systems of the form $Ax=b$.
In this case 
$g(x) = x + \gamma (b-Ax)$ and $f(x) = b-Ax$.  So, Richardson's iteration
amounts to the simple mixing scheme above where $f(x)$ \emph{is the residual
  $b-Ax$.}  We will refer to the general scheme defined by \eqref{eq:mixing2} as
Richardson's iteration, or `linear mixing',  in order to avoid the term
`simple mixing' which is proper to the physics literature.
Often, linear iterations are expressed in the form 
$x_{j+1} = M x_j + b $. In this case we seek to solve the system $(I-M) x = b$ and 
$f(x) = (I-M)x - b$, is  the negative of the residual.

The next link is with gradient descent algorithms for minimizing a scalar function
$\phi(x)$. Assuming that the function is convex and differentiable, the iteration reads:
\eq{eq:GD}
x_{j+1}  = x_j - \gamma \grad \phi(x_j)
\en
where $\gamma$ is a positive scalar  selected to ensure a decrease of the function $\phi(x)$.
In the special case where $\phi(x) = \half x^T A x - b^T x $ and $A$ is symmetric then
$\grad \phi(x) = Ax-b$ which is the negative of the residual,
i.e.,  $\nabla \phi (x) = - (b - Ax)$. This leads to the gradient method for minimizing
$\phi(x)$, and thereby solving the system $Ax=b$, when $A$ is SPD.

\subsection{Anderson mixing}\label{sec:AA1}
The  article \cite{Anderson65} presented what the author called an
`extrapolation algorithm' for accelerating sequences produced by fixed-point
iterations. The method became well-known in the physics literature as ``Anderson
mixing'' and later as ``Anderson Acceleration'' (AA) among numerical analysts.  The
following description of the algorithm introduces minimal changes to the
notation adopted in the original paper.  Anderson's scheme aimed at generalizing
the simple mixing discussed earlier.  It also starts with a mapping $g$ that
takes an input $x$ into some output $y$, so that $y = g(x)$.  The pair $x,y$ is
`self-consistent' when the output and input values are the same, i.e., when
$\|y-x\|_2 = 0$, or small enough in practice.  In the following we define $f$ to be
the residual $f(x) = g(x) - x$.  First, Anderson considered a pair consisting
of an intermediate iterate $\bar x_j$  and an associated `linear residual' $\bar f_j$:

\begin{align} 
  \bar x_j &= x_j + \sum_{i=[j-m]}^{j-1}  \theta_i (x_{i}-x_j) ,
                 \label{eq:xbar} \\
  \bar f_j &= f_j + \sum_{i=[j-m]}^{j-1}  \theta_i (f_{i}-f_j) .
                 \label{eq:fbar}
\end{align}
Since  $\bar f _j$ can be viewed as a linearized residual for $\bar x_j$,
the idea is to determine the set of $\theta_i$'s so as to minimize the
2-norm of  $\bar f_j$:
\eq{eq:theta} \min_{ \{
  \theta_i \} } \left\| f_j + \sum_{i=[j-m]}^{j-1} \theta_i
  (f_{i}-f_j)\right\|_2 .  \en Once the optimal
$\theta = [\theta_{[j-m]}, \ \cdots \ , \theta_{j-1}]^T $ is found, the vectors
$\bar x_j, \bar f_j$ are computed according to (\ref{eq:xbar} -- \ref{eq:fbar})
and the next iterate is then defined as:
\eq{eq:and3} x_{j+1} = \bar x_j +
\beta \bar f_j . 
\en
In the original article, $\beta$ was dependent on $j$, but we removed
this dependence for simplicity, as this also reflects common practice.

When comparing the Anderson scheme to Pulay mixing discussed in the next section,
it becomes useful to  rewrite  equations
(\ref{eq:xbar} -- \ref{eq:fbar}) by 
defining $\theta_j = 1 - \sum_{i=[j-m]}^{j-1} \theta_i $. This leads to
the mathematically equivalent equations:

\begin{align} 
    \bar x_j &=\sum_{i=[j-m]}^{j}  \theta_i x_{i}
                   \label{eq:xbarP} \\
    \bar f_j &= \sum_{i=[j-m]}^{j}  \theta_i f_{i}
                   \label{eq:fbarP} \\
    \{ \theta_i \}  &= \text{argmin} \left\{   
                      \left\| \sum_{i=[j-m]}^j \theta_i  f_i\right\|_2 
                     \quad \text{subject to:} \ \sum_{i=[j-m]}^j \theta_i = 1
                      \right\}  \label{eq:thetaP} 
\end{align}

Formulation (\ref{eq:xbar} -- \ref{eq:theta}) leads to a standard (unconstrained) least-squares problem to solve,
see \nref{eq:theta}. It can be viewed as   a straightforward alternative
to  the formulation  (\ref{eq:xbarP} -- \ref{eq:thetaP}) which requires solving a 
constrained optimization problem.
The first formulation is still not an efficient one from the implementation point of view
-- mainly because the sets of vectors $ f_i - f_j$  and $ x_i - x_j$ for
$i=[j-m],\cdots, j-1]$ must be computed at each step.
An equivalent algorithm that avoids this will be seen in Section~\ref{sec:AA2}.

\subsection{DIIS (Pulay mixing)}\label{sec:Pulay}
The ``Direct Inversion in the Iterative Subspace'' (DIIS) is a method introduced by
 \cite{Pulay80} to address the same acceleration needs as Anderson's method.
It is well-known as `Pulay mixing' and widely used  in computational chemistry.
The method defines the new iterate in the following form:
\eq{eq:pulay}
 x_{j+1} = \sum_{i=[j-m]}^j \theta_i g(x_{i}) ,
\en
where the $\theta_i$'s are to be determined parameters such that
 $\sum_{i=[j-m]}^j \theta_i = 1$.
Defining the residual for iteration $i$ as $f_i = g(x_i) - x_i $,
the $\theta_i$'s are selected to minimize the norm of the linearized residual,
i.e., they result from solving the following optimization problem
\eq{eq:pul1} 
 \min \left\|   \sum_{i=[j-m]}^j \theta_i f_i \right\|_2  \quad
\text{subject to:} \quad  \sum_{i=[j-m]}^j \theta_i = 1   \ .
  \en
  In the original paper, the above  problem was solved by using standard techniques
  based on Lagrange multipliers. However, an alternative approach is to invoke the 
  constraint and express one of the $\theta_i$'s in term of the others.
  Specifically, we can   define
  \eq{eq:thetal}
  \theta_j = 1 -   \sum_{i=[j-m]}^{j-1} \theta_i . 
  \en
  When this is substituted in the optimization problem~\nref{eq:pul1} we obtain
  the same optimization problem \nref{eq:theta} as for AA. 
  In addition, assume that DIIS is applied to an iteration of the form \nref{eq:mixing2}
  with $\gamma $ replaced by $\beta$. In this case $g(x_i) = x_i + \beta  f(x_i)$ and therefore
  the next iterate defined in \nref{eq:pulay} can be re-written as
  \eq{eq:pulay1}
  x_{j+1} = \sum_{i=1}^j \theta_i ( x_i + \beta  f(x_i))  =
  \sum_{i=[j-m]}^j \theta_i x_i + \beta  \sum_{i=[j-m]}^j \theta_i f(x_i))
  \equiv \bar x_j +  \beta  \bar f_j ,
  \en
  which is identical with Anderson's update in Eq. \nref{eq:and3}.
  Note also that the vectors $f_i$ defined above are now $x_i -g(x_i) = \beta f_i$ and so the
  solution to the problem \nref{eq:pul1} is unchanged.
  So the two schemes
  are equivalent when both are used to accelerate an iteration of the Richardson type 
  \nref{eq:mixing2}. For this reason, the common term ``Anderson-Pulay mixing'' is often
  used to refer to either method.

  Later Pulay published another paper  \cite{Pulay82}  which considered
  improvements to the original scheme discussed above. This improved scheme
  consisted mainly in introducing a new  SCF iteration, i.e., an alternative to
  the function $g$ in our notation, which leads to better convergence. In other
  words the method    is specifically targeted at   SCF iterations.
  Pulay seemed unaware of
  Anderson's work which preceded his by approximately 15 years. Indeed,  neither
  of the two   articles just mentioned cites Anderson's method.

  The article \cite{chupin2021convergence} discusses the convergence of variable
  depth DIIS algorithms and shows that these can lead to superior convergence
  and computationally more effective schemes.
  
\section{Inexact and Quasi Newton approaches}\label{sec:QN} 
Among other technique that  have  been successfully applied to accelerate fixed-point
iterations such as the ones in DFT, are those based on  Quasi-Newton (QN) approaches.
One might argue whether or not it is legitimate to view these as `acceleration techniques'.
If the only restriction we put on acceleration
methods is that they utilize a few of the previous iterates of the sequence, and
that they apply the fixed-point mapping $g$, to one of more vectors, then  inexact
Newton and Quasi-Newton methods satisfy these requirements.

\subsection{Inexact Newton} \label{sec:InexNewt} Many of the approaches
developed for solving \eqref{eq:fx} are derived as approximations to Newton's
approach which is based on the local linear model around some current
approximation $x_j$:
\begin{equation}
\label{eq:linmod}
f(x_j+\Delta x)\approx f(x_j) + J(x_j)\Delta x ,
\end{equation}
where $J(x_j)$ is the Jacobian matrix at $x_j$. Newton's method
determines $\delta = \Delta x_j = x_{j+1} - x_j$ at step $j$, to make the right-hand side
on \eqref{eq:linmod} equal to zero. This is achieved by solving the Newton
linear system $ J(x_j) \delta + f(x_j) = 0.$ Inexact Newton methods, see e.g.,
\cite{Kelley-book,Dembo-al,Brown-Saad} among many references, compute a sequence
of iterates in which the above Newton systems are solved approximately,
typically by an iterative method.  Starting from an initial guess $x_0$, the
iteration proceeds as follows:
\begin{align} 
 &      \mbox{Solve}  & J(x_j) \delta_j   &\approx -f(x_j) &  \label {eq:inex1}\\
 &       \mbox{Set}    & x_{j+1}  & = x_j + \delta_j & \label {eq:inex2}
\end{align}  
The right-hand side of  the Newton  system is $-f(x_j)$ and this is also
the residual for the linear system when $\delta_j = 0$. 
Therefore, in later sections we will define the residual vector  $r_j \equiv -f(x_j)$.

Suppose that we invoke a Krylov subspace method for solving~\nref{eq:inex1}. If we set $J \equiv J(x_j)$ then the method 
will usually generate an approximate solution that can be written in the form
\eq{eq:deltaj}
\delta_j = V_j y_j,
\en
where $V_j$ is an orthonormal basis of the Krylov subspace
\eq{eq:Kry}
\K_j = \Span \{r_j, J r_j, \cdots, J^{m-1} r_j \}.
\en
The vector $y_j$ represents the expression of
the solution in the basis $V_j$.  
For example, if GMRES or, equivalently Generalized Conjugate Residual (GCR)
\cite{Eis-Elm-Sch}, is used, then $y_j$ becomes
$y_j = (J V_j)^\dagger (-f(x_j))$. In essence, the inverse Jacobian is locally
approximated by the rank $m$ matrix:
\eq{eq:ApInv}
B_{j, GMRES} = V_j (J    V_j)^\dagger.
\en
In inexact Newton methods the approximation just defined can be termed `local',
since  it is only used for the $j$-th step: once the solution is
updated, the approximate inverse Jacobian \nref{eq:ApInv} is discarded and
the process will essentially compute a new Krylov subspace and related
approximate Jacobian at the next iteration. This `lack of memory' can be an
impediment to performance. Herein lies an important distinction between these methods
and Quasi-Newton methods whose goal is to compute 
approximations to the Jacobians, or their inverses, by an accumulative
process, which utilizes  the most recent iterates to gradually update these approximations.  
  
\subsection{Quasi-Newton} \label{sec:broyden}
Standard quasi-Newton methods build a  local approximation 
$J_{j}$ to the Jacobian $J(x_{j}) $ progressively by using previous iterates.  These methods require  the relation  \eqref{eq:linmod} to be satisfied by
the updated $J_{j+1}$ which is built at step $j$. 
First  the following {\em secant condition}, is imposed:
\begin{equation}
\label{eq:secant1}
J_{j+1}\Delta x_j = \Delta f_j,
\end{equation}
where $\Delta f_j:=f(x_{j+1})-f(x_j)$. Such a condition will force the
new approximation to be exact on the \emph{pair}
$\Delta x_j,  \Delta f_j$ in the sense that the mapping $J_{j+1}$ must transform
$\Delta x_j$ into $\Delta f_j$ exactly.
A second  common requirement is  the
{\em no-change condition}: 
\begin{equation}
\label{eq:nochange1}
J_{j+1}q =  J_{j}q, \quad \forall q \quad \mbox{such that} \quad 
q^T\Delta x_j=0 . 
\end{equation}
In other words, there should be no new information from $J_j$ to $J_{j+1}$
along any direction $q$ orthogonal to $\Delta x_j$.
Broyden showed that there is a unique  matrix $J_{j+1}$ that satisfies
both  conditions (\ref{eq:secant1}) and (\ref{eq:nochange1}) and it can be
obtained by the update formula: 
\begin{equation} 
J_{j+1} = 
J_j + (\Delta f_{j}-J_j\Delta x_j)\frac{\Delta x_j^T}{\Delta x_j^T\Delta x_j}.
\label{eq:broyden1updatej}
\end{equation}


Broyden's \emph{second method} approximates
the inverse Jacobian directly instead of the Jacobian itself.
If $G_j$  denotes this approximate inverse Jacobian
at the $j$-th iteration, then the secant condition (\ref{eq:secant1})
becomes: 
\begin{equation}
\label{eq:secant2}
G_{j+1}\Delta f_j = \Delta x_j.
\end{equation}
By minimizing $E(G_{j+1})=\|G_{j+1}-G_{j}\|_F^2$
with respect to $G_{j+1}$
subject to (\ref{eq:secant2}),
one  finds this update formula for the inverse Jacobian:
\begin{equation}
\label{eq:broyden2update}
G_{j+1}=G_{j}+(\Delta x_j - G_j 
\Delta f_j)\frac{\Delta f_j^T}{\Delta f_j^T\Delta f_j},
\end{equation}
which is also the only update satisfying both
the secant condition (\ref{eq:secant2})
and the no-change condition for the inverse Jacobian:
\eq{eq:NoCh}
(G_{j+1} -G_{j}) q = 0, \quad \forall \  q \perp \Delta f_j.
\en
AA can be viewed from the angle of \emph{multi-secant} methods, i.e., block
forms of the secant methods just discussed, in which we impose a secant
condition on a whole set of vectors $\Delta x_i, \Delta f_i$ at the same time,
see \ref{sec:AA3}.

\section{A detailed look at   Anderson Acceleration}\label{sec:AA2} 

In the original article  \cite{Anderson65} the author hinted that he was 
influenced by the literature on `extrapolation methods' of the type presented by
Richardson and Shanks and in fact he named his method the ``extrapolation
method'', although more recent terminology reserves this term to a different
class of methods. There were several
rediscoveries of the method, as well as variations in the implementations.  In
this section we will take a deeper look at AA and explore different
implementations, variants of the algorithm, as well as theoretical aspects.
  
  \subsection{Reformulating Anderson's method}\label{sec:reformul}
  We begin by making a small adjustment to the notation in order to rewrite AA
  in a form that resembles that of extrapolation methods.  This variation
  requires a simple change of basis. As was seen in Section~\ref{sec:AA1}
  Anderson used the basis vectors
  \eq{eq:di}
  d_i = f_j - f_i, \quad \text{for}\quad i=[j-m]:j-1
  \en
  to express the vector  added to $f_j$ to obtain $\bar f_j$ as
  shown in \eqref{eq:fbar} which becomes
  \eq{eq:fbar1} \bar f_i = f_j - \sum_{i=[j-m]}^{j-1} \theta_i d_i .
  \en
  We assume  that the $d_i$'s do indeed form a basis. 
  It is common in extrapolation methods to use forward differences, e.g.,
  $\Delta f_{i} = f_{i+1} - f_{i}$ and $\Delta x_{i} = x_{i+1} - x_{i}$.
  These can be exploited to form an  alternate basis consisting of the vectors
  $\Delta f_{i-1} = f_i - f_{i-1}$ for $i = [j-m]: j-1$  instead of the $d_i$'s.
  Note that the simple relations:
  \begin{align}
 \Delta f_{i} &= (f_{i+1} - f_{j}) - (f_{i} - f_{j}) = d_{i} - d_{i+1} , \quad i=[j-m], \cdots, j-1  \label{eq:basch1}\\
    f_j - f_i 
              &= (f_j- f_{j-1}) + (f_{j-1}-f_{j-2}) + \cdots + (f_{i+1}-f_{i})
                \nonumber \\
             &= \Delta f_{j-1} + \Delta f_{j-2} + \cdots + \Delta f_{i} ,
 \quad i=[j-m], \cdots, j-1 , \label{eq:basch2}
 \end{align} 
 allow to switch from one basis representation  to the other.
 The linear independance of one of these two sets of vectors implies the
 linear independance of the other.
 A simliar transformation can   also applied to express the vectors $x_j - x_i$ in
 terms of  $\Delta x_i$'s and vice-versa.
 

  With this new notation at hand, we can rewrite AA as follows.
Starting with an initial $x_0$ and  $x_1 \equiv g(x_0)=x_0+\beta f_0$, where
$\beta >0$ is a parameter, we define blocks of  forward differences
\begin{equation}
\label{eq:dfdx}
\XX_j=[\Delta x_{[j-m]}\;\cdots\;\Delta x_{j-1}] , 
\qquad
\FF_j=[\Delta f_{[j-m]}\;\cdots\;\Delta f_{j-1}] .
\end{equation}
We will define  $m_j=\min\{m,j\}$ which is the number of columns in $\XX_j $ and $\FF_j$.
The least-squares problem \nref{eq:theta} is translated to the new problem:
\eq{eq:gammak}
\gamma\up{j} = \text{argmin}_{\gamma \in \mathbb R^{m_j}}\| f_j - \FF_j \gamma \|_2 .  
\en
from which we get the vectors in Equations~\nref{eq:xbar}, \nref{eq:fbar}, and \nref{eq:and3}:
\begin{align} 
\bar x_j & = x_j-\XX_j \ \gamma\up{j} ,   \label{eq:AA1} \\ 
\bar f_j & = f_j-\FF_j \ \gamma\up{j} , \label{eq:AA1b} \\
x_{j+1} & =\bar x_j + \beta \bar f_j . \label{eq:AA1c} 
\end{align}
We show this as Algorithm~\ref{alg:AAm}.
\begin{algorithm}[H]
  \centering
  \caption{Anderson-Acceleration AA (m)}\label{alg:AAm}
  \begin{algorithmic}[1]
  \State \textbf{Input}: Function $f(x)$, initial  $x_0$. \\
  Set $f_0 \equiv f(x_0)$; \\
  $x_1=x_0+\beta_0f_0$;\\
  $f_1 \equiv f(x_1)$.
\For{$j=1,2,\cdots,$ until convergence} 
\State $\Delta x_{j-1}  = x_j-x_{j-1} \quad \Delta f_{j-1}  = f_j - f_{j-1}$
\State Set $\XX_j = [\Delta x_{[j-m]}, \cdots, \Delta x_{j-1}] ,\qquad  \FF_j = [\Delta f_{[j-m]}, \cdots, \Delta f_{j-1}] $
\State Compute $\gamma\up{j} = \text{argmin}_\gamma \| f_j - \FF_j \gamma \|_2 $
\State Compute $x_{j+1} = (x_j - \XX_j \gamma\up{j} ) + \beta   (f_j - \FF_j \gamma\up{j} ) $ 
\EndFor
\end{algorithmic}
\end{algorithm}

It is clear that the two methods are mathematically equivalent, as they both
express the same underlying problem in two different bases. To better see the
correspondance between the two in the simplest case when $m = \infty$ (in which
case $[j-m] = 0 $) we can expand $f_j - \FF_j \gamma $ and exploit
\nref{eq:basch1}:
\begin{align*} 
  f_j - \FF_j \gamma
  &=  f_j - \sum_{i=0}^{j-1} \gamma_{i+1} (f_{i+1} -f_{i})  \\
  &=  f_j - \sum_{i=0}^{j-1} \gamma_{i+1} [d_i - d_{i+1}]  \qquad (\text{Note:} \ d_j \equiv 0)       \\
  &=  f_j - \sum_{i=0}^{j-1}  (\gamma_{i+1} - \gamma_{i} ) d_{i}   \qquad (\text{with:} \ \gamma_0 \equiv 0 )  .  
\end{align*}
Thus, a comparison with \nref{eq:fbar1} shows that
the optimal $\theta_i$'s in the original Anderson algorithm can be
obtained from the $\gamma_i$'s of the variant just presented by using the relation
$\theta_i = \gamma_{i+1}-\gamma_{i}$ for $i=0, \cdots, j-1$ with the convention
that $\gamma_0 \equiv 0$.
It is also easy to go in the opposite direction and express the $\gamma_i$'s from
the $\theta_i$'s of the original algorithm. In fact, a simple induction argument using
the relations
$\theta_i = \gamma_{i+1}-\gamma_i$ for $i \ge 0$ with $\gamma_0 \equiv 0$, will show that
$\gamma_{i+1} = \sum_{j=0}^i \theta_j $ for $i \ge 0$.

The intermediate solution $\bar x_j$ in \nref{eq:AA1} can be interpreted as
the minimizer of the linear model:
\eq{eq:AALinMod}
f(x_j - \XX_j \gamma) \approx  f(x_j) - \FF_j \gamma
\en
where we use the approximation
$f(x_j - \gamma_i \Delta x_i) \approx f(x_j) - \gamma_i \Delta f_i$.
The method computes the minimizer of the linear model \nref{eq:AALinMod} and the
corresponding optimal $\bar x_j$. The vector $\bar f_j$ is the corresponding
linear residual. Anderson's method obtains this intermediate solution
$\bar x_j$ and proceeds to perform a fixed-point iteration using the
linear model represented by the pair
$\bar x_j, \bar f_j$ as shown in \nref{eq:AA1c}.

From an implementation point of view, to compute the new iterate $x_{j+1} $ as
defined in \nref{eq:AA} we need the pair $x_j, f_j$, the pair of arrays
$\XX_j, \FF_j$, and $\beta$.  We may write this as
$ x_{j+1} = \textbf{AA} (x_j, \beta,\XX_j, \ \FF_j ).  $ The algorithm
would first obtain $\gamma\up{j}$ from solving \nref{eq:gammak}, then compute
$\bar x_j, \bar f_j $ and finally obtain the next iterate $x_{j+1}$ from
\nref{eq:AA1c}.  In the restarted version of the algorithm $\XX_j, \FF_j$ are
essentially reset to empty arrays every, say, $k$ iterations.
The case where all previous vectors are used, called the `full-window AA' (or just `full AA'),
corresponds to setting $m=\infty$ in Algorithm~\ref{alg:AAm}. 
The `finite window AA' or `truncated AA' corresponds
Algorithm~\ref{alg:AAm} where $m$ is finite.
The parameter $m$ is often termed  called
   the `window-size' or `depth' of the algorithm .  

\subsection{Classical Implementations}\label{sec:NEq} 
In the classical implementation of AA the least-squares problem \nref{eq:gammak}
or \nref{eq:theta} are solved via the normal equations:
\eq{eq:NEq} (\FF_j^T \FF_j) \gamma\up{j} = \FF_j^T f_j \en
where it is assumed
that $\FF_j$ has full rank.  When the iterates near convergence, the column
vectors of $\FF_j$ will tend to be close to zero or they may just become
linerarly dependent to within the available working accuracy.  Solving the
normal equations \nref{eq:NEq} in these situations will fail or be subject to
severe numerical issues.  In spite of this, the normal equation approach is
rather common, especially when the window-size $m$ is small.  The ideal solution
to the problem from a numerical point of view is to resort to the Singular Value
Decomposition (SVD) and apply the truncated SVD \cite[sec. 5.5]{GVL-book}.
However, this `gold-standard' approach is expensive and has been avoided by
practitioners.  An alternative is to regularize the least-equares problem, 
replacing \nref{eq:NEq} by \eq{eq:NEqR} (\FF_j^T \FF_j +\tau I) \gamma\up{j} =
\FF_j^T f_j \en where $\tau $ is a regularization parameter.  This compromise
works reasonably well in practice and has the advantage of being inexpensive
in terms of arithmetic and memory.
Note that the memory cost of an approach based on normal  equations is modest,
requiring  mainly to keep  the sets $\FF_j, \XX_j$,  i.e.,  a total of
$2m $ vectors of length $n$. 
 
\subsection{Implementation with ``Downdating QR''}\label{sec:ddqr} 
A  numerically effective alternative  to the normal equations is based on exploiting the QR factorization.
Here we will focus on Problem~\nref{eq:gammak} of the formulation of AA discussed above. 
In principle the idea of using QR is straightforward.
Start with a QR factorization of $\FF_j$,
which we write as $\FF_j = Q R$  with $ Q \in \ \RR^{n \times m_j}$
where we recall the notation $m_j = \min \{ m, j\}$ and
$ R \in \ \RR^{m_j \times m_j}$. Then obtain $\gamma\up{j}$  by simply solving the
triangular system 
$R \gamma\up{j} =  Q^T f_j$. Such a trivial implementation runs into a practical
difficulty in that recomputing the QR factorization of $\FF_j$ at each step is
wasteful as it ignores the evolving nature of the block $\FF_j$.

Assume at first that $j \le  m-1$.
Then at the end of the $j$-th step represented in the main loop
of Algorithm~\ref{alg:AAm}, one column is added to
$\FF_{j}$ (which has $j$ columns) to obtain $\FF_{j+1}$ (which has $j+1 \le m$ columns).
If $\FF_{j}  = Q  R $ and the new column $v \equiv \Delta f_j$ is added
then we just need to carry out one additional step of the Gram-Schmidt process
whereby this $v$ is orthonormalized against the existing columns of $Q$, i.e., we write
$\hat q = v - Q  h $ where $h = Q^T v$ and then $q_{j+1} = \hat q / \rho $
where $\rho = \| \hat q \|_2 $.
Hence, the new factorization is
\eq{eq:NewQR}
 \FF_{j+1} = [\FF_{j}, v] = [Q ,  q_{j+1}] \begin{pmatrix} R  & h \\ 0 & \rho \end{pmatrix}
  \equiv \tilde Q \tilde R .
\en


Assume now that $j>m-1$. Then, to form $\FF_{j+1}$ in the next step, a column $v$
is added to $\FF_{j}$ while the oldest one, $\Delta f_{m_j}$, must be discarded
to create room for the new vector. Recall that
the number of columns must stay  $ \le m$.  This calls for a different strategy
based on `downdating' the QR factorization, i.e., updating a QR factorization
when a column of the original matrix is deleted.

\typeout{========================================Issue with indexing Q. Start at 0?}
\typeout{========================================INDEXING in aa-tgs to fix} 

To simplify notation in describing the process, we will remove indices and write
our current $\FF_j$ simply as $F = [v_1,v_2,\cdots,v_m]$ ($m$ columns) and
assume that $F$ was previously factorized as $F=QR$.  The aim is to build the
QR factorization of $[v_2, v_3, \cdots, v_m]$ from the existing factors $Q$, $R$
of $F$. Once this is done we will be in the same situation as the one where
$j \le m-1$ which was treated above, and we can finish the process in the same way
as will be seen shortly. We write
$ F = [v_1, v_2, \cdots, v_m]$, $ R = [ r_1, r_2, \cdots, r_m] $ and denote the
same matrices with their 1st columns deleted by: \eq{eq:QRknot1} F\dwn{-} =
[v_2, v_3, \cdots, v_m] , \qquad H = [ r_2, r_3, \cdots, r_m] .  \en The matrix
$H $ is of size $m \times (m-1)$ and has an upper Hessenberg form, i.e.,
$h_{ij} = 0 $ for $i>j+1$.  The matrix $F\dwn{-} $, satisfies \eq{eq:FQH}
F\dwn{-} = Q H \en and our goal is to obtain a QR factorization of $F\dwn{-}$
from this relation.  For this we need to transform $H $ into upper triangular
form with the help of Givens rotation matrices \cite{GVL-book} as is often done.  Givens
rotations are matrices of the form:
\[
G(i,k,\theta) = \begin{pmatrix}
1       &  \ldots &     0       &      & \ldots  & 0         &   0     \cr
\vdots  & \ddots  &     \vdots  & \vdots & \vdots  & \vdots  &\vdots  \cr  
   0    & \cdots  &     c       & \cdots &  s      & \cdots  & 0      \cr  
\vdots  & \cdots  &     \vdots  & \ddots & \vdots  & \vdots  &\vdots  \cr  
   0    & \cdots  &     -s      & \cdots &  c      & \cdots  & 0      \cr  
 \vdots & \cdots  &     \vdots  & \cdots & \vdots  & \cdots  & \vdots \cr  
 0    &  \ldots &     0       &        & \ldots  &         &  1
\end{pmatrix} 
 \begin{array}{l} 
  \mbox{ \ } \\ \mbox{ \ } \\ i \\\mbox{ \ }  \\ k \\ \mbox{ \ } 
 \\ \mbox{ \ } 
 \end{array} \]
with $ c = \cos \theta$ and $s = \sin \theta$.
This matrix represents a rotation  of angle $-\theta$ in the span of $e_i$ and $e_k$.
A rotation of this type is usually applied to transform some entry in the $k$-th row of
a matrix  $A$ to zero  by an appropriate choice of $\theta$.
If $H  = \{ h_{ij} \}_{i=1:m,j=1:m-1}$ we can find  a rotation
$G_1 =G(1,2,\theta_1)$ so that
$(G_1 H)_{2,1} = 0$. The values of $c$ and $s$ for this 1st rotation are
\[
  c = \frac{h_{11}}{\sqrt{h_{11}^2 + h_{12}^2}}, \quad
  s = \frac{h_{12}}{\sqrt{h_{11}^2 + h_{12}^2}}.
  \]
This is then followed by applying a rotation
$G_2=G(2,3,\theta_2)$ so that $(G_2 (G_1 H))_{3,2} = 0$, etc.
Let $\Omega $ be the composition of these rotation matrices:
\[
  \Omega = G_{m-1} G_{m-2} \cdots G_{2} G_{1} . 
\]
The process just described transforms $H$ to an $m\times (m-1)$ upper triangular
matrix  with a zero last row:
\[
  \Omega H  = \begin{pmatrix}  \hat R \\ 0 \end{pmatrix} . 
\]
  Defining $\hat Q \equiv Q  \Omega^T \equiv [\hat q_1, \hat q_2, \cdots, \hat q_m]$
and $\hat Q\up{-} \equiv [\hat q_1, \cdots, \hat q_{m-1}]$, we see that: 
\eq{eq:DDqr}
F\dwn{-} = Q  H  = (Q  \Omega^T) \ (\Omega H )  = \hat Q  \times
\begin{pmatrix} \hat R \\ 0 \end{pmatrix} \  = \hat Q\up{-} \hat R , 
\en
which is the factorization we were seeking.  Once the factorization in
\nref{eq:DDqr} is available, we can proceed just as was done in the full window
case ($j\le m-1$) seen earlier to add a new vector $v$ to the subspace, by
updating the QR factorization of $[F\dwn{-},v]$ from that of $F\dwn{-}$, see
equation \nref{eq:NewQR}.  Computing the downdated QR factors in this way is
much less expensive than recomputing a new QR factorization.

This procedure yields a pair of matrices $Q_j, R_j$ such that
$\FF_j = Q_j R_j$ at each step $j$.
  The
  solution to the least-squares problem in Line~8 is
  $\gamma\up{j} = R_j\inv \eta\up{j}$ where $\eta\up{j} = Q_j^T f_j$.  Then,
  clearly in Line~9 of Algorithm~\ref{alg:AAm}, we have
  \eq{eq:8y}
  f_j - \FF_j \gamma\up{j} = f_j - (Q_j R_j) \gamma\up{j}
  = f_j - (Q_j R_j) R_j\inv \eta\up{j} = f_j -  Q_j\eta\up{j} , 
  \en 
  and so Lines~8-9 need to be replaced by:
  \begin{align}
    (8a) & & \eta\up{j} &= Q_j^T f_j ; \quad \gamma\up{j}  = R_j\inv \eta\up{j}
             \label{eq:8a}\\
    (9a) & & x_{j+1} &= (x_j - \XX_j \gamma\up{j}) + \beta (f_j - Q_j \eta\up{j} ) . 
    \label{eq:9a}
  \end{align}

It is clear that the matrix $Q_j$ can be
computed `in-place' in the sense that the new $Q_j$ can be computed and
overwritten on the old one without requiring additional storage.
Also, it is no longer necessary to store $\FF_j$ in Line~7. Instead,
we will perform a standard factorization where the vector
$v = \Delta f_{j-1}$ is orthonormalized against those previous $q_i$'s that are saved.
When $j > m-1$ a downdating QR factorization is performed to prepare the
matrices $\hat Q\up{-}, \hat R$ from which the updated QR is performed.


We will refer to this algorithm as \emph{AA-QR} if there is need to specifically
emphasize this particular implementation of AA.  It will lead us to an
interesting variant of AA to be discussed in Section~\ref{sec:TGS}.

  \subsection{Simple downdating}\label{sec:simpleDD}
  We now briefly discuss an alternative method for downdating the
  QR factorization. Returning to  \nref{eq:FQH} we split the Hessenberg matrix $H$ into
  its $(m-1)\times (m-1)$ lower block which is triangular and which we denote by $R$ and
  its first row which we denote by $h_1^T$. With this 
  \nref{eq:FQH} can be  rewritten as follows:
  \begin{align}
    F\dwn{-} &=  q_1 h_1^T   + [q_2, q_2, \cdots, q_m] R  \nonumber \\
             & = [Q_{(-)}  + q_1 h_1^T R^{-1} ] R \nonumber  \\
             & \equiv [Q_{(-)}  + q_1 s^T  ] R  \label{eq:qminr}
  \end{align}
  where we have set $s^T \equiv h_1^T R^{-1} $.
  From here, there are two ways of proceeding.  We can either get a QR factorization of
  $Q_{(-)} + q_1 s^T $ which can then be retrofitted into \nref{eq:qminr} or we
  can provide some other orthogonal factorization that can be used to solve the
  least-squares problem effectively.
  
  Consider the first approach.
  One way to get the QR factorization of $Q_{(-)}  + q_1 s^T $ is via the Cholesky
  factorization of the matrix $ [Q_{(-)}  + q_1 s^T  ]^T [Q_{(-)}  + q_1 s^T  ] $.
  Observing that the columns of $Q_{(-)} $ are orthonormal and orthogonal to $q_1$, we see that
  \eq{eq:qrfromchol}
    [Q_{(-)}  + q_1 s^T  ]^T [Q_{(-)}  + q_1 s^T  ] = I + s s^T . 
  \en
  As it turns out matrices of the form $I+ss^T$ lead to an inexpensive Cholesky
  factorization due to a nice structure that can be
  exploited  \footnote{In a nutshell, the entries below the main diagonal of the 
    $j$-th column of the $L$ matrix in the LDLT decomposition of the matrix $I + s s^T$
    are a constant times the entries $j+1:n$ of $s$. This observation leads to
    a Cholesky factorization that is inexpensive and easy to compute and to exploit.}.
  From this factorization,
  say,   $I + s s^T = G G^T$ where $G$ is lower triangular we get the updated factorization:
   \begin{align} 
   F\dwn{-} & =  (Q_{(-)}  + q_1 s^T ) R \label{eq:supd0} \\
  &=  \underbrace{ [(Q_{(-)}  + q_1 s^T)G^{-T}  ]}_Q \ \underbrace{[G^T R ]}_R  \equiv Q R .  \label{eq:supd1} 
  \end{align} 
  The Q-factor defined above
  can be verified to be indeed an orthogonal matrix while the R factor is clearly
  upper triangular.  Implementation and other details are omitted.

  The second approach is a simplfied, possibly more appealing, procedure. Here
  we no longer rely on a formal QR factorization, but a factorization that
  nonetheless exploits an orthogonal factor.
  Starting from \nref{eq:supd0}  we now write:
  \eq{eq:supd2} F\dwn{-} =
  \underbrace{ [(Q_{(-)} + q_1 s^T) (I+ss^T)^{-1/2}
    ]}_Q \ \underbrace{(I+ss^T)^{1/2} R ]}_S \equiv Q S .
  \en
  Note in passing that the product $ Q (I+ss^T)^{1/2} $ is just the
  \emph{polar decomposition} \cite{GVL-book} of the matrix
  $Q_{(-)} + q_1 s^T $.  Here too the structure of $I+ss^T$ leads to
  simplifications, in this case for the terms in the fractional powers of
  $I+ss^T$. Indeed, if we set $\lambda = 1 + s^T s$, then it can be shown that
  \begin{align}
    (I+ss^T)^{1/2}  & = I + \alpha  s s^T & \mbox{with}   & & \alpha &= 
 \frac{1}{1 + \sqrt{\lambda }}
  \label{eq:prws1} \\
    (I+ss^T)^{-1/2}  & = I - \beta s s^T  & \mbox{with}   & & \beta & =
      \frac{1}{ \lambda + \sqrt{\lambda}} 
  \label{eq:prws2} .
  \end{align}
  Noting that $Q^T Q = I$,
  we wind up with a factorization of the form $F_{(-)}=Q S$ where $Q$ has orthonormal columns
  but  $S=(I+s s^T)^{1/2} R $ is not triangular.
  It is possible to implement this by keeping the matrix $S$ as a square matrix or
  in factored form in order to exploit   its inverse which is  
  $S\inv =R\inv (I-\beta s s^T) $ where $\beta $ is given in \nref{eq:prws2}.
  The next step, adding a vector to the system, can be processed as in a standard
QR factorization, see eq. \nref{eq:NewQR} where the block $R$ is to be replaced by $S$.

These two simple alternatives to the Givens-based downdating QR do not seem to
have been considered in the literature. Their main merit is that they focus 
more explicitly on the consequence of deleting a column and show what remedies can be
applied. The deletion of a column leads to the QR-like
factorization~\nref{eq:supd0} of $F_{(-)}$. This resembles a QR factorization
but the Q-part, namely, $Q_{(-)} + q_1 s^T $ is not orthogonal.  The standard
remedy is to proceed with a downdating QR factorization which obtains QR factors
from \nref{eq:FQH} instead of \nref{eq:qminr}.  Instead, the remedies outlined above
proceed from \nref{eq:qminr} from which either the QR factorization
of this matrix or its polar decomposition are derived. In both cases, the results are then
retrofitted into \nref{eq:supd0} to obtain a corrected decomposition, with an
orthogonal Q factor.

\subsection{The Truncated Gram-Schmidt (TGS) variant}\label{sec:TGS}
There is an appealing alternative to AA-QR worth considering: bypass the
downdating QR, and just use the orthogonal factor $Q_j$ obtained from a
Truncated Gram-Schmidt (TGS) orthogonalization. This means that  a new vector $q_j$
is obtained by orthonormalizing $\Delta f_{j-1}$ against the previous
$q_i'$s that are saved, and when we 
delete a column we just omit the adjustment consisting of the downdating QR process.
Without this adjustment the resulting
factors $Q_j, R_j$ are no longer a QR factorization of $\XX_j$
when a truncation has taken place, i.e., when $j>m$.
\typeout{=========== REDO NOTATOION $m_j$ --> } 
The solution to the least-squares problem  now uses $Q_j$
instead of $\FF_j$, i.e., we minimize $\| f_j - Q_j \gamma \|_2$, ending with 
$ 
\gamma\up{j} = Q_j^T f_j . 
$
In the limited window case, the resulting least-squares approximation to $f_j$, i.e.,
$Q_j \gamma\up{j}$, is no longer equal to the one resulting from  the Downdating QR
approach which solves the original LS problem \nref{eq:theta} with the set
$\FF_j$ by exploiting its QR factorization. In Anderson Acceleration with
Truncated Gram-Schmidt (AA-TGS), \emph{the resulting matrix $Q_j$ is not required to be the
Q-factor of $\FF_j$ and so the range of $\FF_j $ is different from that of
$ Q_j$}. Therefore  the method is \emph{not equivalent} to AA in the finite window case.

We can write the orthogonalization process at the $j$-th step as \eq{eq:recTGS1}
q_j = \frac{1}{s_{jj}} \left[\Delta f_{j-1} - \sum_{i=[j-m+1]}^{j-1} s_{ij}
  q_i\right] \ .  \en Here the scalars $s_{ij}, i=[j-m+1],\cdots,j-1$ are those
utilized in a modified Gram-Schmidt process, and $s_{jj}$ is a normalizing
factor so that $\| q_{j} \|_2 = 1$.

With  $m_j \equiv \min \{ m,j+1\}$ define  the $m_j\times m_j$ upper triangular matrix
$S_j = \{s_{ik}\}_{i={[j-m+1]}:j,k={[j-m+1]}:j}$ resulting from the
orthogonalization process in \nref{eq:recTGS1}.
Having selected what set of columns to use in place of $\FF_j$,
the question now is how do we
compute the solution $x_{j+1}$, i.e., how do we change Line~9 of Algorithm~\ref{alg:AAm}?
One is tempted to take equations (\ref{eq:8a}--\ref{eq:9a}) of AA-QR as a model where
the matrix $R_j$ in \nref{eq:8a} is replaced by
\footnote{Recall that  for convenience the indexing of the columns in $Q$ and other arrays
 in Section~\ref{sec:ddqr} starts at 1 instead of 0.}  $S_j$.
However, the relation $\FF_j = Q_j R_j$ is only valid in the full-window case,
so the relation 
$f_j - \FF_j R_j\inv \eta\up{j} = f_j -  Q_j\eta\up{j} $
in  \nref{eq:8y} no longer holds, and so we cannot write
$\bar x_j$ in the form $\bar x_j = x_j -  \XX_j R_j \inv \eta\up{j}$.
The solution is to make use of the basis $U_j$  that is defined from the
set of $\Delta x_i$'s in the same way that $Q_j$ is defined from the $\Delta f_i$'s.
Thus, we compute the $j$-th column of $U_j$ by the same
process we applied to obtain $q_j$ namely:
\eq{eq:recTGS2}
u_j = \frac{1}{s_{jj}} \left[\Delta x_{j-1} - \sum_{i=[j-m+1]}^{j-1} s_{ij} u_i\right] ,
\en
where the scalars $s_{ij}$ are the same as those utilized to obtain $q_j$ in \nref{eq:recTGS1}.
With this, the relations (\ref{eq:AA1}--\ref{eq:AA1c}) are replaced by 
\eq{eq:AA2N} \bar x_j = x_j - U_j \eta\up{j} ,
\quad  \bar f_j = f_j - Q_j \eta\up{j} , \quad
x_{j+1} = \bar x_j + \beta \bar f_j . 
\en

The procedure  is sketched as Algorithm~\ref{alg:TGS}.
Let us examine the algorithm and compare it with the downdating QR
version seen in Section~\ref{sec:ddqr}.
When $j  \le  m  $ (full window case), the $i$ loop starting in Line~6, begins at 
$i=0$, and the block in Lines 6--12
essentially performs a Gram-Schmidt QR factorization of the matrix $\FF_j$, while
also enforcing identical operations on the set $\XX_j$.  A result
of the algorithm is that 
\eq{eq:QRrel} \FF_j = Q_j S_j ; \qquad \XX_j = U_j S_j.
\en
The above relations are not valid when $j>m$.
In particular, the subspace spanned by $Q_j$ is not the span of 
$\FF_j$ anymore. We saw how to deal with this
issue in Section~\ref{sec:ddqr} in order to recover a QR factorization for
$\FF_j$ from a QR downdating process. AA-TGS provides an alternative solution by relaxing the
requirement of having to use a QR factorization of $\XX_j$.


\begin{algorithm}[H]
  \centering
  \caption{AA-TGS(m)}\label{alg:TGS}
    \begin{algorithmic}[1]
  \State \textbf{Input}: Function $f(x)$, initial guess $x_0$, window size $m$ \\
  Set $f_0 \equiv f(x_0)$, \quad  $x_1=x_0+\beta_0f_0$, \quad 
  $f_1 \equiv f(x_1)$
\For{$j=1,2,\cdots,$ until convergence} 
\State $u := \Delta x = x_j-x_{j-1}$
\State $q := \Delta f = f_j - f_{j-1}$
\For{{$i=[j-m+1],\ldots, j-1$}} 
\State $s_{ij} :=  (q, q_i) $
\State $u := u - s_{ij} u_i$
\State $q := q - s_{ij} q_i$
\EndFor
\State $s_{jj}=\Vert q \Vert_2$
\State $q_{j} := q /s_{jj}$, \quad  $u_{j} := u /s_{jj}$\
\State {Set ${Q}_j = [q_{[j-m+1]}, \ldots, q_{j}],\quad {U}_j = [u_{[j-m+1]}, \ldots, u_{j}]$ } 
\State Compute $\eta\up{j} = {Q}^{\top}_jf_j$
\State $x_{j+1} = (x_{j}-{U}_j\eta\up{j})+\beta_{j} (f_j- {Q}_j\eta\up{j})$
\State $f_{j+1} = f(x_{j+1})$
\EndFor
\end{algorithmic}
\end{algorithm}

To better understand the process, we  examine what happens specifically when $j=m+1$, focussing on  the set of
$q_i$'s. We adopt the same lightened notation as in Section~\ref{sec:ddqr},
and in particular the indexing  in arrays  $Q$ and $S$ start at 1 instead of zero.
Before the orthogonalization step we have the QR factorization
$\FF_m = Q_m S_m$. Dropping the oldest (first) column is captured by equations 
\nref{eq:QRknot1} and \nref{eq:FQH}. We rewrite \nref{eq:FQH} as follows:
\[
  F_{(-)} = Q H = [q_1, q_2, ..., q_m] H
  = q_1 h_1^T + [q_2, q_3, \cdots, q_m] S_{(-)} . 
\]
As before $H$ is $m\times (m-1)$ Hessenberg matrix obtained from the upper
triangular matrix $S_m$ by deleting its first column.  The row vector $h_1^T$ is
the first row of $H$ (1st row of $S_m$, omitting its first entry).  The matrix
$ S_{(-)} $ is the $(m-1) \times (m-1) $ upper triangular matrix obtained from
$S_m$ by deleting its first row and its first column.  Thus, if we let
$Q_{(-)} = [q_2, q_3, \cdots, q_m]$ as before, we obtain a reduced QR
factorization but it is for a different matrix, i.e.,
\eq{eq:Fmm} F_{(-)} - q_1
h_1^T = Q_{(-)} S_{(-)} .  \en After the truncation, the new vector
$v_{m+1} = \Delta f_{m}$ is orthonormalized against $q_2, q_3, \cdots, q_m$,
leading to the next vector $q_{m+1}$.  Then we can write:
  \begin{align} 
        [F_{(-)} - q_1 h_1^T,  v_{m+1}]
    &= [Q_{(-)},q_{m+1}] \times
      \begin{bmatrix} S_{(-)} & s_{2:m,m+1} \\ 0 & s_{m+1,m+1} \end{bmatrix} \nonumber \\
    & \equiv Q_{m+1} S_{m+1} . \label{eq:Fmm2}
  \end{align} 
  Therefore, at step $j=m+1$ the pair of matrices $Q_{m+1}, S_{m+1}$
  \emph{produce a QR factorization of
    a rank-one perturbation of the matrix   $\FF_{m+1}$. }

  
Herein lies the only difference between the two methods: the Downdating QR
enforces the relation $\FF_j = Q_j R_j$ by correcting the relation~\nref{eq:Fmm} into a valid
factorization of $F_{(-)}$ before proceeding. We saw in Sections~\ref{sec:ddqr} and \ref{sec:simpleDD} how this
can be done. This ensures that we obtain the same solution as with AA. 
In contrast, AA-TGS simplifies the process by not insisting on having
a QR factorization of $\FF_j$. Instead, it exploits a QR factorization of a 
modified version of $\FF_j$, see \nref{eq:Fmm2}  for the case $j=m+1$.
Note that when $j>m+1$ the rank-1 modification on the left-hand side of \nref{eq:Fmm2} becomes
a sum  of rank-one matrices.

Let us now consider the full window case, i.e., the situation $j \le m$.
It is easy to see that in this case the subspaces spanned by $\FF_j$ and $Q_j$ are
identical and in this situation the iterates $x_{j+1}$ resulting from AA and
AA-TGS will be the same. In particular, when $m = \infty$ this will always be the
case. We will state this mathematical equivalence of the two algorithms in the 
 following proposition.
 \begin{proposition}\label{prop:infcase}
   Assuming that they start from the same initial guess $x_0$,  AA-TGS ($\infty$) and AA$(\infty)$
   return the same   iterates at each iteration, in exact arithmetic.
  In addition, they also break down under the same condition.
\end{proposition}
The proof is straightforward and relies on the equality
$\Span \{Q_j \} = \Span \{\XX_j\}$ for all $j$. Though rather trivial this
property is worth stating explicitly because it will help us simplify our
analysis of AA in the full-window case.  It is clear that we can also state a
more general result for the restarted versions of the algorithms
\footnote{Restarting can be implemented for any of the algorithms seen in this
  article. By restarting we mean that every $k$ iterations, the algorithm starts
  anew with $x_0$ replaced by the most recent approximation computed.  The
  parameter $k$ is called the `restart dimension', or `period'. Other restarting
  strategies are not periodic and restart instead when deemeed
  necessary by the numerical behavior of the iterates.}.

We end this section with an important addition to the AA-TGS algorithm whose
goal is to circumvent some numerical stability issues.  The two recurrences
induced by Equations~(\ref{eq:recTGS1}--\ref{eq:recTGS2}) are linear recurrences
that can lead to instability.  A mechanism must be added to monitor the behavior
of the above sequence with the help of a scalar sequence whose numerical
behavior imitates that of the vector sequences. This will help prevent
excessive growth of rounding errors by restarting the process when deemed
necessary. A process of this type was developed in \cite{AATGS}.  Readers
are referred to the article for details.


\subsection{Numerical Illustration}\label{sec:numer1}
We will illustrate a few of the methods seen in this paper so far with two
examples. The first example is usually  viewed as an easy problem to solve because
 its level of nonlinearity can be characterized as mild. The second is an optimization
 problem in molecular physics with  highly nonlinear  coefficients.

\subsubsection{The Bratu problem}
The Bratu problem appears in modeling  
thermal combustion, radiative heat transfer, thermal reaction,
among other applications, see, e.g., \cite{Mohsen14,jacobsen02} for references.
It consists in the following   nonlinear elliptic Partial Differential Equation (PDE)
with Dirichlet boundary conditions on an open domain $\Omega$:
{\begin{align*}
   \Delta u + \lambda e^ u &= 0  \ \ \text{in} \ \ \Omega \\ 
   u(x,y)  &= 0  \ \text{for} \  (x,y) \in \partial \Omega  . 
 \end{align*}
 Here $\lambda$ is a parameter and there is a solution only for values of
 $\lambda $ in a certain interval.
 We set $\lambda $ to the value $\lambda = 0.5$ and define the domain $\Omega$ to
 be the square $\Omega = (0, 1) \times (0, 1)$. Discretization with
 centered finite differences  using 100 interior points in each direction results
 in a system of nonlinear equations $f(x) = 0$ where $f$ is a mapping from
 $\mathbb{R}^{n}$ to itself, with $n = 10,000$.

  The first step in applying acceleration techniques to solve the problem is to
  formulate an equivalent fixed-point iteration of the form
\eq{eq:gmux} g(x) = x - \mu f(x) . 
\en
The reader may have noted the negatve sign used for $\mu$ instead of the
positive sign seen in earlier formulas for the mixing scheme \nref{eq:mixing2}.
In earlier notation $f(x)$ represented the `residual', i.e., typically the
negative of the gradient of some function $\phi$, instead of the
gradient as is the case here. 
What value of $\mu$ should we use?
Noting that  the Jacobian of
$g$ is
\[ \frac{\partial g }{\partial x}  (x)
  =  I -\mu \frac{\partial f }{\partial x} (x)\]
a rule of thumb, admitedly a vague one, is that a small value is needed
only when we expect the Jacobian of $f$ to have large values.
For all experiments dealing with the Bratu problem, we will set $\mu = 0.1$.

\subsubsection{Molecular optimization with Lennard-Jones potential}
The goal of geometry optimization is to find atom positions that minimize total
potential energy as described by a certain potential.  The Lennard-Jones (LJ) 
potential has a long history and is commonly used is computational chemistry,
see \cite[p. 61]{Kittel-book}. Its aim is to  represent both attration and repulsion
forces between atoms by the inclusion of two
terms with opposite signs:
\begin{equation}
        E = \sum_{i=1}^{N} \sum_{j=1}^{i-1} 4 \times
    \left[\frac{1}{\|x_i - x_j\|_2^{12}} - \frac{1}{\|x_i - x_j\|_2^{6}} \right] .
\end{equation}
Each $x_i$ is a 3-dimensional vector whose components are the coordinates of the
location of atom $i$. A common problem is to start with a certain configuration
and then optimize the geometry by minimizing the potential starting from that
position. Note that the resulting position is not a
global optimum but a local minimum around some initial
configuration. In this particular example, we simulate an Argon cluster by
taking the initial position of the atoms to consist of a perturbed initial
Face-Cented-Cubic structure \cite{meyer1964new}.  We took 3 cells per direction,
resulting in 27 unit cells. FCC cells include 4 atoms each and so we end up
with a total of 108 atoms.  The problem can be challenging  due to the high powers in
the potential.

Instead of a nonlinear system of equations as was the case for the Bratu problem,
we now need to minimize $E(\{x_i\})$.  The gradient of $E$ with respect to atom
positions can be readily computed.  If we denote by $x$ the vector that
concatenates the coordinates of all atoms, we can call $f(x) $ this gradient:
$f(x) \ = \ \nabla E(x)$.  The associated fixed-point mapping is again of the
form \nref{eq:gmux} but this time we will need to take a much smaller value for $\mu$,
namely 
$\mu = 0.0001$.  Larger values of $\mu$ often result in unstable iterates,
overflow, or convergence to a non-optimal configuration.  Note that 
we are seeking a local minimum and as such we do need to
verify for each run that the
scheme being tested converges to the correct optimal configuration, in this case
a configuration that thas the potential $E_{opt} = -579.4639..$


  \subsubsection{Gradient Descent, RRE,  Anderson, and Anderson-TGS}
  We will illustrate four  algorithms for the two problems discussed above.  The
  first is a simple adaptive gradient descent algorithm of the form
  $x_{j+1} = x_j - \mu f(x_j$, where $\mu$ is set adaptively by a very simple
  scheme: if the norm of $f(x_j)$ increases multiply $\mu$ by $.3$ and if it decreases
  multiply $\mu$ by $1.05$. The initial $\mu$ is as defined earlier: $\mu = 0.1$ for Bratu and
  $\mu = 0.0001$ for LJ. We will call this scheme \texttt{adaptGD}.  The second scheme
  tested is a restarted version of the RRE algorithm seen in
  Section~\ref{sec:RRE}. If $m$ is the restart dimension then the scheme
  computes an accelerated solution $y_m$ using $x_0, \cdots, x_{m+1}$ and then
  sets $x_{0} $ to be equal to $y_m$ and the algorithm is continued from this
  $x_0$.  It is interesting to note that we use $x_0, \cdots, x_m, x_{m+1}$ to
  compute $y_m$ which is an update to $x_m$, not $x_{m+1}$, see
  \nref{eq:rre2}. Thus, in effect $x_{m+1}$ is only used to obtain the optimal
  $\gamma$ in \nref{eq:rre2}.  Anderson acceleration takes care of this by the
  extra approximate fixed-point step \nref{eq:AA1c}. We can perform a similar
  step for RRE, and it will translate to \eq{eq:rreMod} x_{m+1} = y_m + \beta
  \bar f_m \en where $\bar f_m = \Delta x_m + \Delta X_0\gamma$ is the linear
  residual.  This modificiation is implemented with $\beta = 1 $ in the
  experiments that follow. In addition
  to the baseline Adaptive GD, we test RRE(3), RRE(5), and Anderson(5,10) for
  both problems. We also tested Anderson-TGS(5,.) with the automatic restarting
  strategy briefly mentioned at the end of Section~\ref{sec:TGS} and
  described in detail in \cite{AATGS}.  We should point out that RRE(5),
  Anderson(5,10), and Anderson-TGS(5,.) all use roughly the same amount of
  memory.
  
  \begin{figure}[H]
  \includegraphics[width=0.49\textwidth]{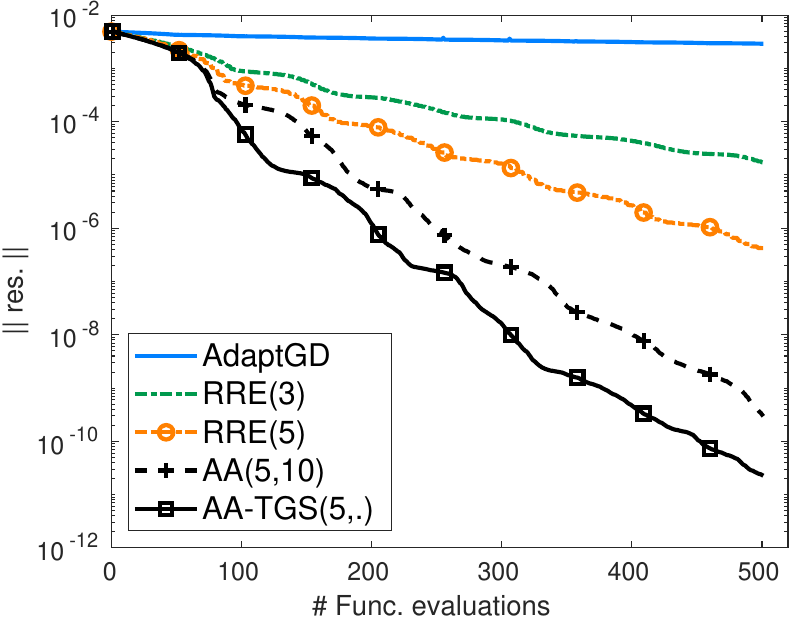}
  \includegraphics[width=0.51\textwidth]{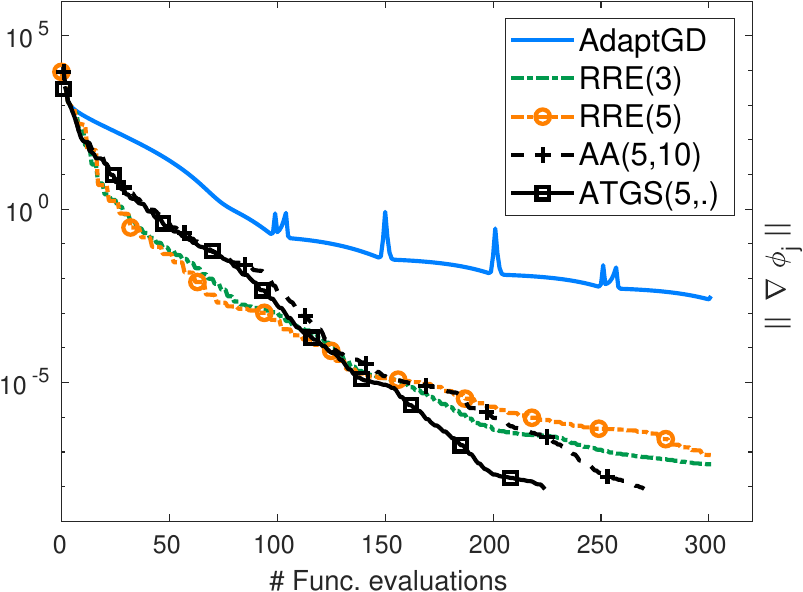}
  \caption{Comparison of 5 different methods
    for the Bratu problem (left) and the Lennard-Jones
    optimization problem (right).\label{fig:Exp1}}
  \end{figure}

  Figure~\ref{fig:Exp1} shows the results obtained by these methods for both
  problems.  Here AA(5,10) stands for AA with a window size of 5, and a restart
  dimension of $k=10$.  RRE($m$) is the restarted RRE procedure described above
  with a restart dimension of $m$, where $m=3$ in the first experiment with RRE
  and $m=5$ in the second.  The plots show the norm of the residual (for Bratu)
  and the norm of the gradient (for LJ) versus the number of function
  evaluations.  All methods start from the same random initial guess and are
  stopped when either the residual norm decreases by $tol = 10^{-12}$ or the
  number of function evaluations exceeds a certain maximum (500 for the Bratu
  problem, 300 for Lennard-Jones). As can be been, for the Bratu problem,
  Anderson and its TGS variant both yield a good improvement relative to the
  simpler RRE schemes.  In contrast, for the Lennard Jones problem, the
  performance of AA(5,10) is close to that of RRE. Surprisingly, RRE(3) which
  uses much less memory does quite well for this example. In both test problems,
  AA-TGS outperforms the standard Anderson algorithm by a moderate margin.  The
  standard Anderson scheme discussed in these experiments is based on the AA-QR
  (AA with QR downdating) implementation discussed in \ref{sec:ddqr}.

\subsection{Theory}
Around the year 2009, Anderson Acceleration started to be noticed
in the linear algebra community, no doubt owing to its simplicity and its success in dealing
with a wide range of problems.  The article~\cite{FangSaad07} discussed \emph{multisecant
  methods}, first introduced as `Generalized Broyden methods' in
article~\cite{vl:energy84}. These methods can be
called block secant methods in which rank-1 updates of the form
\nref{eq:broyden1updatej} or \nref{eq:broyden2update} are replaced by
rank-$m$ updates. As it turns out, AA is a multisecant method and this
specific link was first unraveled   
in~\cite{eyert:acceleration96} and discussed further in
~\cite{FangSaad07}.  A number of other results were subsequently shown.
Among these is an equivalence between Krylov methods  and Anderson in the linear case as well
as convergence studies for nonlinear sequences.    
This goal of this section is to  summarize these results.

    \subsubsection{AA as a  multi-secant approach}\label{sec:AA3} 

    Acceleration methods, such as AA, do not aim at solving a system like
    \nref{eq:fx} directly. As was seen in Section~\ref{sec:AA1} their goal is to
    accelerate a given fixed-point iteration of the form \nref{eq:fixp}. The
    method implicitly expresses an approximation to the Jacobian via a secant
    relation which puts $\FF_j$ in correspondence with $\XX_j$. Roughly
    speaking, AA develops some approximation to a Jacobian $J$ that satisfies a
    secant condition of the form $\FF_j \approx J \XX_j$.  The classic text by
    \cite[pp. 204-205]{ort}  already mentions AA as a form of quasi-Newton    
    approach. At the time, however, Ortega and Rheinbolt
    viewed the method  negatively,
    possibly due to its potential for numerically unstable behavior.

    As seen in Section~\ref{sec:QN}
 Broyden-type methods replace Newton's iteration:
$x_{j+1} = x_j - J(x_j) \inv f_j$ with something like $ x_{j+1} = x_j - G_j f_j $
where $G_j $ approximates the inverse of the Jacobian $J(x_j)$ at $x_j$ by the
update formula $G_{j+1} = G_j + (\Delta x_j - G_j \Delta f_j) v_j^T $ in which
$v_j$ is defined in different ways see \cite{FangSaad07} for details.
AA belongs to the related class of \emph{multi-secant methods}.  Indeed,
the approximation \nref{eq:AA1c} can be written as:
\begin{align} 
  x_{j+1} &=   x_j+\beta f_j -(\XX_j+\beta \FF_j)\ \theta\up{j}   \label{eq:AA}  \\
          &= x_j - [- \beta I + (\XX_j+\beta\FF_j)(\FF_j^T\FF_j)^{-1}\FF_j^T ] f_j  \label{eq:AA2}  \\
          &\equiv x_j - G_j  f_j  \quad \mbox{with} \quad G_j \equiv - \beta I + (\XX_j+\beta\FF_j)(\FF_j^T\FF_j)^{-1}\FF_j^T 
  . \label{eq:AndQN}
\end{align} 
Thus, $G_{j} $ can be seen as an update to the (approximate) inverse
Jacobian $ G_{[j-m]} = -\beta I$ by the formula: \eq{eq:msec} G_{j} = G_{[j-m]} +
(\XX_j - G_{[j-m]} \FF_j)(\FF_j^T\FF_j)^{-1}\FF_j^T.  \en It can be shown that the
approximate inverse Jacobian $G_j$ is the result of minimizing
$\| G_j + \beta I \|_F$ under the \emph{multi-secant condition} of type II
\footnote{Type I Broyden conditions involve approximations to the Jacobian,
  while type II conditions deal with the inverse Jacobian.}
\eq{eq:mscond} G_j
\FF_j = \XX_j.  \en This link between AA and Broyden multi-secant type updates
was first unraveled by Eyert~\cite{eyert:acceleration96} and expanded upon in
\cite{FangSaad07}. Thus, the method is in essence what we might call a `block
version' of Broyden's second update method, whereby a rank $m$, instead of rank
$1$, update is applied at each step.


    In addition, we also have a multi-secant version of the no-change
  condition \eqref{eq:NoCh}.
  This is just a block version of the no-change condition 
  Equation~\eqref{eq:NoCh} as represented by Equation (15) in  \cite{FangSaad07},
  which stipulates that 
  \eq{eq:NoCh2} 
  (G_j - G_{[j-m]}) q = 0 \quad \forall q \perp \Span \{ \FF_j \} 
  \en
  provided we define  $ G_{[j-m]}= 0$.


  The essence of AA is that it approximates $f(x_j - \Delta X_j \ \gamma) $ by
  $f(x_j) - \Delta F_j \ \gamma $. If $J_j$ is the Jacobian at $x_j$, we are
  approximating $J_j \ \Delta X_j $ by $ \Delta F_j $ and aim at making
  $f(x_j - \Delta X_j \ \gamma) $ small by selecting $\gamma$ so that its linear
  approximation $f(x_j) - \Delta F_j \ \gamma $ has a minimal norm.  Thus, just
  like quasi-Newton approaches, the method also aims at exploiting earlier
  iterates in order to approximate $J_j$, or its action on a vector or a set of
  vectors. The main approach relies of two sets of vectors - which we call
  $P_j = [p_{[j-m+1]}, p_{[j-m+1]+1}, \cdots, p_j] $ and
  $V_j = [v_{[j-m+1]}, v_{[j-m+1]+1}, \cdots, v_j] $ in this paper with the
  requirement that 
  \eq{eq:Jpk} J(x_j) p_j \approx v_j .
  \en These `secant' conditions establish a correspondence between the range of
  $P_j$ and the range of $V_j$ and are at the core of any multi-secant method.

  \subsubsection{Interpretations of AA}

  Anderson's original algorithm can be interpreted from a number of different
  perspectives. The author acknowledged being inspired the work on extrapolation methods
  similar to those discussed in Section~\ref{sec:extrapol}. However, the method he
  introduced does not fit the definition of an extrapolation technique according to 
  the terminology we use in this paper.

  Equation~\nref{eq:AA1c} resembles a simple Richardson iteration, see
  \nref{eq:Richardson0}, applied to the intermediate iterate $\bar x_j$ and its
  (linear) residual $\bar f_j$.  Therefore, the first question we could address
  is what do $\bar x_j $ and $\bar f_j$ represent?  AA starts by computing an
  intermediate approximation solution which is denoted by $\bar x_j$. The
  approximation $\bar x_j$ is just a member of the affine space
  $x_j + \Span \{ \XX_j \}$.  (In the equivalent initial presentation of AA, it
  was of the form given by Eq. \nref{eq:xbar}.)  We would normally write any
  vector on this space as $x_j + \XX_j \gamma$ where $\gamma$ is an arbitrary
  vector in $\RR^{m_j}$, where $m_j = \min \{j,m\}$ is the number of columns of
  $\XX_j$.
  Anderson changed the sign of the coefficient $\gamma$
  so instead he considered vectors of the form
  \eq{eq:xofgam} x(\gamma) = x_j -
  \XX_j \gamma .  \en
  Ideally, we would have liked to compute a coefficient
  vector $\gamma$ that minimizes the norm $\| f(x(\gamma))\|_2$. This is
  computable by a line-search but at a high cost, so instead, Anderson exploits
  the linear model around $x_j$: \eq{eq:LinModelAA} f(x(\gamma)) \equiv f(x_j -
  \XX_j \gamma ) \approx f(x_j) - \FF_j \gamma = f_j - \FF_j \gamma .  \en The
  above expression is therefore the \emph{linear residual} at $x_j$ for vectors
  of the form \nref{eq:xofgam}.  The optimal $\gamma$ which we denoted by
  $\gamma\up{j}$ is precisely what is computed in \nref{eq:gammak}.  Therefore,
  $\bar x_j$ is just \emph{the vector of the form \nref{eq:xofgam} that achieves
    the smallest linear residual.}

An important observation here is that we could have considered the new sequence
$\{ \bar x_j \}_{j=0, 1, \cdots}$ by itself. Remarkably, each vector $\bar x_j$
is just a linear combination of the previous $x_i$'s so it represents an
\emph{extrapolated sequence} of the form \nref{eq:extrapolGen}.
 In fact the vector $\bar x_j$ can be seen to be identical with the vector produced
 by the RRE algorithm seen in Section~\ref{sec:RRE}.
 Computing this extrapolated sequence by itself, \emph{without mixing it with
   the original iterates} will be identical with applying the RRE procedure to
 the $x_i$'s.  It gets us closer to the solution by combining previous iterates
 but we can do better.  Since $\bar x_j$ is likely to be a better approximation
 than $x_j$ a reasonable option would be to define the next iterate as a fixed-point
 iteration from it: \eq{eq:andersIdeal} x_{j+1}\up{+} = \bar x_j + \beta
 f(\bar x_j).  \en This, however, would require an additional function
 evaluation. Therefore, AA replaces $f(\bar x_j) = f(x_j - \XX_j \gamma\up{j})$
 by its linear approximation given in \nref{eq:LinModelAA} which is just
 $\bar f_j$, resulting in the Anderson update given by \nref{eq:AA1c}.  Thus,
 Anderson Acceleration can be viewed as a process that intermingles one step of
 RRE applied to previous iterates with one linearized gradient descent of the
 form: \nref{eq:AA1c}.  An alternative would be to restart - say every $k$ RRE
 steps - and reset the next iterate to be the linearized update \nref{eq:AA1c}.
 We tested a technique of this type in the experiment shown in Section~\ref{sec:numer1}.
 

    \subsubsection{Linear case: Links with Krylov subspace methods}
    \label{sec:LnkKry}
In this section we consider the case when the problem is linear
and set $f(x) = b - Ax$ (note the sign difference with earlier notation).
We assume that $A$ is nonsingular. 
In this situation we have \eq{eq:APj} \FF_j = -A \XX_j.
\en
The following lemma shows that 
that the matrix $U_j$ resulting from Algorithm \ref{alg:TGS} is
a  basis of the Krylov
subspace $\K_j(A,f_0)$ and that under mild conditions,
$Q_j, U_j$ satisfy the same relation as
$\FF_j, \XX_j$ in \eqref{eq:APj}  for AA-TGS($\infty)$.

\begin{lemma}
  Assume $A$ is invertible and $f(x)=b-Ax$. If Algorithm \ref{alg:TGS}
 applied to solve $f(x) = 0 $  with
  $m=\infty$ does not break down at step $j$, then the system
  $U_j$ forms a basis of the Krylov subspace $\mathcal{K}_j(A,f_0)$. In
  addition, the orthonormal system $Q_j$ built by Algorithm \ref{alg:TGS}
  satisfies $Q_j= -AU_j$.
\label{lemma:QAU}
\end{lemma}
\begin{proof}
  We first prove $Q_j=-AU_j$ by induction. When $j=1$, we have
  $q_1=(f_1-f_0)/s_{11}=-Au_1$. Assume $Q_{j-1}=-AU_{j-1}$. Then we have
    \begin{align*}
        s_{jj}q_j &= (f_j-f_{j-1})-\sum_{i=1}^{j-1}s_{ij}q_{i} =-A(x_j-x_{j-1})-\sum_{i=1}^{j-1}s_{ij}(-Au_{i})\\
        & = -A[(x_j-x_{j-1})-\sum_{i=1}^{j-1}s_{ij}u_{i}]\\ 
        & = s_{jj}(-Au_j).
    \end{align*}
    Thus, since $s_{jj} \ne 0$ we get $q_j=-Au_j$ and therefore $Q_j=-AU_j$, completing the
    induction proof.
    
    Next, we prove by induction  that $U_j$ forms a basis of $\mathcal{K}_{j}(A,f_0)$.
    It is more convenient to prove by induction the property that 
    for each $i\le j$, $U_i$ forms a basis of $\mathcal{K}_{i}(A,f_0)$.
    The result is true for $j=1$ since we have
    $u_1 = (x_1 - x_0)/s_{11} = \beta_0 f_0 /s_{11}$.  
    Now let us assume the property is true for $j-1$, i.e., that for each $i=1,2,\cdots,j-1$,
    $U_{i}$ is a basis of the Krylov subspace $\mathcal{K}_{i}(A,f_0)$. Then we have
\begin{align}
  s_{jj}u_j
  &= (x_j-x_{j-1})-\sum_{i=1}^{j-1}s_{ij}u_i \label{eq:sjjuj} \\                  
  &= -U_{j-1}{\theta}_{j-1} +\beta_{j-1}(f_{j-1}-Q_{j-1}{\theta}_{j-1})-\sum_{i=1}^{j-1}s_{ij}u_i \nonumber\\
  & = -U_{j-1}{\theta}_{j-1} + \beta_{j-1}f_{j-1} -\beta_{j-1}Q_{j-1}{\theta}_{j-1}- \sum_{i=1}^{j-1}s_{ij}u_i\nonumber \\
  & =\beta_{j-1}f_{j-1} -U_{j-1}{\theta}_{j-1}+\beta_{j-1}AU_{j-1}{\theta}_{j-1}-\sum_{i=1}^{j-1}s_{ij}u_i. \nonumber
\end{align}
  The induction hypothesis  shows that 
  $-U_{j-1}{\theta}_{j-1}+\beta_{j-1}AU_{j-1}\theta_{j-1}-\sum_{i=1}^{j-1}s_{ij}u_i\in\mathcal{K}_{j}(A,f_0)$.
  It remains to show that  $f_{j-1} = b-Ax_{j-1}\in \mathcal{K}_{j}(A,f_0)$. For this,  expand 
  $b-Ax_{j-1}$ as
  \begin{align*}
    b-Ax_{j-1} &= b- Ax_{j-1}+Ax_{j-2}-Ax_{j-2}+\ldots-Ax_{1}+Ax_{0}-Ax_{0} \\
               &= \sum_{i=1}^{j-1}-A(x_{i}-x_{i-1})+f_0 .
  \end{align*}
    From the relation \eqref{eq:sjjuj} applied with $j$ replaced by $i$, we see that $x_i - x_{i-1}$ is a linear
    combination of $u_1, u_2, \cdots, u_i$, i.e., it is a member $\K_i$ by the induction hypothesis.
    Therefore $ -A (x_i - x_{i-1}) \in \K_{i+1}$ - but since $i \le j-1$ this
    vector belongs to $ \K_{j}$.  The remaining term $f_0$ is clearly in $\K_j$. Because $U_{j}=-A^{-1}Q_j$ has full column rank and $u_i\in\mathcal{K}_{j}(A,f_0)$ for $i=1,\ldots,j$, $U_j$ forms a basis of $\mathcal{K}_{j}(A,f_0)$. This completes the induction proof.
  \end{proof}

From \eqref{eq:AA2N}, we see that in the linear case under consideration
the vector $\bar f_j$ is the residual for $\bar x_j$:
\begin{align}
  \bar f_j  = f_j- Q_j {\theta}_{j} &= (b-A x_j) - Q_j {\theta}_{j} \nonumber  \\
  & = (b-A x_j) + A U_j {\theta}_{j}  = b - A (x_j - U_j {\theta}_{j}) = b - A \bar x_j . \label{eq:fbarj}
\end{align}
The next theorem shows that $\bar x_j$ minimizes $\| b - A x\|_2$ over the affine space $x_0+\mathcal{K}_j{(A,f_0)}$.

\begin{theorem}
  \label{lem:opt1}
  The vector  $\bar x_{j}$ generated at the $j$-th step
  of AA-TGS($\infty$) minimizes the residual norm $\| b - A x\|_2$
  over all vectors $x$ in the affine space $x_0 + \mathcal{K}_j(A,f_0)$.
  It also minimizes the same residual norm over the subspace $x_k + \mathcal{K}_j(A,f_0)$ for any $k$ such that $0 \le k \le j$.
  \end{theorem}

  \begin{proof}
    Consider a  vector of the form $x = x_j - \delta $ where $\delta = U_j y$
    is an arbitrary member of $\mathcal{K}_j(A,f_0)$. We have
    \eq{eq:prf1}
      b - A x  = b - A (x_j - U_j y) = f_j + A U_j y = f_j - Q_j y .
      \en
      The minimal norm $\|b-Ax\|_2$ is reached when $y = Q_j^{\top} f_j $
      and the corresponding optimal $x$ is $\bar x_j$. Therefore, $\bar x_j$
      is the vector $x $ of the affine space $x_j + \mathcal{K}_j(A,f_0)$  with the smallest
      residual norm. We now write $x$ as:
\begin{align} 
x &= x_j - U_j y \nonumber   \\
  &= x_0 + (x_1-x_0) + (x_2-x_1) + 
    (x_3-x_2) + \cdots (x_{i+1} - x_i) + \ \nonumber \\
  & \qquad \cdots + (x_j - x_{j-1}) - U_j y \label{eq:lem2.2prf1}\\
  &= x_0 + \Delta x_0 + \Delta x_1 +  \cdots +\Delta x_{j-1} - U_j y.  \label{eq:lem2.2prf2}
\end{align}
We now exploit the relation obtained from the QR factorization of Algorithm \ref{alg:TGS}, namely $\XX_j = U_j S_j $ in \eqref{eq:QRrel}:
If $e$ is the vector of all ones, then
$\Delta x_0 + \Delta x_1 + \cdots + \Delta x_{j-1}  = \XX_j e = U_j S_j e$.
Define $t_j \equiv S_j e$. Then, from \eqref{eq:lem2.2prf2} we obtain 
\eq{eq:xepr00}
x = x_j - \delta  = x_0 - U_j [y - t_j] .
\en 
Hence, the set of all vectors of the form $ x_j - \delta = x_j - U_j y$
is the same as
the set of all vectors of the form $x_0 - \delta'$ where $\delta' \ \in \ \mathcal{K}_j(A,f_0)$.
As a result, $\bar x_j$ also minimizes $b - A x$ over all vectors in the affine
space
$x_0 +  \ \mathcal{K}_j(A,f_0)$.

The proof can be easily repeated if we replace $x_0$ by $x_k$ for 
any $k$ between $0$ and $j$.
The expansion (\ref{eq:lem2.2prf1} --\ref{eq:lem2.2prf2}) 
becomes 
\begin{align} 
  x_k - U_j y &= x_k + (x_{k+1}-x_k) + (x_{k+2}-x_{k+1}) + \nonumber \\
  & \qquad \qquad \cdots (x_{i+1} - x_i) +\cdots (x_j - x_{j-1}) - U_j y \label{eq:lem2.2prf3}\\
  &= x_k + \Delta x_k + \Delta x_{k+1} + \cdots + \Delta x_{j-1} - U_j y. 
\label{eq:lem2.2prf4}
\end{align} 
The rest of the proof is similar and straightforward.
\end{proof}

Theorem \ref{lem:opt1} shows that $\bar x_j$ is the $j$-th iterate of the GMRES algorithm for solving
$Ax=b$ with the initial guess $x_0$ and that $\bar f_j $ is the corresponding residual. The value of $\bar x_j$ is independent of the choice of $\beta_i$ for $i\leq j$. Now consider the residual $f_{j+1}$ of AA-TGS($\infty$) at step $j+1$. From the relations
$x_{j+1} = \bar x_j + \beta_j \bar f_j$ and \eqref{eq:fbarj} we get:
\eq{eq:fjp1}
f_{j+1} = b - A[\bar x_j + \beta_j \bar f_j]  =b -  A \bar x_j - \beta_j A \bar f_j
 = \bar f_j - \beta_{j} A \bar f_j = (I-\beta_j A) \bar f_j.
\en
This implies that the vector $f_{j+1}$ is the residual for $x_{j+1}$ obtained from $x_{j+1} = \bar x_j + \beta_j \bar f_j$ - which is a simple Richardson iteration
starting from the iterate $\bar x_j$. Therefore,  $x_{j+1}$ 
in Line~15 of Algorithm \ref{alg:TGS}  
is nothing but a Richardson iteration step
from this GMRES iterate. This is stated in the following proposition.

\begin{proposition}
  \label{prop:opt1}
  The residual $f_{j+1}$  of the iterate   $x_{j+1}$ generated at the $j$-th step
  of AA-TGS($\infty$) is equal to  $(I -\beta_j A) \bar f_j$ where $\bar f_j = b - A \bar x_j$ 
  minimizes the residual norm $\| b - A x\|_2$  over all vectors $x$ in the affine space $x_0 + \mathcal{K}_j(A,f_0)$.
  In other words, the $(j+1)$-st iterate of AA-TGS($\infty$) can be obtained by performing one step of a Richardson iteration applied to the
  $j$-th GMRES iterate.
\end{proposition}
A similar result has also been proved for the standard AA by  \cite{WalkerNi2011}
under slightly
different assumptions, see Section~\ref{sec:WalkerNi}.

Convergence in the linear case can be therefore analyzed by relating the residual of full AA-TGS with that of
GMRES. The following corollary of the above proposition shows a simple but useful inequality.
\begin{corollary}\label{cor:1}
  If AA-TGS($\infty$) is used to solve the  system \nref{eq:Ax=b}, then the residual norm of the
  iterate $x_{j+1}\up{AA-TGS} $ satisfies the inequality:
\eq{eq:aatgs-mr}
\Vert b -  A x\up{AA-TGS}_{j+1} \Vert_2   \le \Vert (I-\beta A)\|_2  \ \| b - A x_{j}\up{GMRES}\|_2,
\en
where $ x_{j}\up{GMRES}\|_2$ is the iterate obtained by $j$ steps of full GMRES starting with the same initial guess
$x_0$.
\end{corollary}
\begin{proof}
  According to Proposition \ref{prop:opt1}: $r_{j+1} = - f(x_{j+1}) = - (I-\beta A) \bar f_j = (I-\beta A)  r_j$
  where $r_j$ is the residual obtained from $j$ steps of GMRES starting with the same initial guess $x_0$.
  Therefore:
\begin{align}
  \Vert b -  A x\up{AA-TGS}_{j+1} \Vert_2 & = \Vert (I-\beta A)(b  - A x\up{GMRES}_j) \Vert_2
                                            = \Vert (I-\beta A) r_{j}\up{GMRES}\Vert_2 \nonumber \\
  & \le \Vert (I-\beta A)\|_2  \ \| r_{j}\up{GMRES}\|_2 .  \label{eq:tgpf1}  
\end{align}
\end{proof}
Thus, essentially all the convergence analysis of GMRES can be adapted to AA-TGS
when it is applied to linear systems.  The next section examines the special
case of linear symmetric systems.

\subsubsection{The linear symmetric case}\label{sec:st_aatgs}
A simple experiment will reveal a remarkable observation for the linear case
when the matrix $A$ is symmetric.  Indeed, the orthogonalization process (Lines
6-10 of Algorithm \ref{alg:TGS}) simplifies in this case in the sense that $S_j$
consists of only 3 non-zero diagonals in the upper triangular part when $A$ is
symmetric. In other words, when $A$ is symmetric  we only need to save
$q_{j-2}, q_{j-1}$ and $u_{j-2}, u_{j-1}$ in order to generate $q_j$ and $u_j$
in the full-depth case, i.e., when $m = \infty$. This is similar to the simplification obtained by
a Krylov method like FOM or GMRES when the matrix is symmetric.
We  first examine the components of the vector $Q^{\top}_jf_j$ in Line 14
of Algorithm \ref{alg:TGS}.

\begin{lemma}
  When $f(x)=b-Ax$ where $A$ is a real non-singular symmetric matrix
  then the entries of the vector ${\theta}_j=Q^T_jf_{j}$ in Algorithm \ref{alg:TGS} are all zeros except the last two.
    \label{lemma:twoentries}
\end{lemma}
\begin{proof}
 Let $i\leq j-1$. From \eqref{eq:fjp1}, we have 
    \[
    (f_{j},q_i) = (\bar{f}_{j-1}-\beta_{j-1}A\bar{f}_{j-1},q_i) = (\bar{f}_{j-1},q_i)-\beta_{j-1}(A\bar{f}_{j-1},q_i).
    \]
    The first term on the right-hand side vanishes because:
    \[  (\bar f_{j-1},q_i) = (f_{j-1}-Q_{j-1}{\theta}_{j-1},q_i)=((I-Q_{j-1}Q_{j-1}^T)f_{j-1},q_i)=0.
    \]
    For  the second term we write  
    $
    (A\bar{f}_{j-1},q_i)=(\bar{f}_{j-1},Aq_i)
    $
and observe that since $u_i\in\mathcal{K}_{i}(A,f_0)$,
then $q_i=-Au_i$ belongs to the Krylov subspace
$\mathcal{K}_{i+1}(A,f_0)$ which is the same as $\text{Span}  \{ U_{i+1} \}$
according to Lemma \ref{lemma:QAU}. Thus,
it can be written as $ q_i = - Au_i = U_{i+1} y$ for some $y$ and
hence, $A q_i = A U_{i+1} y = - Q_{i+1} y$, i.e., $Aq_i$ is in the span
of $q_1, \cdots, q_{i+1} $.
    Therefore, recalling that $\bar{f}_{j-1} \perp \text{Span}  \{Q_{j-1}\}$, we have: 
    \eq{eq:2entprf}
    (\bar{f}_{j-1},Aq_i)=0 \quad \text{for} \quad i\leq j-2.
    \en
    In the end, we obtain $(f_{j},q_i)=0$ for $i\leq j-2$.
\end{proof}
Lemma \ref{lemma:twoentries} indicates that the computation of $x_{j+1}$ in Line
15 of Algorithm \ref{alg:TGS} only depends on the two most recent
$q_{i}$'s and $u_{i}$'s.
In addition, as is shown next, the vectors 
$q_j$ and $u_j$ in Line 12 can be
computed from  $q_{j-2},q_{j-1}$ and $u_{j-2},u_{j-1}$  instead of all
previous $q_i$'s and $u_i$'s.
\begin{theorem}\label{thm:thetaj}
  When $f(x)=b-Ax$ where $A$ is a real non-singular symmetric matrix, then the upper triangular matrix $S_j$
  produced in Algorithm~\ref{alg:TGS} is banded with bandwidth $3$, i.e., we have $s_{ik}=0$ for $i<k-2$. 
\end{theorem}
\begin{proof}
  It is notationally more convenient to consider column $k+1$ of $S_j$ where $k+1 \le j$.
  Let $\Delta f_k\equiv f_{k+1}-f_{k}$, and $\Delta x_k=x_{k+1}-x_{k}$.  The inner product
  $s_{i,k+1}$ in Algorithm \ref{alg:TGS} is the same as
  $s_{i,k+1}   =(\Delta f_{k},q_i)$ that would be obtained by a classical Gram-Schmidt algorithm.
  We note that for $i\leq k$  we have $s_{i,k+1} = -(A\Delta x_{k},q_i)$.  
 Exploiting the relation: 
    \[
        \Delta x_{k} = (\bar{x}_k+\beta_k\bar{f}_k)-x_k = x_k-U_k{\theta}_{k} + \beta_k\bar{f}_k-x_k = -U_k{\theta}_{k} + \beta_k\bar{f}_k,
    \]
    we write 
    \begin{align*}
        A\Delta x_{k} &= -AU_k{\theta}_{k}+\beta_kA\bar{f}_k = Q_k{\theta}_{k}+\beta_kA\bar{f}_k\\
        &= -(f_k-Q_k{\theta}_{k})+f_k+\beta_kA\bar{f}_k\\
        &= -\bar{f}_k + f_k+\beta_kA\bar{f}_k,
    \end{align*}
and hence, 
\eq{eq:ADxk}
    (A\Delta x_{k},q_i) = -(\bar{f}_k,q_i)+(f_k,q_i)+\beta_k(A\bar{f}_k,q_i).
\en 
The first term on the right-hand side, $(\bar{f}_k, q_i)$, vanishes since
$i\leq k$. According to Lemma \ref{lemma:twoentries} the inner product
$(f_k, q_i)$ is zero for $i\leq k-2$.
In the proof of Lemma~\ref{lemma:twoentries} we showed that 
$ (\bar{f}_{j-1},Aq_i)=0$ for $i\leq j-2$, see \nref{eq:2entprf},
which means that $ (\bar{f}_{k},Aq_i)=0$ for $i\leq k-1$.
In the end  $s_{i,k+1} = -(A\Delta x_{k},q_i) =0$ for $i<k-1$
which is equivalent to the desired result. 
\end{proof}

Lemma \ref{lemma:twoentries} and Theorem \ref{thm:thetaj} show that when
AA-TGS($\infty$) is applied to solving linear symmetric problems, only the two
most recent $q_{i}$'s and $u_{i}$'s, i.e.,
$q_{j-2}, q_{j-1}$ and $u_{j-2}, u_{j-1}$ are needed to compute the
next iterate $x_{j+1}$.  In other words, the for loop in Line~6 of the algorithm
needs only to be executed for $i=j-2, $ and $i=j-1$ which means that AA-TGS(3)
is equivalent to AA-TGS($\infty$) in the linear symmetric case. Practically,
this leads to a significant reduction in  memory and computational cost. 

We saw earlier that in all cases, the full AA-TGS algorithm is equivalent to the
full-window Anderson, at least in exact arithmetic.  AA-TGS is just a different
implementation of AA in this case.  In the linear case,
Proposition~\ref{prop:opt1} states that the full AA-TGS is equivalent to (full)
GMRES followed by a Richardson step.  This led to Corollary~\ref{cor:1} which
enables us to establish convergence results by exploiting already known theory.
One specific such result is an analysis of the special case when the problem
is linear and   symmetric.

\begin{theorem}
\label{thm:convegencespd}
Assume that $A$ is symmetric positive definite and that a constant $\beta$ is used in
AA-TGS. Then  the iterate $x\up{AA-TGS}_{j+1}$ obtained  at the $(j+1)$-st step of AA-TGS($\infty$)
satisfies~:
\begin{align*}
       \Vert b - A x\up{AA-TGS}_{j+1}\Vert_2
  & \le \Vert I-\beta A \Vert_2 \frac{\Vert b - A x_0 \Vert_2}{T_{j}(\frac{\kappa +1 }{\kappa - 1})}\\
  & \le 2  \Vert I-\beta A \Vert_2 \Vert b - A x_0 \Vert_2
    \left(\frac{\sqrt{\kappa}-1}{\sqrt{\kappa}+1}\right)^{j}
\end{align*}
where $T_{j}$ is the Chebyshev polynomial of first kind of degree $j$,  
and $\kappa = \kappa(A)$ is the 2-norm condition number of $A$.
\end{theorem}

\begin{proof}
  We start from inequality \nref{eq:aatgs-mr} of Corollary~\ref{cor:1}.
  An analysis similar to that in \cite[sec. 6.11.3]{Saad-book2} for the CG method,
  will show that
\[
  \| r_{j}^{GMRES}\|_2 \le \frac{\| r_0 \|_2 }{T_j (1 + 2 \eta )}
\]
in which $\eta = \lambda_{min} /(\lambda_{max} - \lambda_{min}) = 1/(\kappa - 1)$.
Noting that $1+2 \eta = (\kappa + 1)/(\kappa - 1)$  establishes the first inequality. The second one follows from
using standard expressions of the Chebyshev polynomials based on the hyperbolic cosine 
\cite[p. 204-205]{Saad-book2}, which shows that
\[ T_j (1 + 2 \eta) \ge \half \left(\frac{\sqrt{\kappa(A)}+1}{\sqrt{\kappa(A)}-1}\right)^{j} . \]
This completes the proof. 
\end{proof}
    
\subsubsection{Other links between AA and Krylov methods in linear
  case} \label{sec:WalkerNi} The analysis shown above establishes strong
connections between full-depth AA-TGS and GMRES.  Since AA-TGS is equivalent to
standard AA in the full-window case, these results are also valid for AA. Such
connections were established well before the recent article \cite{AATGS}.

Specifically, in 2011, \cite{WalkerNi2011} studied the algorithm
and showed a form of equivalence between AA and GMRES in the linear
case. Another study along the same lines, discussed at the end of this section,
is the 2010 article by  \cite{Haelterman-al-2010} which is
concerned with a slightly different version of AA.

Because the Walker and Ni result is somewhat different from the one of
Proposition~\ref{prop:opt1} seen in Section~\ref{sec:LnkKry}, we will now
summarize it.  The paper makes the following set of assumptions.

\paragraph{Assumption \ (A):}
\begin{itemize}
\item AA is applied for the fixed-point mapping $ g(x) = Ax + b$
\item  Anderson acceleration is not truncated, i.e., $m = \infty$ 
\item  $(I-A)$ is nonsingular.
\item  GMRES is applied to solve $(I-A)x = b$ with the same initial guess $x_0$ as for AA.
\end{itemize}

The main result of the article is stated for the formulation of AA that follows
the notation of Pulay mixing seen in Section~\ref{sec:Pulay}. Accounting for
this change of notation, their result is stated below.

\begin{theorem} Suppose that Assumption~(A)  holds and that, for some $j > 0$,
  $r_{j-1}\up{GMRES} \ne  0 $ and also that
  $\|r_{i-1}\up{GMRES}\|_2 > \|r_{i}\up{GMRES}\|_2 $ for each $i$ such that $0 < i < j$.
  Then, $\bar x_j = x_j\up{GMRES}$ and $ x_{j+1} = g(x_j\up{GMRES})$.
\end{theorem}

The  article by \cite{DegrooteQN-ILS09} described a
method called QN-ILS (`Quasi-Newton Inverse Least-Squares') which resembles AA
although it is presented as a Quasi-Newton approach (hence its name).  The
authors seemed unaware of the Anderson article and the related literature but
the spirit of their method is quite close to that of AA. In fact, the authors
even use the same formulation for their method as that of the original AA in
that they invoke the basis \nref{eq:di} to formulate their method, instead of
the bases using forward differences.  Their algorithm is viewed from the angle
of a quasi-Newton method with updates of type-II, where the inverse Jacobian is
approximated. Using our notation, the only difference with AA is that in their
case the update \nref{eq:AA1c} for the new iterate becomes:
\eq{eq:QN-ils}
x_{j+1} = \bar x_j + f_j .
\en
Thus, relative to the update equation \nref{eq:AA1c} of AA, $\beta$ is set to one
and $ \bar f_j$  is replaced by $f_j$ in QN-ILS. In the linear case,
and full window case, the two
methods are mathematically equivalent since
$\bar f_j = f_j - \FF_j \gamma\up{j}$ and the two methods will produce the same
space of approximants in the projection process, at each step.
In the nonlinear case, the
two methods will not generate the same iterates in general.

From an implementation point of view, QN-ILS is more expensive than AA. As
described in \cite{DegrooteQN-ILS09} the algorithm recomputes a new QR
factorization each time, and does not exploit any form of downdating.  The main
reason for this is that, as already mentioned, QN-ILS relies on a basis of the
form \nref{eq:di}, which changes entirely at each new step $j$. This also means
that each of the vectors $d_i = f_j - f_i$ for $i=([j-m]), \cdots, j-1$ and the
related differences $x_j - x_i$, must be recomputed at every step $j$ leading to
a substantial added cost when compared to modern implementations of AA.
Indeed, the basis of
the $\Delta x_i$'s used in the modern version of AA, requires only that we
compute the most recent pair $\Delta f_{j-1}, \Delta x_{j-1}$ since the other
needed pairs were computed in earlier steps. In AA, one column is computed and
added to $\FF_j$ and one is dropped from it (when $j > m$). Similarly for
$\XX_j$.

The article~\cite{Haelterman-al-2010} studied the method in the linear case,
and established that it is equivalent to GMRES in this situation.
This result is similar to that of \cite{WalkerNi2011},
but we need to remember that AA and QN-ILS are different in the nonlinear case.

    
\subsubsection{Convergence properties of AA}
    \cite{TothKelley15} proved that AA 
    is locally $r$-linearly convergent under the condition that 
    the ﬁxed point map $g$ is a contraction mapping and the coefficients in the
    linear combination remain bounded. A number of other results were proved
    under different assumptions.

    \typeout{FIX me -- which norms are used? }
    
    The article    \cite{TothKelley15} starts by considering the linear case in which
    $g(x) = M x + b$ and shows that when $M$ is contracting with $\| M \| =c<1$ then
    the iterates of Anderson acceleration applied to $g$ will converge to the fixed-point
    $x^* = (I-M)\inv b$. In addition, the residuals converge $q$-linearly to zero,
    i.e., if $f(x)=g(x)-x$ then    $\| f(x_{k+1}) \| \le c \| f(x_k)\|$.
    This is used as a starting point for proving a result in the nonlinear case.

    Consider the situation where AA is applied to find the fixed-point
    of a function $g$ and let $f(x) = g(x)-x$. The authors invoke
    formulation (\ref{eq:xbarP}--\ref{eq:thetaP}) because their results require to make
    assumptions on 
    the coefficients $\theta_i$. With this in mind, their main result can be stated as follows.

    \begin{theorem}[\cite{TothKelley15}, Theorem 2.3]
      Assume that:
    \begin{enumerate}
    \item There is constant $\mu_{\theta}$ such that       $ \sum_{i=j- j_m}^j |\theta_i| \le \mu_{\theta} $
      for all $j \ge 0$.
    \item There is an $x^*$ such that $f(x^*) = g(x^*)-x^* = 0$.
    \item $g$ is Lipschitz continuously differentiable in the ball
      $\mathcal{B} (\hat \rho) = \{ x \ | \ \| x - x^* \| \le  \hat \rho \}$.
    \item There is a $c \ \in (0, 1)$ such that 
      $\| g(u) - g(v) \| \le c \|u-v\|$ \ for all $u,v$  in $\mathcal{B} (\hat \rho) $.
    \end{enumerate}
    Let  $c < \hat c < 1$.     Then  if $x_0$  is
sufficiently close to $x^*$, the Anderson iteration converges to $x^*$ r-linearly with r-factor
no greater than $\hat c$. Specifically: 
\eq{eq:TK1} f(x_k) \le  {\hat c}^k  f(x_0) , \en 
and
\eq{eq:TK2} \| x_k - u^* \|  \le  \frac{1+c}{1-c} \hat c^k  \| x_0 - u^* \| .
\en
\end{theorem}
Not that the result is valid for any norm not just for the case when the 2-norm
minimization is used in \nref{eq:theta}.  The first condition only states that
the coefficients $\theta_i$ resulting from the constrained least-squares problem
\nref{eq:thetaP} (or equivalently the unconstrained problem \nref{eq:theta}) all
remain bounded in magnitude. It cannot be proved that this condition will be
satisfied and the ill-conditioning of the least-squares problem may lead to
large values the the $\theta_i$'s.  However, the authors of the paper show how
to modify the standard AA scheme to enforce the boundedness of the coefficients
in practice.

In addition, \cite{TothKelley15} consider the particular case when the window size is $m=1$ and show that
in this situation the coefficients $\theta_i$ are bounded if $c$ is small enough that
$\hat c \equiv ((3c-c^2)/(1-c))<1$. If this condition is satisfied and if
$x_0 \ \in \ \mathcal{B}(\hat \rho)$ then they show that AA(1) with least-squares optimization
converges q-linearly with q-factor $\hat c$.

Even though these results are proved under somewhat restrictive assumptions they nevertheless
establish strong theoretical convergence properties. In particular, the results show that
under certain conditions, the AA-accelerated iterates will converge to the solution at least as
fast as the  original fixed-point sequence.

The theory in the Toth-Kelley article does not prove that the convergence of an
AA accelerated sequence will be faster than that of the original fixed-point
iteration. 
The article \cite{EvansPollock20}
addresses this issue by showing theoretically that Anderson acceleration (AA)
does improve the convergence rate of contractive fixed-point iterations in the
vicinity of the fixed-point.
Their experiments illustrate the improved linear convergence rates. However,
they also show that when the initial fixed-point iteration converges
quadratically, then its convergence is slowed by the AA scheme.

In another paper \cite{PollockRebholz21} discuss further theoretical aspects of
the AA algorithm  and show a number of strategies to improve
convergence. These include techniques for adapting the window-size $m$
dynamically, as well as filtering out columns of $\FF_j$ when linear dependence
is detected.  Along the same lines, building on work by \cite{Rohwedder-Phd},
\cite{Stefano22} present a \emph{stabilized version of AA} which examines the
linear independence of the latest $\Delta f_j$ from previous differences. The
main idea is to ensure that we keep a subset of the differences that are
\emph{sufficiently linearly independent} for the projection process needed to
solve the least-squares problem. Local convergence properties are proved under
some assumptions.

It has been observed that AA works fairly well in practice especially in the
situation when the underlying fixed-point iteration that is accelerated has
adequate convergence properties.  However, without any modifications, it is not
possible to guarantee that the method will converge. A few papers address this
`global convergence' issue.  \cite{JunziBoyd20} consider `safeguarding
strategies' to ensure global convergence of type-I AA methods. Their technique
assumes that the underlying fixed-point mapping $g$ is non-expansive and,
adopting a multisecant viewpoint, develop a Type-I based AA update whereby the
focus is to approximate the Jacobian instead of its inverse as is done in
AA. Their main scheme relies on two ingredients.  The first is to add a
regularization of the approximate Jacobian to deal with the potential
(near)-singularity of the approximate Jacobian. The second is to interleave the
AA scheme with a linear mixing scheme of the form \nref{eq:mixing2}.  This is
done in order to `safeguard the decrease in the residual norms.'

\section{Nonlinear Truncated GCR}\label{sec:nlKrylov}
 Krylov accelerators for linear systems, which were reviewed in
 Section~\ref{sec:Krylov}, can be adapted in a number of ways for nonlinear
 problems.  We already noted that AA can be viewed as a modified Krylov subspace
 method in the linear case.  We also showed strong links between Krylov methods
 and a few extrapolation techniques in Section~\ref{sec:extrapol}.  One way to
 uncover generalizations of Krylov methods for nonlinear equations, is to take a
 multisecant viewpoint. The process begins with a subspace spanned by a set of
 vectors $ \{ p_{[j-m+1]}, p_{[j-m+1]+1}, \cdots, p_j \}$ -- typically related to a
 Krylov subspace -- and finds an approximation to the Jacobian or its inverse, when
 it is restricted to this subspace.  This second step can take different forms
 but it is typically expressed as a multisecant requirement, whereby a set
 vectors $v_i$ are coupled to the vectors $p_i$'s such that \eq{eq:JPj} v_i
 \approx J(x_i) p_i , \en where $J(x_i)$ is the Jacobian of $f$ at
 $x_i$. Observe that a different Jacobian is involved for each index $i$.  There are
 a number of variations to this scheme.  For example, $J(x_i)$ can be replaced
 by a fixed Jacobian at some other point, e.g., $x_{[j-m+1]} $ or $x_0$ as in
 Inexact Newton methods.
 
 We will say that the two sets:
 \eq{eq:PjVjB}
 P_j = [p_{[j-m+1]}, p_{[j-m+1]+1}, \cdots, p_j ] , \qquad
 V_j = [v_{[j-m+1]}, v_{[j-m+1]+1}, \cdots, v_j ]
 \en
 are \emph{paired.}
 This setting was encountered in the linear case, see equations \nref{eq:PjVjA}, and in 
 Anderson Acceleration where
 $P_j$ was just the set $\XX_j$ and $V_j$ was $\FF_j$. Similarly,
 in AA-TGS these two sets were $U_j$ and $Q_j$ respectively.
It is possible to develop
  \emph{a broad class of multi-purpose accelerators
    with this general viewpoint.}
  One of these methods~\cite{nltgcr}, named the NonLinear Truncated
  Generalized Conjugate Residual (NLTGCR), is built as a nonlinear extension of
  the Generalized Conjugate Residual method seen in \ref{sec:MR}.
  It  is discussed next.

  \subsection{nlTGCR}\label{sec:nltgcr} 
  Recall that in the linear case, where we solve the system linear $Ax =b$, the main
  ingredient of GCR is to build two sets of paired vectors
  $\{ p_i \}$, $\{ Ap_i \}$ where the $p_i$'s
  are the search directions obtained in earlier steps and the $ Ap_i$'s are
  orthogonal to each other.  At the $j$-th step, we introduce a new pair
  $p_{j+1}, Ap_{j+1}$ to the set in which $p_{j+1}$ is initially set equal to the 
  most recent residual, See, Lines 7--12 of Algorithm~\ref{alg:gcr}. This vector is then
  $A^T A$-orthogonalized against the previous $p_i$'s.
  We saw that this process leads to a simple
  expression for the approximate solution, using a projection mechanism, see Lemma~\ref{lem:gcr}.

  The next question we address is how to extend GCR or its truncated version
  TGCR, to the nonlinear case.  The simplest approach is to exploit an inexact
  Newton viewpoint in which the GCR algorithm is invoked to approximately solve
  the linear systems that arise from Newton's method.  However, this is avoided
  for a number of reasons. First, unlike quasi-Newton techniques, inexact Newton
  methods build an approximate Jacobian for the current iterate and this
  approximation is used only for the current step. In other words it is
  discarded after it is used and another one is build in the next step. This is
  to be contrasted with quasi-Newton, or multisecant approaches where these
  approximations are built gradually.  Inexact Newton methods perform best when
  the Jacobian is explicitly available or can be inexpensively approximated.  In
  such cases, it is possible to solve the system in Newton's method as
  accurately as desired leading to superlinear convergence.

  An approach that is more appealing for fixed-point iterations is to exploit the
  multisecant viewpoint sketched above, by adapting it to GCR.  At a given step
  $j$ of TGCR, we would have available the  previous directions
  $p_{[j-m+1]}, \cdots, p_j$ along with their corresponding (paired) $v_i$'s, for
  $i=[j-m+1], \cdots, j$. In the linear case, each $v_i$ equals $A p_i$. In the
  nonlinear case, we would have, instead, $v_i \approx J(x_i) p_i$.   The next
  pair $p_{j+1}, v_{j+1}$ is obtained by the update
  \eq{eq:pjp1} p_{j+1} = r_{j+1} -
  \sum_{ {i=[j-m+1]}}^{j} \beta_{ij} p_i, \quad v_{j+1} = J(x_{j+1})
  r_{j+1} - \sum_{{i=[j-m+1]}}^{j} \beta_{ij} v_i
  \en
  where the $\beta_{ij}$'s are selected so as to make $v_{j+1} $ orthonormal to the 
  previous vectors   $v_{[j-m+1]},\cdots,v_j$.
  One big difference with the linear case, is that the residual vector $r_{j+1}$ is
  now the nonlinear residual, which is $r_{j+1} = -f (x_{j+1})$.

 \typeout{ ============================CHECK indices. }

  This process  builds   a pair $(p_{j+1}, v_{j+1})$ such that $v_{j+1}$ is
  orthonormal to the previous  vectors $v_{[j-m+1]},\cdots, v_j$.
  The current `search' directions
\emph{  $\{ p_i \} $ for $i=[j-m+1],\cdots,j$} are \emph{paired} with the vectors
\emph{ $v_i \approx J(x_i) p_i $, for  $i=[j-m+1], \cdots, j$},
see, \nref{eq:PjVjB}.


Another important difference with TGCR is that the way in which the solution is
updated in Line~6 of Algorithm~\ref{alg:gcr} is no longer valid. This is because
the second part of Lemma~\ref{lem:gcr} no longer holds in the nonlinear case.
Therefore, the update will be of  the form
$x_{j  } + P_j y_j $ where $y_j = V_j^T r_j $.
Putting these together leads to the nonlinear adaptation of GCR shown in
Algorithm~\ref{alg:nltgcr}.



 
\begin{algorithm}
  \centering
\caption{nlTGCR(m)}\label{alg:nltgcr}
  \begin{algorithmic}[1]
  \State \textbf{Input}: $f(x)$,   initial  $x_0$. \\
  Set $r_0 = - f(x_0)$.
  \State Compute  $v = J(x_0)  r_0$; \ \Comment{\emph{Use Frechet}} 
  \State $v_0 = v/ \| v \|_2 $, $p_0 = r_0/ \| v \|_2 $;
  \For{$j=0,1,2,\cdots,$ } 
\State  $y_j = V_j^T r_j $ 
\State $x_{j+1} = x_j + P_j y_j$   \Comment{\emph{Scalar $\alpha_j$ becomes vector $y_j$}} \State $r_{j+1} = -f(x_{j+1}) $   \Comment{\emph{Replaces linear update: $r_{j+1} = r_j - V_j y_j$}}
\State Set: { $p := r_{j+1}$}; and  compute {$v =  J(x_{j+1})  p$} \ \ \Comment{\emph{Use Frechet}} 
\State Compute $\beta_j = V_{j}^T v $
\State $v = v - V_{j}\beta_j, \qquad p = p - P_{j}\beta_j$ 
\State $p_{j+1} :=p /\| v\|_2$ ; \qquad  $v_{j+1} :=v/\|v\|_2$ ; 
\EndFor
\end{algorithmic}
\end{algorithm}



The relation with Newton's method can be understood from the local  linear model
which is at the foundation of the algorithm:
\eq{eq:LinMod}
f(x_j + P_j y) \approx f(x_j) + V_j y , 
\en
which follows  from the following approximation
where the $\gamma_i$'s are  the  components of $y$
and the sum is over $i=[j-m+1]$ to $j$:
\[
  f (x_j + P_j y)
  \approx f(x_j) + \sum \gamma_i J(x_j) p_i
  \approx f(x_j) + \sum \gamma_i J(x_i) p_i
  \approx   f(x_j) + V_j y .
\]
The method essentially minimizes the residual norm of the linear model \nref{eq:LinMod}
  at the current step. 
  Recall that Anderson exploited a  similar local relation represented by
  \nref{eq:AALinMod} and the intermediate solution $\bar x_k$ in \nref{eq:AA1} is a
  local minimizer of the linear model.
  We will often use the notation
  \eq{eq:VjPjb}
  V_j \approx [J] P_j
  \en
  to express the relation represented by Equation~\ref{eq:JPj}.


\subsection{Linear  updates} \label{sec:linUp}
\label{sec:LinUpd}
The reader may have noticed that Algorithm~\ref{alg:nltgcr} requires two
function evaluations per step, one in Line 8 where the residual is computed and
one in Line 9 when invoking the Frechet derivative to compute $J(x_{j+1}) p$
using the formula:
\eq{eq:Frechet}
  J(x)  p \approx \frac{f(x + \eps p) - f(x)}{\eps}  \ , 
  \en
  where $\eps$ is some carefully selected small scalar.
  It is possible to avoid calculating the nonlinear residual by
  simply replacing it with its linear approximation given by expression
  \nref{eq:LinMod} from which we get,
  \[
    r_{j+1} = - f(x_j + P_j y_j)\approx - f(x_j) - V_j y_j = r_j - V_j y_j .
  \]
  Therefore, the idea of this   ``linearized
  update version'' of nlTGCR is  to replace $r_{j+1} $ in Line~8 by
  its linear  approximation   $ r_j - V_j y_j$. 
\bigskip
\betab
\> 8a:  \> \>  $ r_{j+1} = r_j - V_j y_j$
\entab
\smallskip 
  
This is now a method that resembles an Inexact Newton approach.
It will be equivalent to it if we add one more modification to the scheme
namely that we omit  updating the  Jacobian in Line~9, when computing $v$.
In other words, the Jacobian 
$J(x_{j+1})$ invoked in Line 9 is  constant
and equal to $J(x_0)$ and Line 9 becomes: 
\bigskip
\betab
\> 9a. Set: { $p := r_{j+1}$}; and  compute {$v =  J(x_{0})  p$}
\entab
\smallskip 

In practice, this means that when computing the vector $v$ in Line~9 of
Algorithm~\ref{alg:nltgcr} with equation~\ref{eq:Frechet}, the vector $x$ is set
to $x_0$.  This works with restarts, i.e., when the number of steps reaches a
restart dimension, or when the linear residual has shown sufficient decrease,
$x_0$ is reset to be the latest iteration computed and a new subspace and
corresponding approximation are computed.

All this means is that with minor changes to Algorithm~\ref{alg:nltgcr} we can implement
a whole class of methods that have been thoroughly studied in the past; see e.g.,
\cite{Dembo-al,Brown-Saad,Brown-Saad2} and \cite{Eisenstat-Walker94} among others.  
Probably the most significant disadvantage of inexact Newton methods, or to be specific
Krylov-Newton methods, is that a large 
number of function evaluations may be needed to build the Krylov
subspace in order to obtain a single iterate, i.e., the  next (inexact) Newton  iterate. 
After this iterate is computed, all the information gathered at
this step, specifically  $P_k,$ and $ V_k$, is discarded. This is to be contrasted with
quasi-Newton techniques where the most recent function evaluation
contributes to building an updated approximate Jacobian.
Inexact Newton methods are most successful when the Jacobian is available, or
inexpensive to compute, and some  effective preconditioner can readily be computed.

Nevertheless, it may still be cost-effective to reduce the number of function
evaluations from two to one whenever possible. We can update the residual norm
by replacing Line~8 of Algorithm~\ref{alg:nltgcr} with the liner form (8a) when
it is deemed safe, i.e., typically after the iteration reaches a region where
the iterate is close enough to the exact solution that the linear
model~\nref{eq:LinMod} is accurate enough.  A simple strategy to employ linear
updates and  move back to using nonlinear residuals, was implemented and
tested in in~\cite{nltgcr}. It is based on probing periodically how far the
linear residual  $\tilde r_{j+1} = r_j - V_j y_j$ 
is from the actual one. Define:
\eq{eq:djadapt}
d_j = 1-\frac{(\tilde r_j, r_j)}{\|\tilde r_j\|_2 \| r_j\|_2} . 
\en
The \emph{adaptive} nlTGCR switches the linear mode on when $d_j < \tau $ and
returns to the nonlinear mode if $d_j \ge \tau $, where $\tau$ is a small
threshold parameter. In the experiments discussed in Section~\ref{sec:numer2} we
set $\tau = 0.01$.

\subsection{Nonlinear updates}\label{sec:nonlinUp}
We now consider an implementation of 
Algorithm~\ref{alg:nltgcr}  in which nonlinear residuals are computed at each step.
We can study the algorithm by establishing relations  with the linear residual
\eq{eq:LinRes} \tilde r_{j+1} = r_j - V_j y_j , \en
and the deviation between the actual residual $r_{j+1}$ and its
linear version $\tilde r_{j+1}$ at the $j+1$th iteration: 
\eq{eq:ResDev} z_{j+1} = \tilde r_{j+1} - r_{j+1} .
\en

To analyze the magnitude of $z_{j+1}$, we  define
\begin{align}
w_i & = (J(x_j) - J(x_i)) p_i \ \mbox{for} \ i=[j-m+1], \cdots, j ;  \quad \mbox{and}\quad 
W_j = [w_{[j-m+1]}, \cdots, w_j ] \label{eq:wi} \\
s_j  & = f(x_{j+1}) - f(x_j) -  J(x_j) (x_{j+1} - x_j)  \label{eq:2ndOrd} . 
\end{align} 
Observe  that
\eq{eq:Jxjpi}
  J(x_j) p_i = J(x_i) p_i + w_i = v_i + w_i .
  \en
Recall  from the Taylor series expansion that $s_j$ is a second order term relative
to $\| x_{j+1}-x_j \|_2$. 
Then it can be shown \cite{nltgcr} that
  the difference $\tilde r_{j+1} - r_{j+1} $ satisfies the relation:
\eq{eq:ResDiff}
\tilde r_{j+1} - r_{j+1}  =   W_j y_j  + s_j
=   W_j V_j^T r_j  + s_j ,
\en
and therefore that:
\eq{eq:ResDiffN}
\| \tilde r_{j+1} - r_{j+1} \|_2   \le \| W_j \|_2  \ \| r_j \|_2 + \| s_j \|_2 . 
\en
 
When the process nears convergence, 
$\| W_j \|_2 \| r_j \|_2 $ is the product of two first order terms
while $s_j$ is a second
order term according to its  definition \nref{eq:2ndOrd}.
Thus, $z_{j+1}$ is a quantity of the second order.  

The  following properties of Algorithm~\ref{alg:nltgcr} are easy to establish, see
\cite{nltgcr} for the proof and other details.
We denote by $m_j$ the number of columns in $V_j$ and $P_j$,
i.e., $m_j = \min \{m, j+1\}$
\begin{proposition}\label{prop:nltgcr} 
    The following properties are satisfied by the vectors produced by
    Algorithm~\ref{alg:nltgcr}:

  \begin{enumerate}

  \item 
    $ ( \tilde r_{j+1}, v_i) = 0  \quad \mbox{for} \quad [j-m+1] \le i \le j$,
    i.e., $V_{j}^T \tilde r_{j+1} = 0$;

  \item 
    $ \| \tilde r_{j+1} \|_2 = \min_y \| -f(x_j) + [J] P_j y\|_2
    = \min_y \| -f(x_j) + V_j y\|_2 $;

  \item $\langle v_{j+1} , \tilde r_{j+1}\rangle
    = \left \langle v_{j+1} \ , r_j \right \rangle $;

  \item $y_j=V_{j}^T r_{j} = \langle v_{j}, \tilde r_{j} \rangle e_{m_j} - V_{j}^T z_{j}  $
    where $ e_{m_j} = [0, 0, \cdots, 1]^T \in \ \RR^{m_j}$.

\end{enumerate}

\medskip 
\end{proposition}

Property (4) and equation \nref{eq:ResDiffN} suggest that  when $z_j$ is
small,
then $y_{j} $ will have small components everywhere except for the 
last component.
This happens  when the model is \emph{close to being linear} or when it is nearing
convergence,

 \subsection{Connections with  multisecant  methods} 
 The update at step $j$ of nlTGCR  can be written as follows:
\[x_{j+1} = x_j + P_j V_j^T r_j = x_j + P_j V_j^T (-f(x_j)) , \]
showing that   nlTGCR is  a multisecant-type method in which
the inverse Jacobian at step $j$, is approximated by
\eq{eq:Gj}
G_{j+1} \equiv P_j V_j^T .
\en
This approximation satisfies  the \emph{multisecant} equation
\eq{eq:msecant1} 
  G_{j+1} v_i = p_i \quad \mbox{for} \quad [j-m+1] \le i \le j.
  \en
Indeed, 
  $
  G_{j+1} v_i = P_j V_j^T v_i =  p_i =  J(x_i) \inv v_i \quad \mbox{for} \quad [j-m+1] \le i \le j.
$ 
In other words $G_{j+1}$ inverts $J(x_i)$ exactly when applied to $v_i$.

If  we substitute 
$p_{i}$ with  of $\Delta x_j$ and $v_i$ with  $\Delta f_j$ we see that 
equation~\nref{eq:msecant1}  is just  the constraint \nref{eq:secant2} we
encountered in  Broyden's second update method.
In addition, the update $G_{j+1}$ also satisfies the `no-change' condition:
\eq{eq:msecant2} 
  G_{j+1} q = 0 \quad \forall q \perp v_i  \quad \mbox{for} \quad [j-m+1] \le i \le j.
  \en
  The usual no-change condition for secant methods is of the form
  $(G_{j+1}-G_{[j-m+1]}) q = 0 $ for $q \perp \Delta f_i$ which in our case
  would become
  $(G_{j+1}-G_{[j-m+1]}) q = 0 $ for $q \perp \ v_i  \quad \mbox{for} \quad [j-m+1] \le i \le j$.
  This means that  we are in effect updating $G_{[j-m+1]} \equiv 0$.
  Interestingly, $G_{j+1}$ satisfies an optimality result that is  similar to
  that of other secant-type and multisecant-type methods.
  This is stated in the following proposition which is easy to prove,
  see \cite{nltgcr}.
  
  \begin{proposition}
  \label{prop:Boptim}
  The unique minimizer of the following optimization problem:
  \eq{eq:optB}
  \min \{  \| G \|_F , G \in \ \RR^{n \times n} \ \text{subject to: }\quad G V_j = P_j \}
  \en
  is the matrix $G_{j+1} = P_j V_j^T$.
\end{proposition}

The condition $ G_{j+1} V_j = P_j$ is the same as the multisecant condition
$ G_j \mathcal{F}_j ={\cal X}_j $ of Equation \nref{eq:mscond} discussed
when we characterized AA as a multisecant method.
In addition, recall that the multisecant version of the no-change condition, as represented by
Equation~\nref{eq:NoCh2} is satisfied. This is the same as the no-change condition seen above for
nlTGCR. 
    
Therefore  we see that the two methods are quite similar in that they are both
multisecant methods defined from a pairing of two sets of vectors, $V_j, P_j$ for nlTGCR on the one hand and $\FF_j$, $\XX_j$ for AA on the other.
From this viewpoint, the two methods
differ mainly in the way in which the sets
$\FF_j / V_j$ ,  and $\XX_j / P_j $ are defined.
Let us use the more general notation
$V_j, P_j$ for both pairs of subspaces. 

In both cases, a vector $v_j$ is related to the corresponding $p_j$ by the fact
that
\eq{eq:ApproxNLT}
v_j \approx J(x_j) p_j .
\en
Looking at Line~9 of Algorithm~\ref{alg:nltgcr} indicates that,
before the orthogonalization step in nlTGCR (Lines 10-12), this 
relation becomes an equality, or aims to be close to an equality, by
employing a Frechet differentiation. Thus, the process introduces a pair
$p_j, v_j$ to the current paired  subspaces where $p_j$ is accurately mapped to
$v_j$ by $J(x_j)$, see \nref{eq:ApproxNLT}.
In the case of AA, we have  $v_j = \Delta f_{j-1} = f_j - f_{j-1} $ and write
\eq{eq:ApproxAA}
f_j \approx f_{j-1} + J(x_{j-1}) (x_j - x_{j-1}) \to
\Delta f_{j-1} \approx J(x_{j-1}) \Delta  x_{j-1} ,
\en
which is an expression of the form  \nref{eq:ApproxNLT} for index $j-1$.

An advantage of nlTGCR is that relation \nref{eq:ApproxNLT} is a \emph{more
  accurate} representation of the Jacobian than relation \nref{eq:ApproxAA},
which can be a rough approximation when $x_j$ and $x_{j-1}$ are not close to
each other.  This advantage comes at the extra cost of an additional function
evaluation, but this can be mitigated by an adaptive scheme
as was seen at the end of Section~\ref{sec:LinUpd}.

\subsection{Numerical illustration: nlTGCR and Anderson}\label{sec:numer2}
We now return to the numerical examples seen in Section~\ref{sec:numer1} to test
nlTGCR along with Anderson acceleration. As mentioned earlier we can run nlTGCR
in different modes. We can adopt a `linear' mode which is nothing but an inexact
Newton approach where the Jacobian systems are solved with a truncated GCR
method.  It is also possible to run an `adaptive' algorithm, as described
earlier - where we switch between the linear and nonlinear residual modes with
the help of a simple criterion, see the end of Section~\ref{sec:linUp}.
Figure~\ref{fig:Exp2}
reproduces the curves of \texttt{RRE(5)} and \texttt{AA(5,10)}, and \texttt{AA-TGS(5,.)}
shown in Section
\ref{sec:numer1} and add results with \texttt{nlTGCR\_Ln(5,.)}, \texttt{nlTGCR\_Ad(5,.)}, the
linear and adaptive versions of nlTGCR respectively. As before the parameter 5
for these 2 runs represents the window size. What is shown is similar to
what we saw in Section~\ref{sec:numer1} and the parameters, such as residual tolerance
maximum number of iterations, etc., are identical. 
One difference is that, because the methods perform very similarly at the beginning,
we do not plot the initial part of the curves, i.e., we omit 
points for which the number of function evaluations is less than 100.

  \begin{figure}[hbt]
  \includegraphics[width=0.505\textwidth]{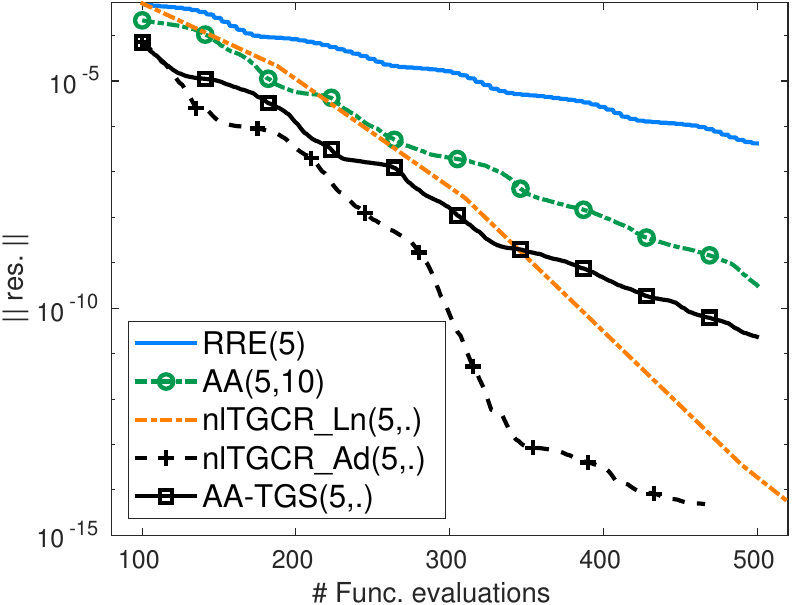}
  \includegraphics[width=0.495\textwidth]{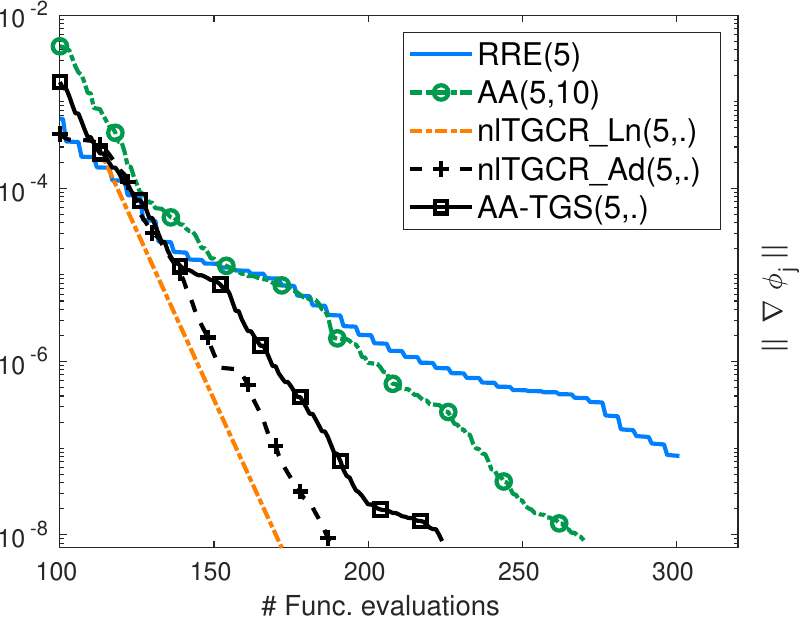}
  \caption{Comparison of 5 different methods
    for the Bratu problem (left) and the Lennard-Jones
    optimization problem (right).\label{fig:Exp2}}
  \end{figure}

  Because \texttt{nlTGCR\_Ln} is in effect an inexact Newton method, one can
  expect a superlinear convergence if a proper strategy is adopted when solving
  the Jacobian systems. These systems are solved inexactly by requiring that we
  reduce the (linear) residual norm by a certain tolerance $\tau$ where $\tau$
  is adapted.  Table~\ref{tab:tabL2} illustrates the superlinear convergence of
  the Linear, i.e., inexact Newton version of nlTGCR, as observed for the
  Lennard-Jones problem. The algorithm takes 10 outer (Newton) iterations to
  converge but we only show the last four iterations (as indicated by the column
  'its'). The second column shows the progress of the norm of the gradient,
  which is nearly quadratic as can be seen.  The 3rd column shows the number of
  inner steps needed to reduce the residual norm by $\eta$ at the given Newton
  step, where $\eta$ is shown in the 4th column.  This tolerance parameter
  $\eta$ is determined according to the Eisenstat-Walker update, see
  \cite{Kelley-book} for details. This update scheme works by trying to produce
  a quadratically decreasing residual based on gains made in the previous step.

  \begin{table}[hbt]
  \begin{center}
  \begin{tabular}{r|r|c|r}
    \multicolumn{1}{c}{its}&   \multicolumn{1}{c}{$\| \grad E \|_2 $}&
   \multicolumn{1}{c}{inner}&\multicolumn{1}{c}{$\eta$ }  \\ \hline 
       7  &   1.894e+00  &   10 & 4.547e-02 \\
       8  &   6.376e-02  &   22 & 1.337e-02 \\
       9  &   3.282e-04  &   47 & 1.020e-03 \\
      10  &   7.052e-09  &   59 & 2.386e-05 \\ 
    \end{tabular} 
\end{center}
\caption{The superlinear convergence of nlTGCR\_Ln for the Lennard-Jones example.}\label{tab:tabL2}
\end{table}


\section{Acceleration methods for machine learning}\label{sec:stoch}
We conclude this article with a few preliminary thoughts on how acceleration
methods might be put to work in a world that is increasingly driven by Machine
Learning (ML) and Artificial Intelligence (AI).  Training deep neural networks
models is accomplished with one of a number of known iterative procedures. What
sets these procedures apart from their counterparts in physical simulations is
their reliance on stochastic approaches that exploit large datasets. This
presents a completely new landscape for acceleration methods, one for which they
were not originally designed. It is too early to definitively say whether or not
acceleration methods will be broadly adopted for ML/AI optimization, but it is
certainly time to start investigating what modifications to traditional
acceleration approaches might be required to  deploy them successfully in this context.

Training an AI model is highly demanding, both in terms of memory and
computational power.  The idea of resorting to acceleration in deep learning is
a rather natural one when considering that standard approaches may require tens
of thousands of iterations to converge.  Anderson acceleration for deep learning
tasks was discussed in a number of recent articles, see, e.g.
\cite{pasini2021stable,PasiniAA4DNN,sun2021damped,Shi2019RegularizedAA}, among
others.  Most of these papers advocate some form of regularization to cope with
the varying/ stochastic nature of the optimization problem.

Before reviewing the challenges posed by stochastic techniques to acceleration
methods we briefly describe the Stochastic Gradient Descent (SGD)
algorithm, one of the simplest methods employed to train DNN models.  In spite
of its simplicity, SGD is a good representative of iterative optimization
algorithms in deep learning, because its use is rather widespread and 
because it shares the same features as those of the more advanced algorithms.

\subsection{Stochastic Gradient Descent}

The classical (deterministic) gradient descent (GD) method
for minimizing a convex function $\phi(w)$ with
respect to $w$, was mentioned in Section~\ref{sec:simpleMix}, see
Equation~\nref{eq:GD}. If  $\phi$
is differentiable the method consists of
taking the iterates: \eq{eq:gradMeth} w_{j+1} = w_j - \eta_j \nabla
\phi(w_j) ,  \en
where $\eta_j$ is a scalar termed the \emph{step-size}
or \emph{learning rate} in  machine learning. The above algorithm is well understood for
functions $\phi$ that are convex. In this situation, the step-size $\eta_j$ is 
usually determined by performing a 
line search, i.e., by selecting the scalar $\eta$ so as to  minimize, or reduce
in specific ways, the cost function  $\phi(w_j -\eta \nabla \phi(w_j))$ with respect to $\eta$.
It was \cite{Cauchy47} who invented the method for solving systems of equations;
see also \cite{PetrovaOnCauchy} for additional details on the origin of the
method.

In data-related applications $w$ is a vector of weights needed to optimize a
process. This often amounts to finding the best parameters to use, say in a
classification method, so that the value of $\phi $ at sample points matches
some given result across items of a dataset according to a certain measure.
Here, we
simplified notation by writing $\phi(w)$ instead of the more accurate
$\phi(x | w)$, which is to be read as ``the value of $\phi$ for $x$ given the
parameter set $w$''.  

In the specific context of Deep Learning, $\phi(w)$ is often the sum of a large
number of other cost functions, i.e., we often have
\eq{eq:fw} \phi(w) =
\sum_{i=1}^N \phi_i (w) \qquad \rightarrow \qquad \nabla \phi(w) = \sum_{i=1}^N
\nabla \phi_i (w) .
\en
The index $i$ here refers to the samples in the
training data.  In the simplest case where a `Mean Square Error' (MSE) cost
function is employed, we would have
\eq{eq:mse} \phi(w) = \sum_{i=1}^n \| \hat
y_i - \phi_w(x_i)\|_2^2
\en
where $\hat y_i$ is the target value for item $x_i$
and $ y_i = \phi_w(x_i) $ is the result of the model for $x_i$ given the weights
represented by $w$.
  
Stochastic Gradient Descent (SGD) methods are designed specifically for the
common case where $\phi(w)$ is of the form \nref{eq:fw}. It is usually expensive
to compute $\nabla \phi$ but inexpensive to compute a component
$\nabla \phi_i (w)$. As a result, the idea is to replace the gradient
$\nabla \phi(w) $ in the gradient descent method by $\nabla \phi_{i_j} (w_j)$
where $i_j$ is drawn at random at step $j$. The result is an iteration of the type:
\eq{eq:SGD} w_{j+1} = w_j -\eta_j \nabla \phi_{i_j} (w_j) \ .
\en
The step-size $\eta_j$, called the `learning rate' in this context,
is rarely selected  by a linesearch but it is determined adaptively
or  set to a constant.   Convergence results for SGD have been established 
in the convex case
\cite{Bubeck14,RobbinsMunro51,hardt16,gower2019sgd}.
The main point, going back to the seminal work by 
\cite{RobbinsMunro51}, is that
when gradients are inaccurately computed, then if they are selected from a process 
whose noise has a mean of zero then the process will converge in
probability to the root.

A straightforward SGD approach that uses a single function $\phi_{i_j}$ at a time
is seldom used in practice because this typically results in a convergence that
is too slow. A common middle ground solution between the one-subfunction SGD and
the full Gradient Descent algorithm is to resort to \emph{mini-batching}.  In
short the idea consists of replacing the single function $\phi_{i_j} $ by an
average of few such functions, again drawn at random from the full set.

  Thus, mini-batching begins with  partitioning
  the set $\{1,2,\cdots,N\}$ into $n_B$ `mini-batches'
  $\calB_j, j=1, \cdots, n_B$ where
  \[
    \bigcup_{j=1}^{n_B} \calB_j =  \{1, 2, \cdots, N\} \quad \mbox{with} \quad
    \calB_j \bigcap \calB_k  = \emptyset \quad \mbox{for} \quad j\ne k .      
  \]
  Here, each $\calB_j \subseteq \{1, 2, \cdots, n\}$ is a small set of indices.
  Then instead of considering a single function $\phi_i$ we will consider  
 \eq{eq:OneBatch}  \phi_{\calB_j} (w) \equiv \frac{1}{ | \calB_j | } 
 \sum_{k \in \calB_j} \phi_k(w) .
 \en
We will cycle through all mini-batches $\calB_j$ of functions at each time
performing a group-gradient step of the form: 
\eq{eq:OneEpoch}
w_{j+1} = w_j - \eta  \nabla \phi_{\calB_j} (w_j)  \quad j=1,2,\cdots, n_B .
\en
If each set $\calB_j $ is small enough computing the gradient will be manageable and
computationally efficient. 
One sweep through the whole set of functions
as in \nref{eq:OneEpoch} is termed an \emph{`epoch'}. The number of iterations of SGD and
other optimization learning in Deep Learning is often measured in terms of epochs.
It is common  to select the partition at random at each new epoch by reshuffling the set
$\{1, 2, \cdots, N\}$ and redrawing  the batches consecutively from  the resulting shuffled set.
Models can be expensive to train: the more complex models often  require
thousands or tens of thousands of epochs to converge. 

Mini-batch processing in the random fashion described above is advantageous from
a computational point of view since it typically leads to fewer sweeps through
each function to achieve convergence. It is also mandatory if we wish to avoid
reaching local minima and overfitting.  Stochastic approaches of the type just
described are at the heart of optimization techniques in deep learning.


\subsection{Acceleration methods for deep learning: The challenges}\label{sec:chall} 
Suppose we want to apply some form of acceleration to the sequence generated by
the batched gradient descent iteration \nref{eq:OneEpoch}.  There is clearly an
issue in that the \emph{function changes at each step} by the nature of the
stochastic approach. Indeed, by the definition \nref{eq:OneBatch}, it is as if
we are trying to find the minimum of a new function at each new step, namely the
function \nref{eq:OneBatch} which depends on the batch $\calB_j$.
We could use the full
gradient which amounts to using a full batch, i.e., the whole data set, at each step.
However, it is often  argued that in deep learning an
exact minimization of the objective function using the full data-set at once is
not only difficult but also counter productive.  Indeed, mini-batching serves
other purposes than just better scalability. For example, it
helps prevent `overfitting': Using all the data samples at once is similar
to interpolating a function in the presence of noise at  all the data  points. 
Randomization also helps the process escape from bad local minima.

This brings us to the second problem namely the lack of convexity of the
objective functions invoked in deep learning. This means that all methods that
feature a second order character, such as a Quasi-Newton approach, will have
both theoretical and practical difficulties.  As was seen earlier, Anderson
Acceleration can be viewed as a secant method similar to a Quasi-Newton
approach. These methods will potentially utilize many additional vectors but
result in no or little acceleration, if not in a breakdown caused by the non SPD
nature of the Hessian.  The lack of convexity and the fact that the
problem is heavily over-parameterized means that there are many solutions to which the
algorithms can converge.  Will acceleration lead to a better solution than that of the
baseline algorithm  being accelerated? 
If we consider only the objective
function as the sole criterion, one may think that the answer is clear: 
the lower the better. However, practitioners in this field are more interested in
`generalization' or the property to obtain good classification results on data that is
not among the training data set.
The problem of generalization has been the object of  numerous studies, see, e.g.,
\cite{zhang21_understanding,lossLandscape18,WeinanELandscape,WeinanE20} among many others.
The paper \cite{zhang21_understanding} shows  by means of
experiments that looking at DL from the angle of minimizing the loss function fails
to explain generalization properties. The authors show that they can achieve a 
perfect loss of zero in training models on well-known datasets (MNIST, CIFAR10) that
have been modified by randomly changing all labels. In other words 
one can obtain parameters whose  loss function is minimum but with the worst
possible generalization since the resulting classification would be akin to assigning a random
label to each item. 
A number of other papers explore this issue further
\cite{WeinanE20,InductiveBias14,lossLandscape18,WeinanELandscape} by attempting
to explain generalization with the help of the `loss landscape', the geometry of
the loss function in high dimensional space. What can be understood from these
works is that the problem is far more complex than just minimizing a function.
There are many minima and some are better than others. A local minimum that has a smaller
loss function will not necessarily lead to better inference accuracy and
the random character  of the learning algorithms plays a central role in
achieving a good generalization.
This suggests that we should study mechanisms that incorporate or encourage randomness.
An illustration will be provided in the next section.

The third  challenge is that acceleration methods tend to be memory intensive,
requiring to store possibly tens of additional vectors to be effective.  In deep learning
this is not an affordable option. For example, a model like Chat-GPT3 has 175B
parameters while the more recent Llama3 involves 450B parameters.  This is the
main reason why simple methods like SGD or Adam \cite{Adam-paper} are favored in
this context.

\subsection{Adapting Acceleration methods for ML}
Given the discussion of the previous section one may ask what benefits can be
obtained by incorporating second-order information in stochastic methods?
Indeed, the article \cite{BottouCurtisNocedal18} discusses the applicability of
second order methods for machine learning and, citing previous work, points out
that the convergence rate of a Quasi-Newton-type ``... stochastic iteration (..)
cannot be faster than sublinear''.  However, \cite{bottou2005line} state that if
the Hessian approximation converges to the exact Hessian at the limit then the
rate of convergence of SGD is independent of the conditioning of the Hessian. In
other words, second order information makes the method ``better equipped to cope
with ill-conditioning than SGD''.  Thus, careful successive rescaling based on
(approximate) second-order derivatives has been successfully exploited to
improve convergence of stochastic approaches.

The above discussion may lead the potential researcher to dismiss all
acceleration methods for deep learning tasks. There are a few simple variations
to acceleration schemes to cope with some of 
issues raised in Section~\ref{sec:chall}.  For example, if  our goal
is to accelerate the iteration \nref{eq:OneEpoch} with a constant learning rate
$\eta$, then we could introduce inner iterations  to ``iterate within
the same batch''.  What this means is that we force the acceleration method to
act on the same batch for a given number of inner iterations. For example, if we use
RRE (see Section~\ref{sec:RRE}) we could decide to restart
every $k$ iterations, each of which is with the
same batch $\calB_j$.  When the next batch is selected, we replace the
latest iterate with the result of the accelerated sequence.
We found with simple experiments that a method like RRE will work very  poorly 
without such a scheme.

On the other hand, if we are to embrace a more randomized viewpoint, we could
adopt a mixing mechanism whereby the subspaces used in the secant equation
evolve across different batches. In other words we no longer force the accelerator to
work only on the same batch. Thus, the columns of the $\XX_j, \FF_j$ in \nref{eq:dfdx} 
of Anderson acceleration are now allowed to be associated with different batches.
In contrast with RRE, our preliminary experiments show that for AA and AA-TGS this in fact works better 
than adding an inner loop.

Here is a very simple experiment carried out with the help of the  PyTorch library
\cite{PyTorch}. We train a small Multilayer Perceptron (MLP) model,
see, e.g.,  \cite{Murphy-book}, with two hidden layers to
recognize handwritten images from the dataset MNIST (60,000 samples of images of
handwritten digits for training, 10,000 samples for testing).  We tested 4
baseline standard methods available in PyTorch: SGD, Adam, RMSprop, and
Adagrad. Each of these is then accelerated with RRE (\texttt{RRE}),
Anderson Acceleration (\texttt{AAc}),
or Anderson-TGS (\texttt{TGS}).  Here we show the results with SGD only. AA and AA-TGS
use a window size of 3 and a restart dimension of 10.  Both implement the
batch-overlapping subspaces discussed above (no inner loop).  In contrast RRE
incorporates an inner loop of 5 steps (without which it performs rather
poorly). The restart dimension for RRE is also set to 5, to match
the number of inner steps. 
We train the model 5 times using a different randomly drawn subset of
2000 items (out of the 60,000 MNIST training set) each time.  With each of the 5
runs we draw a test sample of 200 to (out of the 10,000 MNIST test set) to test
the accuracy of the trained model.  The accuracy is then averaged across the 5
runs. For each run the accuracy  is measured as just 100 times
the ratio of the  number of correctly recognized digits over the total in the test
set (200). We show the loss function for each of the training algorithms
on the left side of Figure~\ref{fig:dnn_expl}.
Two versions of the baseline algorithm (SGD) are tested  both of which use
the same learning rate $\eta = 0.001$.  The first one labeled
\texttt{SGD\_bas} is just the standard SGD while
the second  \texttt{SGD\_inn} incorporates the same inner iteration scheme as RRE
(5 inner steps). This performs slightly better than the original scheme sometimes
markedly better with the other optimizers.   

\begin{figure}[h]
  \includegraphics[width=0.49\textwidth]{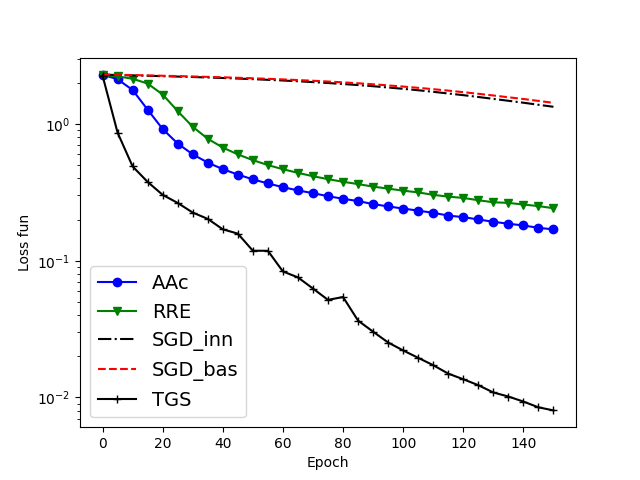}
  \hfill
  \includegraphics[width=0.49\textwidth]{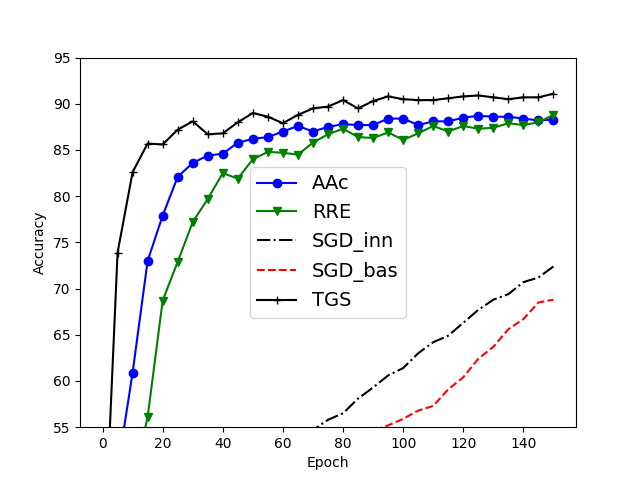}
  \caption{Comparison of various techniques for a simple MLP model on the set
    MNIST.}\label{fig:dnn_expl}
  \end{figure} 
  One can see a big improvement in the convergence of all the accelerated algorithms.
  Other tests indicated that adding an inner loop to avoid mixing  the subspaces
  from different batches is detrimental to both AA and AA-TGS, especially when
  comparing  the accuracy.
  While the improvement in accuracy is significant, we should point out that
  SGD has not yet fully converged as we limited the number of epochs to 150.
  In other tests, e.g., with \texttt{Adam}, we often saw a small
  improvement in accuracy but not as pronounced. However in all cases, the accelerated
  algorithms reach a higher precision much earlier than the baseline methods.

  The second remedy addresses the issue of memory. When training AI models, we
  are limited to a very small number of vectors that we can practically
  store. So methods like a restarted RRE, or Anderson, must utilize a very small
  window size. Note that a window size of $m$ will require a total of
  $\approx 2m $ vectors for most methods we have seen. This is the reason why a
  method like AA-TGS or nlTGCR can be beneficial in deep learning. For both
  methods, there is a simplification in the linear symmetric case as was seen
  earlier. This means that in this special situation a very small window size
  ($m=3$ for AA-TGS, $m=2$ for nlTGCR)
  \typeout{==================================================check window sizes}
  will essentially be equivalent to a window size of $m=\infty$, potentially
  leading to a big advantage.  In a general nonlinear optimization situation the
  Hessian is symmetric so when the iterates are near the optimum and a nearly
  linear regime sets in, then the process should benefit from the short-term
  recurrences of AA-TGS and nlTGCR.  In fact this may explain why AA-TGS does so
  much better at reducing the loss than AA with the same parameters in the
  previous experiment, see Figure~\ref{fig:dnn_expl}.  From this perspective,
  any iteration involving a short-term recurrence is worth exploring for Machine
  Learning.

What is clear is that acceleration methods of the type discussed in this paper
have the potential to emerge as powerful tools for training AI models.  While
adapting them to the stochastic framework poses challenges, it also offers a
promising opportunity for future research.

\bigskip
\subsection*{Acknowledgments}
This work would not have been possible without the invaluable collaborations I
have had with several colleagues and students, both past and ongoing. I would
like to extend special thanks  to the following individuals: Claude
Brezinski, Michela Redivo Zaglia, Yuanzhe Xi, Abdelkader Baggag, Ziyuan Tang,
and Tianshi Xu. I am also grateful for the support provided by the 
National Science Foundation (grant DMS-2208456).

\addcontentsline{toc}{section}{References}

\bibliography{refs}

\label{lastpage}

\end{document}